\documentclass{amsart}
\usepackage{mathtools,amssymb,amsthm,graphicx,mathrsfs,picinpar}
\usepackage{wrapfig}
\usepackage{tikz}
\usepackage{adjustbox}
\usepackage[all]{xy}
\usepackage[pdftitle={Spanning trees of graphs on surfaces and the intensity of loop-erased random walk on planar graphs},
            pdfauthor={Richard W. Kenyon and David B. Wilson},
            bookmarks=true,bookmarksopen=true,bookmarksopenlevel=2]{hyperref}
\usepackage{breakurl}
\usepackage[outline]{contour}
\contourlength{2pt}

\newcommand{\sref}[1]{\hyperref[#1]{\S~\ref*{#1}}}
\newcommand{\aref}[1]{\hyperref[#1]{Appendix~\ref*{#1}}}
\newcommand{\lref}[1]{\hyperref[#1]{Lemma~\ref*{#1}}}
\newcommand{\tref}[1]{\hyperref[#1]{Theorem~\ref*{#1}}}
\newcommand{\cref}[1]{\hyperref[#1]{Corollary~\ref*{#1}}}
\newcommand{\fref}[1]{\hyperref[#1]{Figure~\ref*{#1}}}
\newcommand{\pref}[1]{\hyperref[#1]{Proposition~\ref*{#1}}}

\newcommand{\old}[1]{}

\def\clap#1{\hbox to 0pt{\hss#1\hss}}
\def\mathllap{\mathpalette\mathllapinternal}
\def\mathrlap{\mathpalette\mathrlapinternal}
\def\mathclap{\mathpalette\mathclapinternal}
\def\mathllapinternal#1#2{\llap{$\mathsurround=0pt#1{#2}$}}
\def\mathrlapinternal#1#2{\rlap{$\mathsurround=0pt#1{#2}$}}
\def\mathclapinternal#1#2{\clap{$\mathsurround=0pt#1{#2}$}}

\renewcommand{\MRhref}[2]{\href{http://www.ams.org/mathscinet-getitem?mr=#1}{MR#1}}
\def\@rst #1 #2other{#1}
\renewcommand\MR[1]{\relax\ifhmode\unskip\spacefactor3000 \space\fi
  \MRhref{\expandafter\@rst #1 other}{#1}}

\makeatletter
\newif\if@borderstar
   \def\bordermatrix{\@ifnextchar*{%
       \@borderstartrue\@bordermatrix@i}{\@borderstarfalse\@bordermatrix@i*}%
   }
   \def\@bordermatrix@i*{\@ifnextchar[{\@bordermatrix@ii}{\@bordermatrix@ii[()]}}
   \def\@bordermatrix@ii[#1]#2{%
   \begingroup
     \m@th\@tempdima8.75\p@\setbox\z@\vbox{%
       \def\cr{\crcr\noalign{\kern 2\p@\global\let\cr\endline }}%
       \ialign {$##$\hfil\kern 2\p@\kern\@tempdima & \thinspace %
       \hfil $##$\hfil && \quad\hfil $##$\hfil\crcr\omit\strut %
       \hfil\crcr\noalign{\kern -\baselineskip}#2\crcr\omit %
       \strut\cr}}%
     \setbox\tw@\vbox{\unvcopy\z@\global\setbox\@ne\lastbox}%
     \setbox\tw@\hbox{\unhbox\@ne\unskip\global\setbox\@ne\lastbox}%
     \setbox\tw@\hbox{%
       $\kern\wd\@ne\kern -\@tempdima\left\@firstoftwo#1%
         \if@borderstar\kern2pt\else\kern -\wd\@ne\fi%
       \global\setbox\@ne\vbox{\box\@ne\if@borderstar\else\kern 2\p@\fi}%
       \vcenter{\if@borderstar\else\kern -\ht\@ne\fi%
         \unvbox\z@\kern-\if@borderstar2\fi\baselineskip}%
         \if@borderstar\kern-2\@tempdima\kern2\p@\else\,\fi\right\@secondoftwo#1 $%
     }\null \;\vbox{\kern\ht\@ne\box\tw@}%
   \endgroup
   }
\makeatother

\newtheorem{theorem}{Theorem}[section]

\newtheorem{lemma}[theorem]{Lemma}
\newtheorem{proposition}[theorem]{Proposition}
\newtheorem{corollary}[theorem]{Corollary}

\newtheorem{conjecture}[theorem]{Conjecture}

\theoremstyle{definition}

\theoremstyle{remark}

\numberwithin{equation}{section}

\newcommand{\Tr}{\operatorname{Tr}}
\newcommand{\Z}{\mathbb Z}
\newcommand{\R}{\mathbb R}
\newcommand{\C}{\mathbb C}
\newcommand{\Q}{\mathbb Q}

\newcommand{\N}{\mathbb N}
\newcommand{\No}{\mathcal N}
\newcommand{\E}{\mathbb E}

\newcommand{\G}{\mathcal G}
\newcommand{\U}{\mathrm{U}}
\newcommand{\SU}{\mathrm{SU}}
\newcommand{\SL}{{\mathrm{SL}_2(\C)}}
\newcommand{\GL}{{\mathrm{GL}_2(\C)}}

\newcommand{\Aut}{\operatorname{Aut}}
\newcommand{\Qdet}{\operatorname{qdet}}
\renewcommand{\L}{\mathscr{L}}
\newcommand{\Gv}{\mathscr{G}}
\newcommand{\Zv}{\mathscr{Z}}
\newcommand{\be}{\begin{equation}}
\newcommand{\ee}{\end{equation}}

\newcommand{\A}{\mathsf{A}}

\newcommand{\LL}{\mathbb{L}}

\newcommand{\rf}[1]{\begin{rotatebox}{90}{$#1$}\end{rotatebox}}
\newcommand{\srf}[1]{\smash{\begin{rotatebox}{90}{$#1$}\end{rotatebox}}}
\newcommand{\ffrac}[2]{\genfrac{}{}{0pt}{1}{#1}{#2}}

\begin{document}

\title[Spanning trees of graphs on surfaces and the intensity of LERW]{Spanning trees of graphs on surfaces and\\ the intensity of loop-erased random walk \\ on planar graphs}
\author{Richard W.\!~Kenyon}
\address{\href{http://www.math.brown.edu/~rkenyon/}{Richard W.\!~Kenyon}\\ Brown University\\ Providence, RI 02912, USA}
\thanks{The research of RWK was supported by the NSF}
\author{David B.~\!Wilson}
\address{\href{http://dbwilson.com}{David B.~\!Wilson}\\ Microsoft Research\\ Redmond, WA 98052, USA}
\old{
\begin{abstract}
  We show how to compute the probabilities of various connection
  topologies for uniformly random spanning trees on graphs embedded in
  surfaces.  As an application, we show how to compute the
  ``intensity'' of the loop-erased random walk in $\Z^2$, that is, the
  probability that the walk from $(0,0)$ to $\infty$ passes through a
  given vertex or edge.  For example, the probability that it passes
  through $(1,0)$ is $5/16$; this confirms a conjecture from 1994
  about the stationary sandpile density on $\Z^2$.  We do the
  analogous computation for the triangular lattice, honeycomb lattice
  and $\Z\times\R$, for which the probabilities are $5/18$, $13/36$,
  and $1/4-1/\pi^2$ respectively.
\end{abstract}}
\maketitle

\section{Introduction}

A spanning tree of a graph is a collection of edges which connects all
the vertices and has no cycles.
 Spanning trees were first
investigated by Kirchhoff in his study of electrical resistor networks
\cite{Kirchhoff}; in particular he showed that the determinant of the combinatorial
Laplacian counts spanning trees.

The uniform random spanning tree (UST) is a
well studied probability model, related to several other probability models.
For example, the loop-erased random walk of Lawler (see
\cite{MR1117680,MR1703133,MR2677157}) was shown by Pemantle
\cite{Pemantle} to have the same distribution as the paths connecting
vertices in the uniform spanning tree.  The abelian sandpile model of
self-organized criticality \nocite{BTK} was shown by Majumdar and Dhar
\cite{Majumdar-Dhar} to be closely related to spanning trees (recurrent states in the
sandpile model are in bijection with spanning trees). Lawler, Schramm, and Werner
\cite{LSW} showed that the branches of the spanning tree on $\Z^2$ converge in the scaling
limit to $\text{SLE}_2$ and the ``peano curve'' winding between the spanning tree and its
dual converges to $\text{SLE}_8$.

We show here how to compute the probabilities of various connection
topologies for uniform random spanning trees on graphs embedded in
surfaces.  As an application, we show how to compute the ``intensity''
of the loop-erased random walk in~$\Z^2$, that is, the probability
that the walk from $(0,0)$ to $\infty$ passes through a given vertex
or edge.  For example, the probability that it passes through $(1,0)$
is $5/16$; this confirms a conjecture from 1994 about the stationary
sandpile density on $\Z^2$.  We do the analogous computation for the
triangular lattice, honeycomb lattice and $\Z\times\R$, for which the
probabilities are $5/18$, $13/36$, and $1/4-1/\pi^2$ respectively.

Our techniques involve applying the vector bundle Laplacian
\cite{Kenyon.bundle} and asymptotics of the ``Green's function derivative'' for planar
graphs, together with a generalization of the grove counting techniques of \cite{KW1}
to graphs on annuli.

\subsection{Response matrices and groves}

Let $\G$ be a graph (undirected, with multiple edges and self-loops allowed),
and let $c\colon\G\to\R_{>0}$ be a positive conductance
on each edge.  Our graphs will be finite,
except in sections~\ref{Greens} and~\ref{LERW},
where we take limits to infinite lattices.
Let $\No$ be a subset of $\G$'s vertices such that every
vertex of $\G$ is connected to some vertex of $\No$.  The triple
$(\G,c,\No)$ is a \textbf{resistor network}.  Associated to this data
is the Dirichlet-to-Neumann matrix (also called the response matrix) $L$,
defined as follows.  Given a function $f:\No\to\R$, find its harmonic
extension $h$ on $\G$, that is a function on the vertices of $\G$ that
is harmonic on $\G\setminus\No$ and has values $f$ on $\No$.  Then
$L(f) = -\Delta(h)|_\No$ is a linear function of $f$, where $\Delta$
is the (positive semidefinite) graph Laplacian.  In electrical terms,
$L(f)$ gives the current flow into the nodes $\No$ when they are held
at $f$ volts.  While it is not obvious from this definition,
$L$ is a symmetric negative semidefinite matrix.

\textbf{Circular planar networks} are planar resistor networks
where $\No$ is a subset of the vertices on the outer face
listed in cyclic order.
These networks were studied in \cite{CdV,CGV,CIM}, where the set
of matrices which
occur as response matrices
were classified: they are precisely the matrices
whose ``non-interlaced'' minors are nonnegative.  (A non-interlaced
minor is one in which there are no 4 indices $a<b<c<d$ for which $a$
and $c$ are rows and $b$ and $d$ are columns or vice versa.)
Furthermore, these authors showed how to construct a circular planar
network having a given such response matrix~$L$.

In a resistor network a \textbf{grove} is a spanning forest (set of
edges with no cycles) in which every component contains at least one
vertex in $\No$.  (Our term grove is a generalization, to arbitrary
graphs and arbitrary connections, of the groves defined by Carroll and
Speyer in \cite{CS}.)
A spanning tree on a large graph (such as $\Z^2$) can be studied
by cutting the large graph into two subgraphs which are joined at nodes
along their boundaries.  The spanning tree restricted to either subgraph is a grove.
In \cite{KW1} we studied the natural
probability measure on groves (where each grove occurs with
probability proportional to the product of its edge weights), showing
for circular planar graphs how to compute the probability that a
random grove has a given connection topology in terms of the entries
in $L$.

\subsection{Graphs on surfaces}
We study here the same problem for a graph $\G$ embedded on a surface
$\Sigma$.  Here the usual notion of response matrix is not rich enough
to extract information about the underlying topological structure of a
grove.  Given a resistor network on a surface $\Sigma$, a natural
generalization of the response matrix is a matrix-valued function $\L$
on the representation variety $\text{Hom}(\pi_1(\Sigma),H)$ of
flat $H$-connections on $\G$; here $H=\C^*$ or $\SL$.
We show here how $\L$ can be used to
compute connection probabilities of (certain types of) groves on $\G$.

The question of characterizing which matrices $\L$ occur as a function
of the topology of~$\Sigma$ remains open.  See Lam and Pylyavskyy
\cite{LP:annulus} for related work in the case when the surface
$\Sigma$ is an annulus.

We give special attention to the case where the surface $\Sigma$ is an annulus;
this is the easiest case beyond the planar one (but already quite involved)
and also has applications to the study of spanning trees on planar graphs.
Connectivity questions within a spanning tree involving the path to $\infty$ can be
studied by viewing $\infty$ as one of the nodes of the surface graph, on a boundary
by itself.  When the other vertices are on the same face, the relevant surface is the
annulus.

\subsection{Applications to planar graphs}
Using these techniques one can in principle compute the probability
that the path of the uniform spanning tree from $a$ to $b$ in a planar
graph passes through a given set of edges or vertices (as in
\fref{ZZ-tree}), assuming the response matrix $\L$ can be evaluated.
When one of the nodes is~$\infty$, it is more convenient to work with
the Green's function $\Gv$.  For $\Z^2$ and the honeycomb and
triangular lattices, the usual Green's function $G$ is known, and we
modify the $H$-connection and use the translation and $180^\circ$ rotational symmetries of these graphs to extract
the additional information in $\Gv$ in closed form.
For $\Z^2$, our method shows that the probability that the path from the
origin to $\infty$ passes through a particular edge or vertex
(see \fref{ZZ-intensity}) is in $\Q(\frac1{\pi})$.  For the honeycomb and triangular lattices
these probabilities are in $\Q(\frac{\sqrt{3}}{\pi})$
(see \fref{hex-intensity} and \fref{tri-intensity}).

For example, we show that the probability that the
loop-erased random walk in $\Z^2$ from $(0,0)$ to $\infty$ contains the point
$(1,0)$ is $5/16$.  (See \fref{ZZ-tree}.)  This value was predicted by Levine and Peres
\cite{LP:54} and Poghosyan and Priezzhev \cite{PP:54}, by relating
this probability to the average density of the stationary abelian
sandpile model.
\begin{figure}[tbp]
  \begin{center}
    \includegraphics[height=2.5in]{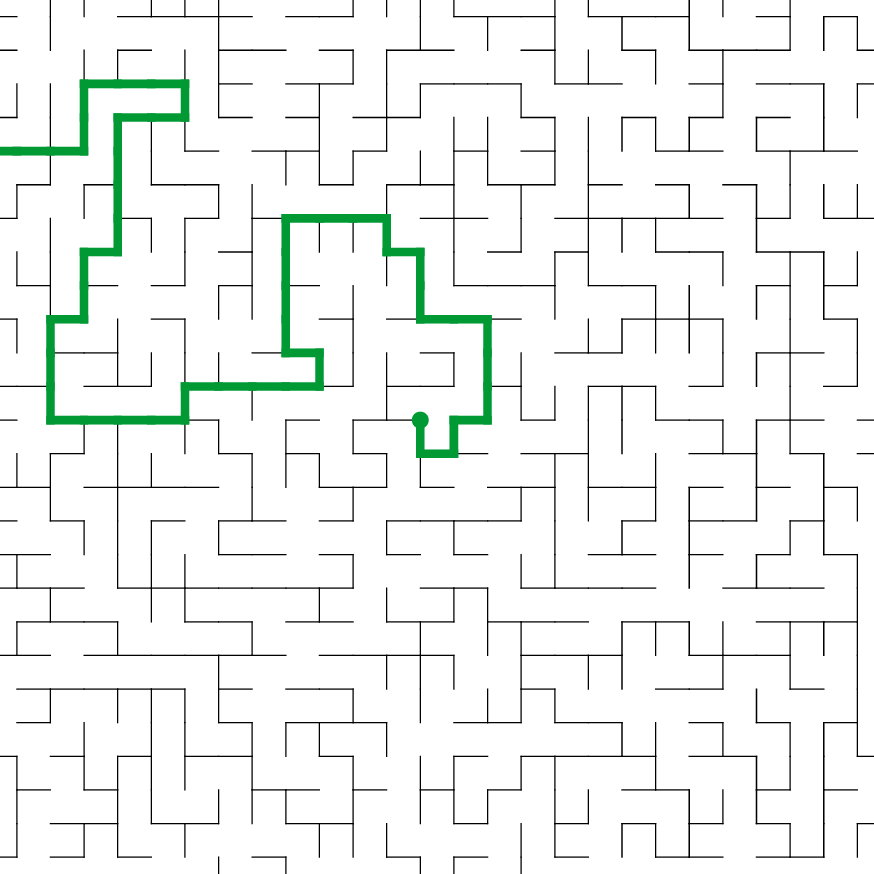}
  \end{center}
  \caption{A portion of the uniform spanning tree on $\Z^2$, with the path from
    $(0,0)$ to $\infty$ shown in bold.  The uniform spanning tree on
    $\Z^2$ can be constructed as a weak limit of uniform spanning
    trees on large boxes.
    The limiting measure
    exists, is unique, and is supported on trees of $\Z^2$
    \cite{Pemantle}.   Almost surely, within the uniform
    spanning tree of $\Z^2$, each vertex has a unique infinite path
    starting from it \cite{BLPS} (see also \cite{LP:book}).
    The path to infinity is a loop-erased random walk
    (LERW) \cite{Pemantle} (see also \cite{Wilson}).}
  \label{ZZ-tree}
\end{figure}
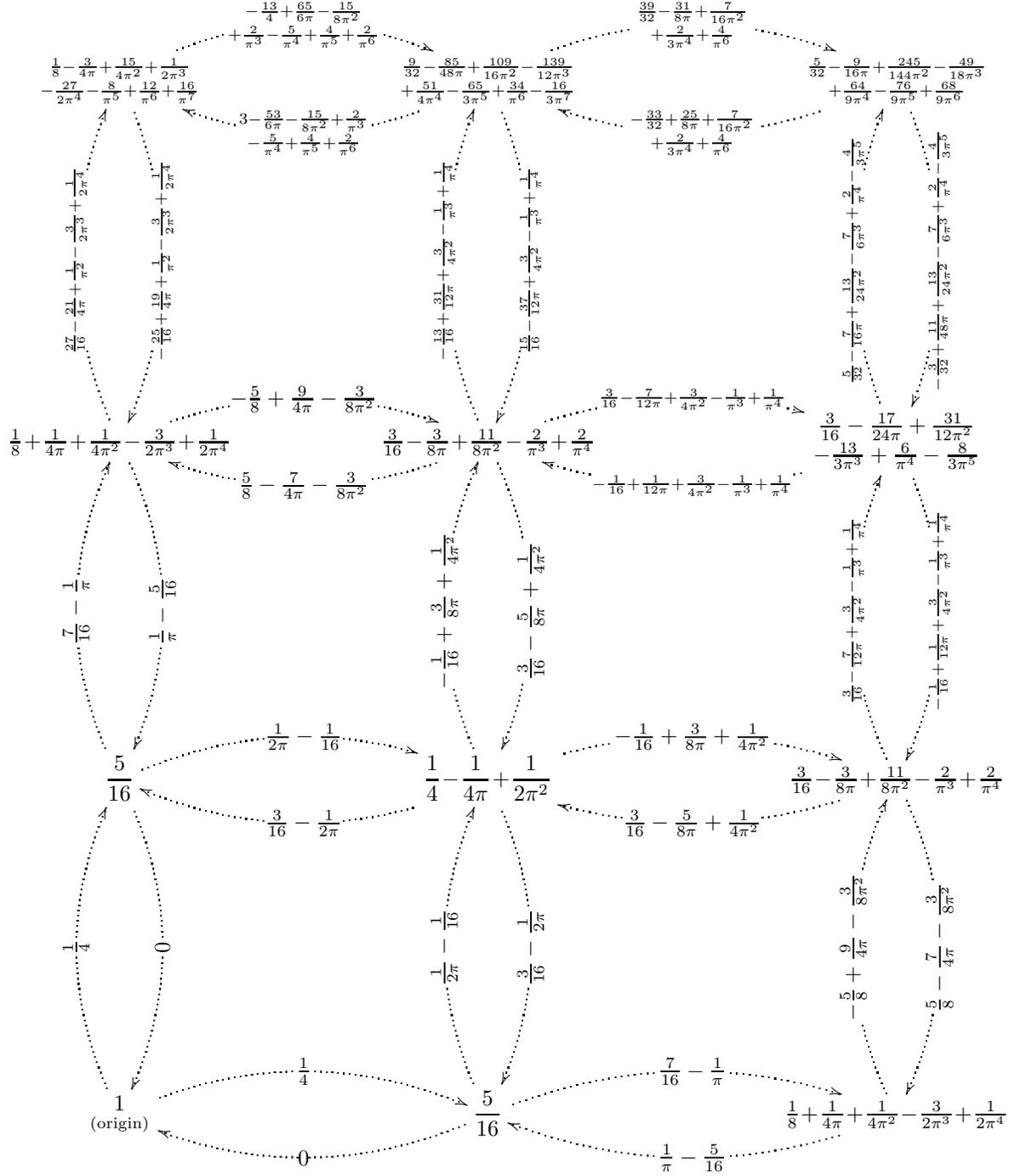
\begin{figure}[thbp]
\newcommand{\figvsep}{1.7in}
\newcommand{\fighsep}{0.2in}
\centerline{
\mbox{{\rule{-0.6in}{0.5in}}\smash{\rlap{
$
  \xymatrix{
\ffrac{\frac{1}{8}-\frac{3}{4 \pi}+\frac{15}{4 \pi ^2}+\frac{1}{2 \pi ^3}}{-\frac{27}{2 \pi ^4}-\frac{8}{\pi ^5}+\frac{12}{\pi ^6}+\frac{16}{\pi ^7}}
\ar@/^2.0pc/@{.>}[r]|{\ffrac{-\frac{13}{4}+\frac{65}{6 \pi}-\frac{15}{8 \pi ^2}}{+\frac{2}{\pi ^3}-\frac{5}{\pi ^4}+\frac{4}{\pi ^5}+\frac{2}{\pi ^6}}}
 \ar@/^1.6pc/@{.>}[d]|{\rf{\scriptstyle -\frac{25}{16}+\frac{19}{4 \pi}+\frac{1}{\pi ^2}-\frac{3}{2 \pi ^3}+\frac{1}{2 \pi ^4}}}
 &
\ffrac{\frac{9}{32}-\frac{85}{48 \pi}+\frac{109}{16 \pi ^2}-\frac{139}{12 \pi ^3}}{+\frac{51}{4 \pi ^4}-\frac{65}{3 \pi ^5}+\frac{34}{\pi ^6}-\frac{16}{3 \pi ^7}}
\ar@/^2.0pc/@{.>}[r]|{\ffrac{\frac{39}{32}-\frac{31}{8 \pi}+\frac{7}{16 \pi ^2}}{+\frac{2}{3 \pi ^4}+\frac{4}{\pi ^6}}}
 \ar@/^2.0pc/@{.>}[l]|{\ffrac{3 -\frac{53}{6 \pi}-\frac{15}{8 \pi ^2}+\frac{2}{\pi ^3}}{-\frac{5}{\pi ^4}+\frac{4}{\pi ^5}+\frac{2}{\pi ^6}}}
 \ar@/^1.6pc/@{.>}[d]|{\rf{\scriptstyle \frac{15}{16}-\frac{37}{12 \pi}+\frac{3}{4 \pi ^2}-\frac{1}{\pi ^3}+\frac{1}{\pi ^4}}}
 &
\ffrac{\frac{5}{32}-\frac{9}{16 \pi}+\frac{245}{144 \pi ^2}-\frac{49}{18 \pi ^3}}{+\frac{64}{9 \pi ^4}-\frac{76}{9 \pi ^5}+\frac{68}{9 \pi ^6}}
 \ar@/^2.0pc/@{.>}[l]|{\ffrac{-\frac{33}{32}+\frac{25}{8 \pi}+\frac{7}{16 \pi ^2}}{+\frac{2}{3 \pi ^4}+\frac{4}{\pi ^6}}}
 \ar@/^1.6pc/@{.>}[d]|{\rf{\scriptstyle -\frac{3}{32}+\frac{11}{48 \pi}+\frac{13}{24 \pi ^2}-\frac{7}{6 \pi ^3}+\frac{2}{\pi ^4}-\frac{4}{3 \pi ^5}}}
 &
\rule{0pt}{\figvsep}\\
\frac{1}{8}\!+\!\frac{1}{4 \pi}\!+\!\frac{1}{4 \pi ^2}\!-\!\frac{3}{2 \pi ^3}\!+\!\frac{1}{2 \pi ^4}
\ar@/^1.6pc/@{.>}[r]|{\textstyle -\frac{5}{8}+\frac{9}{4 \pi}-\frac{3}{8 \pi ^2}}
\ar@/^1.6pc/@{.>}[u]|{\rf{\scriptstyle \frac{27}{16}-\frac{21}{4 \pi}+\frac{1}{\pi ^2}-\frac{3}{2 \pi ^3}+\frac{1}{2 \pi ^4}}}
 \ar@/^1.6pc/@{.>}[d]|{\rf{\textstyle \frac{1}{\pi}-\frac{5}{16}}}
 &
\frac{3}{16}\!-\!\frac{3}{8 \pi}\!+\!\frac{11}{8 \pi ^2}\!-\!\frac{2}{\pi ^3}\!+\!\frac{2}{\pi ^4}
\ar@/^1.6pc/@{.>}[r]|{\frac{3}{16}-\frac{7}{12 \pi}+\frac{3}{4 \pi ^2}-\frac{1}{\pi ^3}+\frac{1}{\pi ^4}}
\ar@/^1.6pc/@{.>}[u]|{\rf{\scriptstyle -\frac{13}{16}+\frac{31}{12 \pi}+\frac{3}{4 \pi ^2}-\frac{1}{\pi ^3}+\frac{1}{\pi ^4}}}
 \ar@/^1.6pc/@{.>}[l]|{\textstyle \frac{5}{8}-\frac{7}{4 \pi}-\frac{3}{8 \pi ^2}}
 \ar@/^1.6pc/@{.>}[d]|{\rf{\textstyle \frac{3}{16}-\frac{5}{8 \pi}+\frac{1}{4 \pi ^2}}}
 &
\ffrac{\textstyle\frac{3}{16}-\frac{17}{24 \pi}+\frac{31}{12 \pi ^2}}{\textstyle-\frac{13}{3 \pi ^3}+\frac{6}{\pi ^4}-\frac{8}{3 \pi ^5}}
\ar@/^1.6pc/@{.>}[u]|{\rf{\scriptstyle \frac{5}{32}-\frac{7}{16 \pi}+\frac{13}{24 \pi ^2}-\frac{7}{6 \pi ^3}+\frac{2}{\pi ^4}-\frac{4}{3 \pi ^5}}}
 \ar@/^1.6pc/@{.>}[l]|{-\frac{1}{16}+\frac{1}{12 \pi}+\frac{3}{4 \pi ^2}-\frac{1}{\pi ^3}+\frac{1}{\pi ^4}}
 \ar@/^1.6pc/@{.>}[d]|{\rf{\scriptstyle -\frac{1}{16}+\frac{1}{12 \pi}+\frac{3}{4 \pi ^2}-\frac{1}{\pi ^3}+\frac{1}{\pi ^4}}}
 &
\rule{0pt}{\figvsep}\\
\displaystyle \frac{5}{16}
\ar@/^1.6pc/@{.>}[r]|{\textstyle \frac{1}{2 \pi}-\frac{1}{16}}
\ar@/^1.6pc/@{.>}[u]|{\rf{\textstyle \frac{7}{16}-\frac{1}{\pi}}}
 \ar@/^1.6pc/@{.>}[d]|{\rf{\textstyle 0}}
 &
\displaystyle \frac{1}{4}\!-\!\frac{1}{4 \pi}\!+\!\frac{1}{2 \pi ^2}
\ar@/^1.6pc/@{.>}[r]|{\textstyle -\frac{1}{16}+\frac{3}{8 \pi}+\frac{1}{4 \pi ^2}}
\ar@/^1.6pc/@{.>}[u]|{\rf{\textstyle -\frac{1}{16}+\frac{3}{8 \pi}+\frac{1}{4 \pi ^2}}}
 \ar@/^1.6pc/@{.>}[l]|{\textstyle \frac{3}{16}-\frac{1}{2 \pi}}
 \ar@/^1.6pc/@{.>}[d]|{\rf{\textstyle \frac{3}{16}-\frac{1}{2 \pi}}}
 &
\frac{3}{16}\!-\!\frac{3}{8 \pi}\!+\!\frac{11}{8 \pi ^2}\!-\!\frac{2}{\pi ^3}\!+\!\frac{2}{\pi ^4}
\ar@/^1.6pc/@{.>}[u]|{\rf{\scriptstyle \frac{3}{16}-\frac{7}{12 \pi}+\frac{3}{4 \pi ^2}-\frac{1}{\pi ^3}+\frac{1}{\pi ^4}}}
 \ar@/^1.6pc/@{.>}[l]|{\textstyle \frac{3}{16}-\frac{5}{8 \pi}+\frac{1}{4 \pi ^2}}
 \ar@/^1.6pc/@{.>}[d]|{\rf{\textstyle \frac{5}{8}-\frac{7}{4 \pi}-\frac{3}{8 \pi ^2}}}
 &
\rule{0pt}{1.4in}\\
\ffrac{\displaystyle 1}{\text{(origin)}}
\ar@/^1.6pc/@{.>}[r]|{\textstyle \frac{1}{4}}
\ar@/^1.6pc/@{.>}[u]|{\rf{\textstyle \frac{1}{4}}}
 &
\displaystyle \frac{5}{16}
\ar@/^1.6pc/@{.>}[r]|{\textstyle \frac{7}{16}-\frac{1}{\pi}}
\ar@/^1.6pc/@{.>}[u]|{\rf{\textstyle \frac{1}{2 \pi}-\frac{1}{16}}}
 \ar@/^1.6pc/@{.>}[l]|{\textstyle 0}
 &
\frac{1}{8}\!+\!\frac{1}{4 \pi}\!+\!\frac{1}{4 \pi ^2}\!-\!\frac{3}{2 \pi ^3}\!+\!\frac{1}{2 \pi ^4}
\ar@/^1.6pc/@{.>}[u]|{\rf{\textstyle -\frac{5}{8}+\frac{9}{4 \pi}-\frac{3}{8 \pi ^2}}}
 \ar@/^1.6pc/@{.>}[l]|{\textstyle \frac{1}{\pi}-\frac{5}{16}}
 &
\rule{0pt}{\figvsep}\\
\rule{1.6in}{0pt} & \rule{2.1in}{0pt}& \rule{2.1in}{0pt}& \rule{1.6in}{0pt}
}
$
}}
\phantom{\rule[-6.7in]{6.4in}{0pt}}
}}
\caption{Intensity of loop-erased random walk on $\Z^2$.  The origin
  is at the lower-left, and directed edge-intensities as well as
  vertex-intensities of the LERW are shown.  (See also \fref{square-intensity}.)}
  \label{ZZ-intensity}
\end{figure}

The connection between the spanning trees and the
abelian sandpile model was discovered by Majumdar and Dhar
\cite{Majumdar-Dhar}, and Priezzhev \cite{Priezzhev} used this
connection to compute the height distribution of the abelian sandpile
model, in terms of two integrals that could not be evaluated in closed
form.  Grassberger evaluated these integrals numerically, and
conjectured that the stationary density of the sandpile on $\Z^2$ is
$17/8$.  Later Jeng, Piroux, and Ruelle \cite{JPR} showed how to
express one of these two integrals in terms of the
other, and determined the sandpile height distribution in closed form,
under the assumption that the remaining integral, which numerically is
$0.5\pm 10^{-12}$, is exactly $1/2$.
Our derivation of this probability that LERW passes through $(1,0)$
confirms these conjectures (although our methods are different), and
shows that this aforementioned integral is exactly~$1/2$.

While we were writing up our results, Poghosyan, Priezzhev, and Ruelle
independently found another proof that the probability of visiting
$(1,0)$ is $5/16$ \cite{PPR:54}.  (They also asked about the probability
about visiting other points, and remarked that the
probability that the LERW visits $(1,1)$ is numerically close to
$2/9$.  This differs from the true value of $1/4-1/(4\pi)+1/(2\pi^2)$ by about $10^{-3}$.)

There are some interesting coincidences in the (undirected) edge
intensities of loop-erased random walk.  For each of the square,
triangular, and honeycomb lattices, there are several groups of edges
which are unrelated by any symmetry of the lattice for which the
undirected edge intensities are identical (see
Figures~\ref{square-intensity}, \ref{hex-intensity}, and
\ref{tri-intensity}).  We do not have an explanation for this
phenomenon.

\section{Bundles and connections}

Let $\G$ be a graph.  Given a fixed vector space $V$, a
\textbf{$V$-bundle}, or simply a \textbf{vector bundle} $B$ on~$\G$ is
the choice of a vector space $V_v$ isomorphic to $V$ for every
vertex~$v$ of~$\G$. We identify the vector bundle with the vector
space $V_{\G}=\oplus_{v}V_v \cong V^{|\G|}$.  A \textbf{section} of a
vector bundle $B$ is an element of $V_{\G}$.

If $H$ is a subgroup of $\Aut(V)$, an \textbf{$H$-connection} $\Phi$ is
the choice for each directed edge $e=(v,w)$ of~$\G$ of an isomorphism
$\phi_{v,w}\in H$ between the corresponding vector spaces
$\phi_{v,w}:V_v\to V_{w}$, with the property that
$\phi_{v,w}=\phi_{w,v}^{-1}$. This isomorphism is called the
\textbf{parallel transport} of vectors in $V_v$ to vectors in
$V_{w}$.  Given an oriented cycle $\gamma$ in $\G$ starting at~$v$,
the \textbf{monodromy} of the connection is the element of
$\text{Aut}(V_v)$ which is the product of these isomorphisms around
$\gamma$.  Monodromies starting at different vertices on~$\gamma$ are
conjugate.

Two connections $\Phi=\{\phi_e\}$ and $\Phi'=\{\phi'_e\}$ are
\textbf{gauge equivalent} if there are maps $\psi_v:V_v\to V_v$ such
that $\phi_{v,w}\circ\psi_v = \psi_{w}\circ\phi'_{v,w}$ for all
vertices $v$ and $w$ of $\G$.

It is useful to extend the notion of bundle and connection to the
edges as well: define for each edge $e$ of $\G$ a copy $V_e$ of $V$,
and define maps $\phi_{v,e}:V_v\to V_e$ whenever~$v$ is an endpoint
of~$e$, with the property that $\phi_{e,w}\circ\phi_{v,e}=\phi_{v,w}$
whenever edge $e$ joins vertices $v$ and $w$.

A \textbf{line bundle} is a $V$-bundle where $V\cong\C$, the
$1$-dimensional complex vector space. In this case if we choose a
basis for each $\C$ then the parallel transport is just multiplication
by an element of $\C^*=\C\setminus\{0\}$.  Furthermore, the monodromy
of a cycle is in $\C^*$ and does not depend on the start vertex of the
cycle.

In this paper we will take $V=\C^1$ or $\C^2$, and use $H=\C^*$- or
$H=\SL$-connections.

\subsection{Laplacian\texorpdfstring{ $\Delta$}{}}

Let $\G$ be a graph with an $H$-connection and let $c\colon E\to\R_{>0}$ be a conductance associated to each edge.
We then define $\Delta:V_{\G}\to V_{\G}$ acting on sections by the formula
\begin{equation}\label{laplaciandef}
\Delta f(v) = \sum_{w:(v,w)\in E} c_{v,w}(f(v)-\phi_{w,v}f(w)).
\end{equation}
A section is said to be \textbf{harmonic} if it is in the kernel of $\Delta$.

\enlargethispage{17pt}
We define an operator $d$
from sections of the bundle over vertices to sections over the edges,
for an oriented edge $e=(v,w)$, by
\[df(e) = c_{v,w}(\phi_{v,e}f(v)-\phi_{w,e}f(w)),\]
and its ``adjoint''
\[d^*\omega(v) = \sum_{e\sim v} \phi_{e,v}\omega(e).\]
Then the Laplacian can be written $\Delta = d^*d$ \cite{Kenyon.bundle}.

A \textbf{cycle-rooted spanning forest (CRSF)} in a graph is a set of
edges each of whose components contains a unique cycle, that is, has
as many vertices as edges. A component of a CRSF is called a
\textbf{cycle-rooted tree (CRT)}.

\begin{theorem}[\cite{Forman,Kenyon.bundle}]\label{crsfthm1}
For a $\C^*$-connection,
\[\det\Delta=\sum_{\textrm{\rm CRSFs}}\,\, \prod_{\text{\rm edges $e$}} c(e)\prod_{\text{\rm cycles $C$}}\left(2-w(C)-\frac1{w(C)}\right),\]
where the sum is over cycle-rooted spanning forests, where the first
product is over all edges of the CRSF and the second is over cycles of
the CRSF, and $w(C)$ is the monodromy of the cycle $C$.
\end{theorem}

For a $\U_1$-connection, $\Delta$ is Hermitian and positive semidefinite
\cite{Kenyon.bundle}.
The monodromy of a cycle is in $\U_1$ and so $2-w(C)-1/w(C)\ge 0$.
We can define a probability measure on CRSFs where each CRSF
has probability proportional to its weight $\prod_{e} c(e)\prod_C\big(2-w(C)-\frac1{w(C)}\big)$,
provided there is a CRSF with nonzero weight.

A similar result holds for an $\SL$-connection. Now $\Delta$ is a
quaternion-Hermitian matrix, that is, a matrix with entries in $\GL$
which satisfies $\Delta_{i,j}=\Delta_{j,i}^*$, where
$\begin{bmatrix}a&b\\c&d\end{bmatrix}^*=\begin{bmatrix}d&-b\\-c&a\end{bmatrix}$.
Its q-determinant counts CRSFs:

\begin{theorem}[\cite{Kenyon.bundle}]\label{crsfthm2}
For an $\SL$-connection,
\[\Qdet\Delta=\sum_{\text{\rm CRSFs}}\,\,\prod_{\text{\rm edges $e$}} c(e)\prod_{\text{\rm cycles $C$}}(2-\Tr w(C)),\]
where the sum is over cycle-rooted spanning forests,
where the first product is over all edges of the CRSF and the second
is over cycles of the CRSF, and $w(C)$ is the monodromy of the cycle $C$.
\end{theorem}

\enlargethispage{12pt}
In the case of an $\SU_2$ connection, any cycle with a nontrivial
monodromy has a positive weight, so these weights define a natural
probability measure, provided there is a CRSF with nonzero weight.

For information on q-determinants, see \cite{Dyson}; for the purposes
of this paper one can define $\Qdet M=\sqrt{\det \widetilde M},$ where
if $M$ is an $N\times N$ matrix with entries in $\GL$ then $\widetilde
M$ is the $2N\times 2N$ matrix with $\C$ entries obtained by replacing
each entry of $M$ by its $2\times 2$ block of complex numbers. In the
cases of primary interest $M$ is a quaternion-Hermitian, or
``self-dual'', matrix; for self-dual matrices $\Qdet$ is a polynomial
in the matrix entries. Matrices with $\GL$ entries enjoy many of the
properties of usual matrices: for example, multiplication and addition
work the same way. The inverse of a self-dual matrix is well-defined
and is both a left- and right-inverse, see e.g., \cite{Dyson}.

\subsection{Dirichlet boundary conditions}\label{dirichletBC}

If $B\subset \G$ is a set of vertices, the Laplacian with Dirichlet
boundary conditions at $B$ is defined on sections over $\G\setminus B$
by the same formula \eqref{laplaciandef} above with $v\in \G\setminus
B$ and the sum over all of $\G$.  In other words $\Delta$ is just a
submatrix of the usual Laplacian on $\G$.  Its determinant also has an
interpretation. A CRSF on a graph with boundary $B$ is a set of edges
such that each component is either a CRT not containing any vertex of
$B$ or else a tree containing a single vertex of $B$.
(When $B=\varnothing$ this specializes to the previous definition.)
In this setting
Theorems~\ref{crsfthm1} and \ref{crsfthm2} have the same statements
(where tree components do not have any monodromy term). See
\cite{Kenyon.bundle}.

\subsection{Green's function \texorpdfstring{$\Gv$}{G}}

The usual Green's function $G$ for the standard Laplacian (with Dirichlet
boundary conditions) is the inverse of the Laplacian.  It has the
probabilistic interpretation that $G_{p,q}$ is $(\sum_r c_{q,r})^{-1}$
times the expected number of visits to $q$ of a simple random walk
started at $p$ (and stopped at the boundary); equivalently, it is
$(\sum_r c_{q,r})^{-1}$ times the sum over all paths from $p$ to $q$
which do not hit the boundary, of the probability of the path.

In the case of a graph with connection, the Green's function $\Gv$ is
again the inverse of the Laplacian, and has a similar probabilistic
interpretation:
\begin{proposition} \label{Gv}
  $\Gv_{p,q}$ is $(\sum_r c_{q,r})^{-1}$ times the sum over all paths
  from $p$ to $q$ of the product of the parallel transports
  along the path (from $q$ to $p$) times the path probability (from
  $p$ to $q$), when the sum converges absolutely.  
  This sum converges absolutely for finite connected graphs with boundary and $\U_1$ or $\SU_2$ connections,
  and will be matrix-valued in the case of an $\SU_2$ connection.
\end{proposition}

\begin{proof}
Using the above definition of $\Gv$ as a sum over paths,
\[\sum_r \Gv_{p,r}\Delta_{r,q}=
\Gv_{p,q}\sum_{r\sim q} c_{q,r}-\sum_{r\sim
  q}\Gv_{p,r}c_{r,q}\phi_{q,r}\] and since any nontrivial path to $q$
must have last step from a neighbor of $q$, this equals zero unless
$p=q$ and the path has length $0$, in which case the second sum is
zero and the first term is $(\sum_r c_{q,r})^{-1}\sum_{r} c_{r,q}=1$.
\end{proof}

\subsection{Response matrix \texorpdfstring{$\L$}{L}}

Let $\No$ be a nonempty set of nodes of $\G$, and $n=|\No|$.  For each
node $v$ pick a preferred basis for $V_v$, the vector space over~$v$.

We define an $n\times n$ matrix $\L=\L_{\Phi}$ (with entries in $H$),
the response matrix, or Dirichlet-to-Neumann matrix, from this data:
$\L\colon V^{\No}\to V^{\No}$ is (minus) the Schur complement of the
Laplacian $\Delta$ to $\No$.  That is, $\L$ is defined as follows.
Order the vertices so that $\No$ comes first.  In this ordering the
Laplacian is
\[\Delta=\begin{bmatrix}A&B\\B^*&C\end{bmatrix}.\]
Then $\L=-A+BC^{-1}B^*$.  Note that $C$ is the Laplacian with
Dirichlet boundary conditions at $\No$.
Since $\det C$ is a weighted sum of cycle-rooted groves (defined below)
and $\G$ is connected to $\No$, $\det C$ is positive
for connections in $\U_1$ or $\SU_2$, and $\det C$ is generically nonzero for other connections.

From the viewpoint of harmonic functions, $\L$ is the Dirichlet-to-Neumann matrix:
given $f\in V^{\No}$,
find the unique section $h$ with boundary values $f$ at the nodes and harmonic
at the interior (non-node) vertices.
Then $\L f$ is $-\Delta h$ evaluated at the nodes.
To see this, let $h_1$ be $h$ at the interior vertices, that is,
$h=\begin{bmatrix}f\\h_1\end{bmatrix}$.
If \[\Delta\, h = \begin{bmatrix}A&B\\B^*&C\end{bmatrix}
\begin{bmatrix}f\\h_1\end{bmatrix}=
\begin{bmatrix}c\\0\end{bmatrix},\]
then $B^*f+C h_1=0$, i.e., $h_1=-C^{-1}B^*f$,
and then $c=A f-BC^{-1}B^*f = -\L f$.

The response matrix $\L$ has entries which are functions of the
parallel transports. See \tref{Lentries} below for an explicit
probabilistic interpretation of the entries of~$\L$.
In order to define $\L$ as a matrix one must choose a basis in $V_v$
for each node~$v$.  Base changes in the $V_v$ then act on $\L$ by
conjugation by diagonal matrices with entries in $H$.

\begin{lemma}
  When the Laplacian $\Delta$ is nonsingular, the response matrix $\L$ is given by
  \[\L=-(\Delta^{-1}|_\No)^{-1} = -(\Gv|_\No)^{-1}.\]
\end{lemma}
\begin{proof}
  Write $\Delta= \begin{bmatrix}A&B\\B^*&C\end{bmatrix}$ where $A$ is
  the submatrix indexed by the nodes.  Submatrix~$C$ is invertible since
  it is the Laplacian with Dirichlet boundary conditions.  We have
\[\Delta=\begin{bmatrix}A-BC^{-1}B^*&BC^{-1}\\0&I\end{bmatrix}
\begin{bmatrix}I&0\\B^*&C\end{bmatrix}\]
and using $\L=-A+BC^{-1}B^*$,
\begin{align*}
\Delta^{-1}&=\begin{bmatrix}I&0\\-C^{-1}B^*&C^{-1}\end{bmatrix}
\begin{bmatrix}-\L^{-1}&\L^{-1}B C^{-1}\\0&I\end{bmatrix}=
\begin{bmatrix}-\L^{-1}&*\\{}*&{}*\end{bmatrix}.\qedhere
\end{align*}
\end{proof}

If $\Gv$ is the Green's function with boundary at node $n$, then
$\Gv=\tilde\Delta^{-1}$, where the Dirichlet Laplacian $\tilde\Delta$
is obtained from $\Delta$ simply by removing row and column~$n$.
Since $\L$ has the same response as $\Delta$ on the set $\No$,
the response matrix of $\tilde\Delta$ is just $\L$ with row and column $n$ removed.
Thus $\left[\L_{i,j}\right]^{j=1,\dots,n-1}_{i=1,\dots,n-1} = -\big(\tilde\Delta^{-1}|_{\No\setminus\{n\}}\big)^{-1}$,
or equivalently,
\begin{equation} \label{LG}
  \left[\L_{i,j}\right]^{j=1,\dots,n-1}_{i=1,\dots,n-1} =-
  \left(\left[\Gv_{i,j}\right]^{j=1,\dots,n-1}_{i=1,\dots,n-1}\right)^{-1}.
\end{equation}
By perturbing $\Delta$, we see that \eqref{LG} holds even if the
Laplacian $\Delta$ is singular, so long as the Dirichlet Laplacian is
nonsingular.

\section{Graphs on surfaces}

Let $\Sigma$ be an oriented surface, possibly with boundary, and $\G$
a graph embedded on $\Sigma$ in such a way that complementary
components (the connected components of the surface after it is cut
along the edges of $\G$) are contractible or peripheral annuli (that
is, an annular neighborhood of a boundary component).  We call the
pair $(\G,\Sigma)$ a \textbf{surface graph} (see Figure~\ref{fig:surface}).

\begin{figure}[htbp]
  \vspace{-6pt}
  \begin{center}
    \includegraphics[height=1.4in]{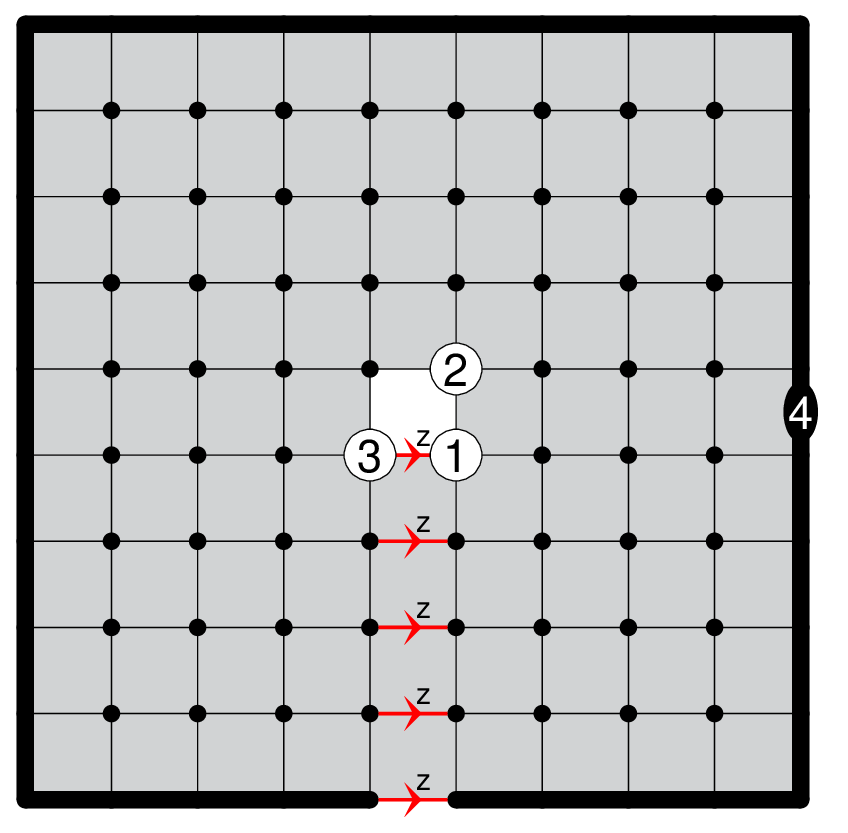} \hskip.5in
    \includegraphics[height=1.4in]{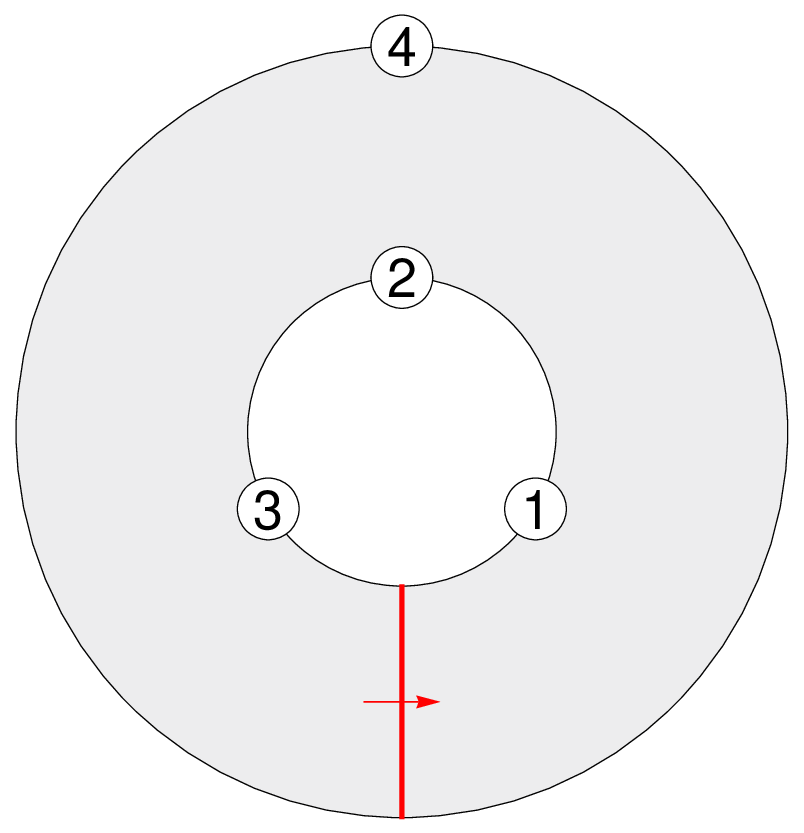}
  \end{center}
  \vspace{-8pt}
  \caption{\label{figure1}On the left is a graph $\G$ with wired boundary conditions: the
    outer boundary is one vertex (and the bottom edge is a self-loop) embedded in
    an annulus whose inner boundary is
    one of the squares of the grid.  There is a ``zipper''
    (edges crossing a dual path) connecting the
    inner boundary to the outer boundary of the annulus, and edges crossing
    the zipper have parallel transport $z$ from their left endpoint to
    their right endpoint.  We have labeled four of the vertices on the
    boundary of $\Sigma$, which we call nodes.  On the right is a schematic
    diagram of the surface graph.
    }
  \label{fig:surface}
\end{figure}

\subsection{Nodes and interior vertices}

For each boundary component $C$ of the surface $\Sigma$ there is a
``peripheral'' cycle on $\G$, bounding the annular complementary component whose other
boundary is $C$.
We select from this cycle a (possibly empty) set of vertices.  The
union of these special vertices over all boundary components will be
the nodes $\No$; the non-node vertices are \textbf{interior vertices}
(even though these may be on the boundary of $\Sigma$).

Planar maps, in which $\Sigma$ is a topological disk, are examples of
surface graphs: these are called \textbf{circular planar graphs} in
\cite{CIM}. In this case the nodes are a subset of the vertices on the
outer face.

\subsection{Flat bundles}

Given a surface graph $(\G,\Sigma)$, a vector bundle on $\G$ with
connection $\Phi$ is \textbf{flat} if it has trivial monodromy around
any loop which is contractible on $\Sigma$.  In this case, the
monodromy around a loop only depends on the homotopy class of the
(pointed) loop in $\pi_1(\Sigma)$, and so the monodromy determines a
representation of $\pi_1(\Sigma)$ into $\Aut(V_p)$, where $p$ is the
base point.  This representation depends on the base point~$p$ for
$\pi_1$; choosing a different base point will conjugate the
representation.

Conversely, let $\rho\in\text{Hom}(\pi_1(\Sigma),\Aut(V))$ be a
representation of $\pi_1(\Sigma)$ into $\Aut(V)$; there is a
unique (up to gauge equivalence) flat bundle with monodromy $\rho$.
It is easy to construct: for example start with a trivial bundle on a
spanning tree of~$\G$; for each additional edge the parallel transport
along it is determined by the topological type of the resulting cycle
created.

In the case of a line bundle, $\Aut(\C)=\C^*$ is abelian and the
monodromy of a loop is well defined without regard to base point.
Moreover in this case $\rho(\gamma)$ only depends on the homology
class of the loop $\gamma$, since any map from $\pi_1(\Sigma)$ to an abelian
group factors through $H_1(\Sigma)$.

\section{The response matrix and probabilities}

\subsection{Circular planar graphs}

In the case of a planar graph, there is no monodromy and $\L=L$ is a
matrix of real numbers.  This case was analyzed by \cite{CdV}, see
also \cite{CGV, CIM}.  Colin de Verdi\`ere showed that response
matrices $L$ of planar graphs are characterized by having nonnegative
``non-interlaced'' minors.  Given two disjoint subsets of nodes $R$ and
$S$, we say that $R$ and $S$ are non-interlaced if $R$ and $S$ are
contained in disjoint intervals in the circular order on the nodes.
When $|R|=|S|$, the corresponding minor is $\det(L_R^S)\ge0$ (the
determinant of the submatrix whose rows are indexed by $R$ and columns
by $S$).

In \cite[Proposition~2.8]{KW1} (see also \cite[Lemma~4.1]{CIM}), there
is an interpretation of the entries of $L$ in terms of groves.  A
\textbf{grove} is a spanning forest with the property that every
component contains at least one node.  The weight of a grove is the
product of the conductances of its edges.

\begin{theorem}[\cite{CIM,KW1,Fomin}]\label{Lentriesplanar}
  For disjoint non-interlaced
  subsets $R,S\subset\No$ with $|R|=|S|$,
  $\det(L_R^S)$ is a ratio of two terms:
  the denominator is the weighted sum of groves in which every node is in
  its own component, and the numerator is the weighted sum of groves in which
  the nodes in $R$ are connected pairwise with nodes in $S$,
  and other nodes are in their own component.
\end{theorem}

In particular this proves that the non-interlaced minors
of $L$ are nonnegative.

Groves can be grouped into subsets according to the way they partition
the nodes (that is, the way the nodes are connected in a grove). For
example, a grove of type $1,2\mid 3,4,5\mid 6$ is one in which nodes $1$ and $2$
are in a tree, nodes $3,4,5$ are in a second tree, and node $6$ is in
its own tree.  For a partition $\sigma$ of the nodes, we let
$Z[\sigma]$ denote the weighted sum of groves of type $\sigma$.  For
circular planar graphs with $n$ nodes on the boundary, we previously
showed \cite{KW1} how to compute the ratio $Z[\sigma]/Z[1|2|\cdots|n]$
for any planar partition $\sigma$ of $\{1,2,\ldots,n\}$.  This ratio is an
integer-coefficient polynomial in the $L_{i,j}$ \cite{KW1}.

It is useful to allow the partition $\sigma$ to have missing indices,
such as $1,2|4,5|6$.
The nodes with the missing labels are
treated as internal vertices which can occur in any part, so that, e.g.,
$Z[1,2|4,5|6]=Z[1,2,3|4,5|6]+Z[1,2|3,4,5|6]+Z[1,2|4,5|3,6]$.

\subsection{\texorpdfstring{$\L$}{L} matrix entries}

Like
in \tref{Lentriesplanar}, in the case of a flat bundle on a
surface graph there is a combinatorial interpretation of the entries
of $\L$.  A collection of edges of a surface graph $(\G, \Sigma)$ is a
\textbf{cycle-rooted grove (CRG)} if each component is either a
CRT (a component containing one cycle) not containing a node, or
a tree containing at least one node.  Moreover for each CRT
component, the cycle must be topologically nontrivial.  A CRG is
distinguished from a CRSF by the fact that in a CRG the tree
components may contain several nodes, while in a CRSF the tree
components contain a unique node.

A CRG has a weight which is the product of its edge conductances times
the product over its cycles of $2-w-1/w$ (for a line bundle) or
$2-\Tr(w)$ (for an $\SL$-bundle), where $w$ is the monodromy around
the cycle.  For a partition $\sigma$ of the nodes, we define
\begin{gather*}
\Zv[\sigma] := \text{weighted sum of cycle-rooted groves of type $\sigma$}\\
\Zv := \text{weighted sum of cycle-rooted groves in which all nodes are connected}
\end{gather*}
For example, the weighted sum of CRSFs is $\Zv[1|2|\cdots|n]$.
Suppose the partition $\sigma$ is a partial pairing, i.e., $\sigma$
consists of doubleton and singleton parts, say
$\sigma=r_1,s_1\mid\cdots\mid r_k,s_k\mid t_1\mid\cdots\mid t_\ell$.
We can define
\[
 \Zv[{}_{r_1}^{s_1}|\cdots|{}_{r_k}^{s_k}|t_1|\cdots|t_\ell] :=
\sum_{\text{CRGs of type $\sigma$}}
\text{(weight of CRG)} \times
\prod_{i=1}^k \text{parallel transport to $r_i$ from $s_i$}
\]
for line bundles (so that the structure group is commutative and the above product makes sense),
and for vector bundles when $\sigma$ has only one doubleton part.

\begin{theorem}\label{Lentries}
  If $i\ne j$, then
  \[ \L_{i,j} = \frac{\Zv\big[{}_i^j|\text{\rm(nodes other than $i$ and $j$ in singleton parts)}\big]}{\Zv[1|2|\cdots|n]}.\]
\end{theorem}

\begin{proof}
  Let us first do the line bundle case.  Let
  $\Delta=\begin{bmatrix}A&B\\B^*&C\end{bmatrix}$ be the Laplacian of
  $\G$, with $A$ indexed by the nodes $\No$.

  We make a new graph $\tilde\G$ by adding an edge $e_{i,j}$ (with unit conductance) to $\G$
  which connects $i$ and $j$ and has parallel transport $z$ when
  directed from $i$ to $j$.  Let $\tilde\Delta$ be the line bundle
  Laplacian on the new graph $\tilde\G$, with Dirichlet boundary
  conditions at the nodes except nodes $i$ and $j$, that is,
  $\tilde\Delta=\begin{bmatrix}a&b\\b^*&C\end{bmatrix}$ where
  $a=\begin{bmatrix}A_{i,i}+1&A_{i,j}-z^{-1}\\A_{j,i}-z&A_{j,j}+1\end{bmatrix}$
  and $b$ is the $i$th and $j$th column of $B$.

  By \tref{crsfthm1} (and its extension discussed in
  section~\ref{dirichletBC}), $-[z](\det\tilde\Delta)$ is a sum of
  CRSFs with each node except $i,j$ in its own tree component, and
  nodes $i,j$ in a cycle containing edge $e_{i,j}$, and the weight
  includes the parallel transport of the path in $\G$ from $j$ to $i$.
  (Here $[z^\alpha]f(z)$ refers to the coefficient of $z^\alpha$ in
  $f(z)$.)  We can write
\begin{align}
  \tilde\Delta&=\begin{bmatrix}a-b\,C^{-1}b^*&b\,C^{-1}\\0&I\end{bmatrix}
  \begin{bmatrix}I&0\\b^*&C\end{bmatrix}, \label{uppertri} \\
  \det\tilde\Delta&=\det[a-b\,C^{-1}b^*] \det C \notag \\
  -[z]\det\tilde\Delta&=-[z^0][a-b\,C^{-1}b^*]_{1,2} \det C \notag
\end{align}
   However \[[z^0][a-b\,C^{-1}b^*]_{1,2}=[A-BC^{-1}B^*]_{i,j}=-\L_{i,j}.\]
   Finally, $\det C$ is the sum of CRSFs.

   The proof in the $\SL$-bundle case is similar.  Let
   $\Delta=\begin{bmatrix}A&B\\B^*&C\end{bmatrix}$ be the $\SL$-bundle
   Laplacian of $\G$.  Add an edge $e_{i,j}$ to $\G$ from node $i$ to node $j$
   with parallel transport $M\in\SL.$ As above let
   $\tilde\Delta=\begin{bmatrix}a&b\\b^*&C\end{bmatrix}$ where
   $a=\begin{bmatrix}A_{i,i}+I&A_{i,j}-M^{-1}\\A_{j,i}-M&A_{j,j}+I\end{bmatrix}$
   and $b$ is the $i$th and $j$th column of $B$.  Now $\det\tilde\Delta$
   gives a weighted sum of CRSFs with each node except $i,j$ in its
   own tree component, where the weight is the product of the
   monodromies along the cycles.  If the CRSF contains a cycle that
   uses edge $e_{i,j}$, then the monodromy of this cycle will depend
   on $M$, and otherwise, the weight of the CRSF does not depend on
   $M$.  We write
   \begin{equation}\label{KM}
     \Qdet\Delta'= C_0 + \sum_{\omega} C_{\omega}(2-\Tr(K_\gamma M)),
   \end{equation}
   where the sum is over configurations $\omega$ with a cycle $\gamma$
   containing edge $e_{i,j}$, $K_\gamma$ is the parallel transport
   to $i$ from $j$ in the cycle $\gamma$, and $C_0,C_\gamma$ and
   $K_\gamma$ do not depend on $M$.

   We have \eqref{uppertri} in this case as well, where $C$ does not
   depend on $M$.  Letting $D=b\,C^{-1}b^*$, which does not depend on
   $M$, we can write
   \[a-b\,C^{-1}b^*=\begin{bmatrix}A_{i,i}+I-D_{i,i}&A_{i,j}-D_{i,j}-M^{-1}\\A_{j,i}-D_{j,i}-M&A_{j,j}+I-D_{j,j}\end{bmatrix}.\]
   This is a $2\times2$ matrix with entries in $\GL$.
   The reader may check that for a $2\times 2$ matrix with entries in $\GL$,
   \[ \Qdet\begin{bmatrix} x I & Y \\ Y^* & z I\end{bmatrix} = x z -\det Y = xz - \frac12 \Tr Y Y^*.\]
   Consequently
   \[\Qdet(a-b\,C^{-1}b^*)=\frac12 \Tr[(A_{i,j}-D_{i,j}-M^*)(A_{i,j}^*-D_{i,j}^*-M)]+C_1=\Tr[(A_{i,j}-D_{i,j})M]+C_2\] where
   $C_1$ and $C_2$ do not depend on $M$.  Comparing with \eqref{KM} we see
   that, since $M$ was arbitrary,
   \begin{align*}
     -\sum_\omega C_\omega K_\omega= A_{i,j}-D_{ij}&=[A-BC^{-1}B^*]_{i,j}\det C=-\L_{i,j}. \hfill\qedhere
   \end{align*}
\end{proof}

Principal minors of $\L$ also have probabilistic interpretations:
\begin{theorem}\label{Lprincminors}
  Suppose $T\subset\No$ and $Q=\{q_1,\dots,q_\ell\}=\No\setminus T$.  Then
  \[\det\L_T^T = (-1)^{|T|}\frac{\Zv[q_1|q_2|\cdots|q_\ell]}{\Zv[1|2|\cdots|n]}.\]
\end{theorem}

\begin{proof} Order the vertices of $G$ by first $\No\setminus T$, then $T$,
then the internal nodes. In this order we have
\[\Delta=\begin{bmatrix}A_1&A_2&B_1\\A_2^*&A_3&B_2\\B_1^*&B_2^*&C
\end{bmatrix}.\]
Then \[\L=-\begin{bmatrix}A_1&A_2\\A_2^*&A_3\end{bmatrix}+
\begin{bmatrix}B_1\\B_2\end{bmatrix}C^{-1}\begin{bmatrix}B_1^*&B_2^*\end{bmatrix}\]
and $\det\L_T^T=\det(-A_3+B_2C^{-1}B_2^*).$ The proof follows from the identity
\[\begin{bmatrix}A_3&B_2\\B_2^*&C
\end{bmatrix}=\begin{bmatrix}A_3-B_2C^{-1}B_2^*&B_2C^{-1}\\0&I
\end{bmatrix}\begin{bmatrix}I&0\\B_2^*&C\end{bmatrix}
\] upon taking determinants: the left-hand side determinant is the
weighted sum of CRGs of $\G_T$, the right-hand side determinant
is $(-1)^{|T|}$ times
the product of
$\det\L_T^T$ and $\det C$ which counts CRSFs.
In the $\SL$-case we used the fact that
$\Qdet\begin{bmatrix}X&Y\\0&I\end{bmatrix}=\Qdet\begin{bmatrix}X&0\\Y&I\end{bmatrix}=\Qdet X$.
\end{proof}

For line bundles there is an interpretation of more general minors:
\begin{theorem}\label{LABC}
  Suppose $Q=\{q_1,\dots,q_\ell\}$,
$R=\{r_1,\dots,r_k\}$, $S=\{s_1,\dots,s_k\}$, and~$T$
are disjoint sequences of nodes for which $|R|=|S|$ and
  $\No=Q\cup R\cup S\cup T$.
  Then
\[
 (-1)^{|T|} \det\L_{R,T}^{S,T} = \sum_{\text{\rm permutations $\rho$}}(-1)^\rho\frac{\Zv\big[{}_{r_1}^{s_{\rho(1)}}|\cdots|_{r_k}^{s_{\rho(k)}}|q_1|\cdots|q_\ell\big]}{\Zv[1|2|\cdots|n]}
.\]
\end{theorem}

\begin{proof}
We use a block LU-factorization of $\Delta$, as in the previous proof, to find
$(-1)^{|R|+|T|} \det \L_{R,T}^{S,T} \det \Delta_I^I=\det\Delta_{R,T,I}^{S,T,I}$, where $I$ is the set of internal nodes.
$\det\Delta_I^I=\Zv[1|\cdots|n]$.   So we need to evaluate $\det\Delta_{R,T,I}^{S,T,I}$.
The proof now follows the proof of \tref{crsfthm1} which is found in
\cite[proof of Theorem~1]{Kenyon.bundle}.
Write $\Delta=d^*d$ where $d$ is the operator from
sections over the vertices to sections over the edges.
Then $\Delta_{R,T,I}^{S,T,I}=d_{S,T,I}^*d_{R,T,I}$ where $d_X$
is the restriction of $d$ to sections over $X$.
By the Cauchy-Binet theorem,
\[\det d_{S,T,I}^*d_{R,T,I}=\sum_Y\det (d^Y_{S,T,I})^*\det d_{R,T,I}^Y,\]
where the sum is over collections of edges $Y$ of cardinality $|S\cup T\cup
I|$.  The nonzero terms in the sum are collections of edges in which
each component is a CRT if we glue $r_i$ to $s_i$ for
each~$i$. Equivalently, each component is either a CRT or a tree
containing a unique $r\in R$ and $s\in S$.  The weight of a component
with a cycle is $2-w-1/w$ where $w$ is the monodromy of the cycle; the
weight of a path is the parallel transport to the $r$ from the $s$.
It remains to compute the signature of each configuration.

This signature is the same as
the signature in the case of a trivial bundle, which is determined by \cite{CIM}
to be the signature of the permutation from $R$ to $S$ determined by the pairing.
\end{proof}

\subsection{Cycle-rooted grove probabilities}

\begin{theorem}
  \label{thm:all-types}
  The probability of any topological type of CRG involving only
  two-node connections and loops is a function of $\L$ and the
  weighted sum of CRSFs.
\end{theorem}

\begin{proof}
  By a result of Kenyon \cite{Kenyon.bundle} (based on a theorem of
  Fock and Goncharov \cite{FG}), on a graph embedded on a surface
  with no nodes one
  can compute the probability of any topological type of CRSF (that
  is, the probability that a CRSF has a given set of homotopically
  nontrivial cycles up to isotopy) from the determinant of the
  Laplacian considered as a function on the space of flat
  $\SL$-connections. Indeed, as $X$ runs over all possible ``finite laminations'',
  that is, isotopy classes of collections
  of finite, pairwise disjoint, topologically nontrivial simple loops on the surface,
  the products $\prod_{\text{cycles in }X}(2-\Tr w)$ form a basis
  for a vector space (the vector space of regular
  functions on the representation variety) and the Laplacian
  determinant is an element of this vector space.
  In other words, \tref{crsfthm2} above shows that
  $\det\Delta=\sum_X C_X\prod_{\text{cycles in }X}(2-\Tr w)$, where $X$
  runs over finite laminations; such an expression determines
  each coefficient $C_X$ uniquely.

  To compute the probability that a random CRG $Y$ on $\Sigma$ has a
  fixed topology of node connections, add edges $e_{i,j}$ to $\G$
  connecting endpoints of all two-node connections $i\to j$ of~$Y$;
  the resulting graph $\G'$ can be embedded on a surface $\Sigma'$
  containing~$\Sigma$, and the union of~$Y$ and the new edges is a
  CRSF on $\G'$. (We obtain $\Sigma'$ from $\Sigma$ by gluing a single strip
  running from $i$ to $j$ for each $e_{i,j}$; in this way cycles containing different
  $e_{i,j}$s are in different homotopy classes.)

  Any CRG on $\G$ with the same node connection type as $Y$ can be
  completed to a CRSF on $\Sigma'$ by adding the edges $e_{i,j}$.
  Conversely (since each added edge $e_{i,j}$ is in a different homotopy
  class on $\Sigma'$) each CRSF of this topological type comes from a
  CRG on $\G$ with the same connection type as $Y$.

  The flat connection on $\G$ can be extended to a flat connection on
  $\G'$ by taking generic parallel transports $\phi_{i,j}$ along the $e_{i,j}$.

  It remains to show that the Laplacian determinant of $\G'$ is a
  function of $\L$, the new parallel transports $\Phi=\{\phi_{i,j}\}$,
  and $\det C$. However
  $\Delta^{\G'}=\Delta+ S_\Phi$ where $S_\Phi$ is supported on the nodes;
  using $\Delta=\begin{bmatrix}A&B\\B^*&C\end{bmatrix}$
  we have
  \begin{align*}
  \det\Delta^{\G'}=\det\begin{bmatrix}A+S_\Phi&B\\B^*&C\end{bmatrix}
   &=\det\begin{bmatrix} A+S_\Phi-BC^{-1}B^*&BC^{-1}\\0&I\end{bmatrix}
         \begin{bmatrix} I&0\\B^*&C\end{bmatrix}\\
   &=\det(-\L+S_\Phi)\det C. \qedhere
  \end{align*}
\end{proof}

\section{Basic surface graphs} \label{basic}

The simplest non-circular-planar case is the annulus. Since
$\pi_1(\Sigma)$ is abelian in this case it usually suffices to
consider a line bundle rather than a two-dimensional bundle.  The $\L$
matrix then depends on a single variable $z\in\C^*$ which is the
monodromy of a flat connection. For simplicity we choose a connection
which is the identity on all edges except for the edges crossing a
``zipper'', that is, a dual path connecting the boundaries; these
edges have parallel transport $z$.

Suppose $(\G,\Sigma)$ is a surface graph on an annulus with $n_1$
nodes on one boundary component and $n_2$ on the other. Then $\L$ is
an $(n_1+n_2)$-dimensional matrix with entries which are rational in
$z$. Let $\Zv_0=\Zv[1|2|\cdots|n]$ be the weighted sum of CRSFs of $\G$
(CRGs in which each node is in a separate component).
We have
\[ \Zv(1|2|\cdots|n) = \sum_k \alpha_k(2-z-z^{-1})^k,\]
where $\alpha_k$ is the weighted sum of CRSFs with $k$ cycles winding around the annulus.

While we have not attempted to show that every connection probability
can be computed via the $\L$-entries, we present here some cases of
small $n_1,n_2$.

\subsection{Annulus with \texorpdfstring{$(2,0)$}{(2,0)} boundary nodes}

Suppose there are two nodes on one boundary and none on the other.
Then $\L_{1,2}(z)$ counts connections to $1$ from $2$.  There are only
two topologically different configurations, which are illustrated in
\fref{20case}.  We have the following theorem.

\begin{theorem}
  $\frac{\partial}{\partial z}\log \L_{1,2}(z)|_{z=1}$ is the
  probability that the LERW to $1$ from $2$ crosses the zipper.
\end{theorem}

\begin{proof}
  Let $A_k$, $B_k$, and $\alpha_k$ (respectively) be the weighted sum of
  cycle-rooted groves which contain $k$ cycles winding around the
  annulus, and in which nodes $1$ and $2$ are (respectively) connected
  by a path not crossing the zipper, connected by a path crossing the
  zipper, or are not connected.  Then
  \begin{align*}
    \Zv[{}_1^2] &= \sum_{k\geq 0} [A_k (2-z-1/z)^k + B_k z (2-z-1/z)^k] = A_0 + B_0 z + O((z-1)^2)\\
    \Zv[1|2] &= \sum_{k\geq 0} \alpha_k (2-z-1/z)^k = \alpha_0 + O((z-1)^2)
  \end{align*}
  By \tref{Lentries}, $\L_{1,2} = \Zv[{}_1^2]/\Zv[1|2]$, so
  \begin{align*}
    \frac{\partial}{\partial z}\log\L_{1,2}(z) &=
    \frac{\partial_z\Zv[{}_1^2]}{\Zv[{}_1^2]} - \frac{\partial_z \Zv[1|2]}{\Zv[1|2]}
    = \frac{B_0}{A_0+B_0} + O(z-1).  \qedhere
  \end{align*}
\end{proof}

\begin{figure}[h]
  \begin{center}\includegraphics[width=2.5in]{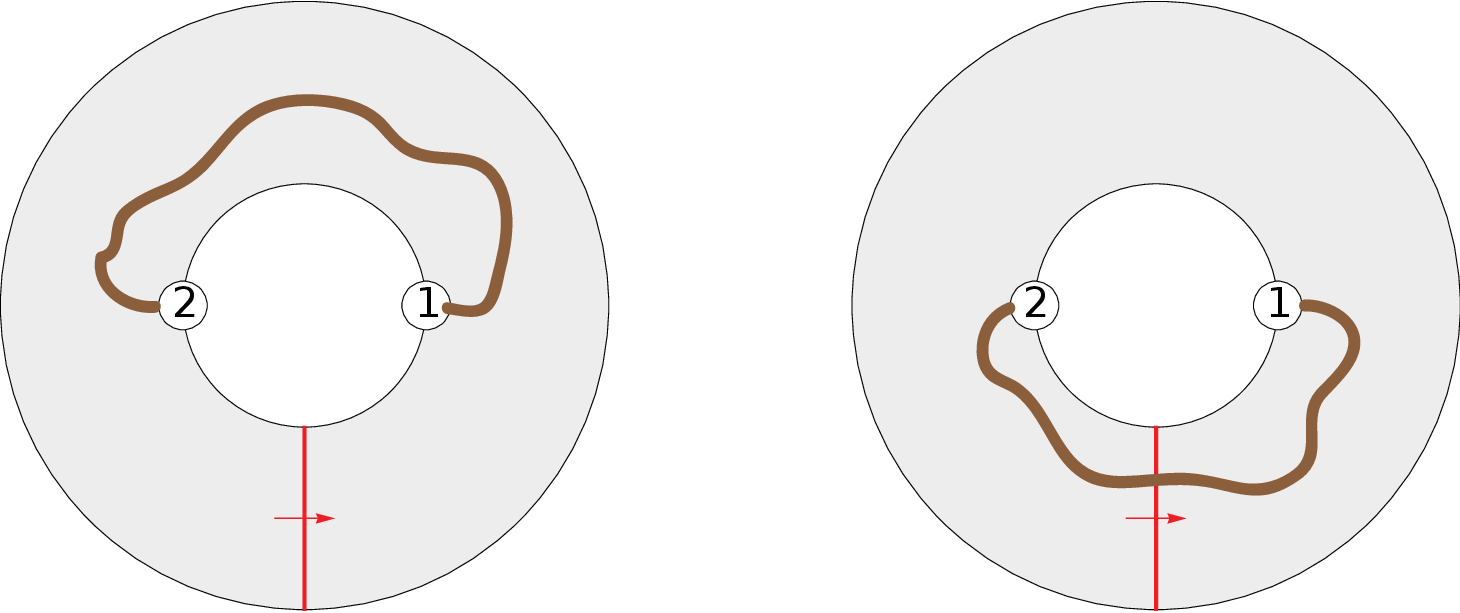}\end{center}
  \vspace{-8pt}
  \caption{The two topologically distinct ways to connect the nodes on an annulus with $(2,0)$ boundary nodes. \label{20case}}
\end{figure}

\subsection{Annulus with \texorpdfstring{$(1,1)$}{(1,1)} boundary nodes}

\begin{figure}[t]
  \begin{center}\includegraphics[width=\textwidth]{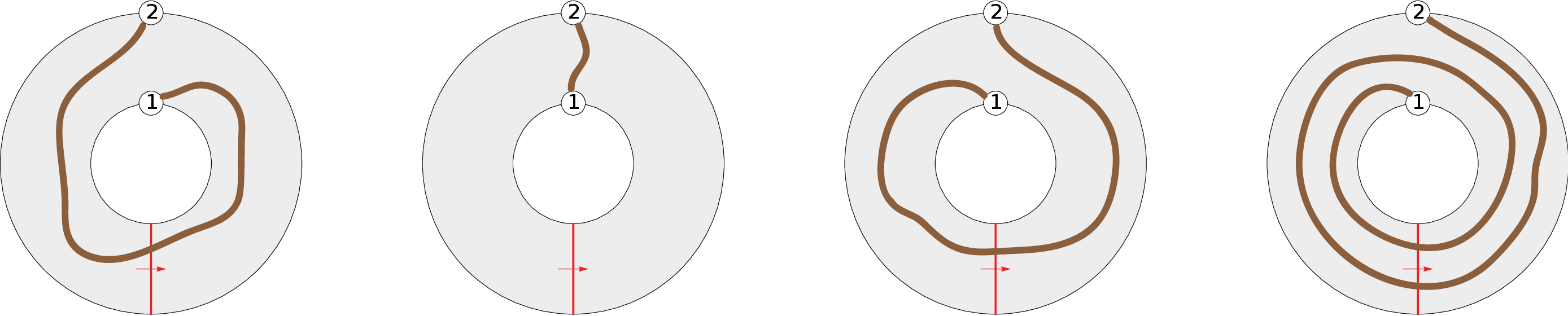}\end{center}
  \caption{Different numbers of windings of the path to $1$ from $2$.
  These configurations contribute to $c_{1}$, $c_0$, $c_{-1}$, and $c_{-2}$
  respectively.\label{11case}}
\end{figure}

We let $c_j$ denote the weighted sum of CRGs connecting $1$ and $2$
and such that the path from $1$ to $2$ crosses the zipper $j$ times
algebraically (see \fref{11case}).  Then $\Zv[{}_1^2]=\sum_{j\in\Z}
c_jz^j$, from which one can extract $c_j$ for each $j$.
\tref{Lentries} shows that $\L_{1,2}=\Zv[{}_1^2]/\Zv[1|2]$,
so $\L_{1,2}$ and $\Zv[1|2]$ together determine the winding
distribution (which we knew already from \tref{thm:all-types}).
But $\Zv[1|2]$ is also a function of $z$, so $\L_{1,2}$ does not by
itself determine the winding distribution.  But from $\L_{1,2}$ we
can extract the expected number of algebraic crossings of the zipper
via
\[ \E[\text{\# algebraic crossing of zipper}]=\left.\frac{\partial}{\partial z} \log\L_{1,2}\right|_{z=1}.\]

\subsection{Annulus with \texorpdfstring{$(3,0)$}{(3,0)} boundary nodes}

This is a case which can be derived from the $(2,0)$ case using \tref{LABC}.
Suppose nodes $1,2,3$ are in counterclockwise order on the inner boundary,
with a counterclockwise zipper between nodes $1$ and $3$.
Consider the case when all nodes are connected;
there are three possible configurations $A_1,A_2,A_3$
correspond to which complementary component of the triple connection
the outer boundary component lies (see \fref{30case}).  The numerator of $\L_{\;1,3}^{2,3}$
is $A_1+A_2+z A_3$, the numerator of $\L_{\;1,2}^{3,2}$ is
$z A_1+A_2+z A_3$, and the numerator of $\L_{\;2,1}^{3,1}$ is $z A_1+A_2+A_3.$
These three quantities suffice to determine $A_1,A_2,A_3$.
\begin{figure}[htbp]
\begin{center}
\includegraphics[height=1.in]{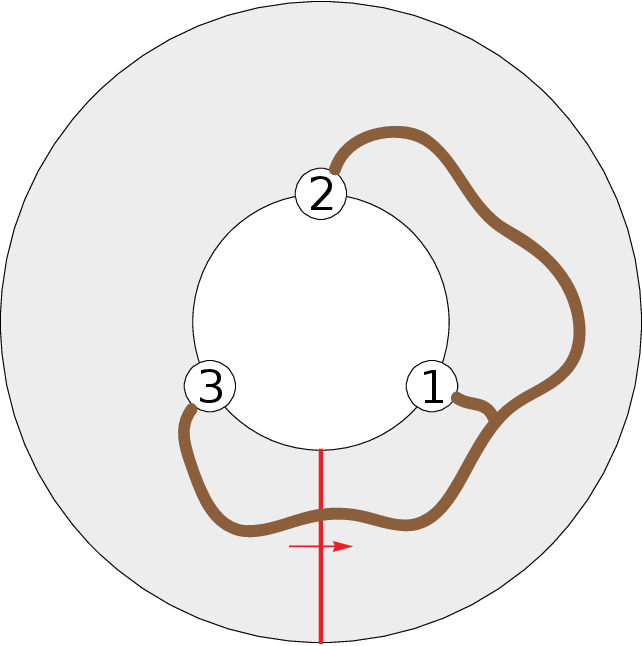} \hfil
\includegraphics[height=1.in]{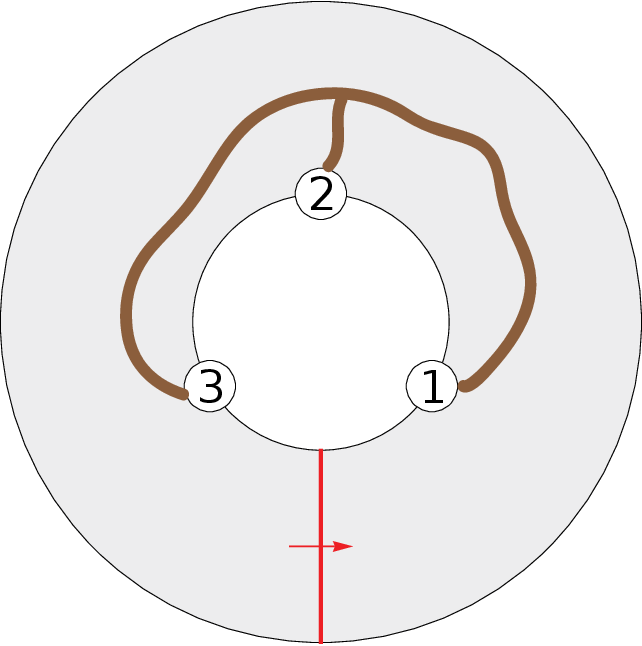} \hfil
\includegraphics[height=1.in]{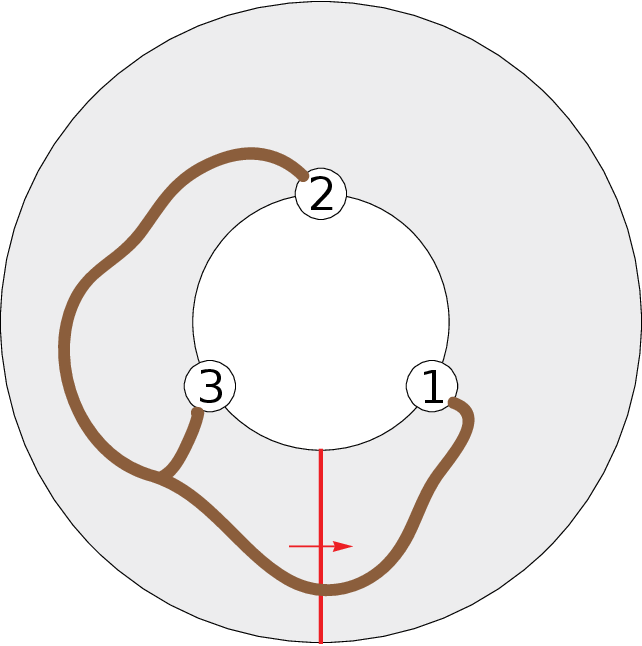}
\end{center}
\caption{\label{30case}The three topologically distinct subcases when
  the three nodes are connected (for the annulus with $(3,0)$ nodes).}
\end{figure}

\subsection{Annulus with \texorpdfstring{$(3,1)$}{(3,1)} boundary nodes} \label{sec:annulus-3-1}

This is a case we
\begin{window}[0,r,\includegraphics[height=1.in]{surface-graph-3}\raisebox{-1in}{\rule{0pt}{1.4in}},{}]
\noindent will need when we do the LERW computations.
Suppose there are
$4$ nodes in all, with nodes $1,2,3$ on the inner boundary in
counterclockwise order and $4$ on the outer boundary. Suppose the zipper
starts between $1$ and $3$ and is oriented counterclockwise, as in the
figure.  We wish to compute the ratios $Z[1,2|3,4]/Z[1|2|3|4]$
and  $Z[1,2|3,4]/Z[1,2,3,4]$ (as before, $Z[\sigma]$ denotes the weighted sum of groves of type~$\sigma$, for the trivial bundle).
\end{window}

Recall that $\Zv[{}_1^2|{}_3^4]$ denotes the weighted sum of cycle-rooted
groves of type $\,{}^2_1|{}^4_3\,$, times the parallel transport of
the path to node $1$ from $2$ and the path to node $3$ from node $4$,
so that for a trivial bundle, $\Zv[{}_1^2|{}_3^4]=Z[1,2|3,4]$.
Because of the path connecting nodes $3$ and $4$, in fact there will
be no cycles in the CRG.
Similarly, $\Zv[{}_1^3|{}_2^4]$ and $\Zv[{}_2^3|{}_1^4]$ denote the
weighted sum of CRGs of type $\,{}_1^3|{}_2^4\,$ and
$\,{}_2^3|{}_1^4\,$, times their respective parallel transports.

By \tref{LABC}, $\L_{\;1,2}^{3,4}$ has
numerator counting connections $\,{}^3_1|{}^4_2\,$ and
$\,{}^4_1|{}^3_2\,$ (with a minus sign).  Similarly for $\L_{\;1,3}^{2,4}$
and $\L_{\;1,4}^{2,3}$.  With $\Zv_0$ denoting the sum of CRSFs, we have
\begin{equation}
\begin{aligned}
\Zv_0 \det\L_{\;1,3}^{2,4}&=\Zv[{}_1^2|{}_3^4] - \Zv[{}_3^2|{}_1^4] = \Zv[{}_1^2|{}_3^4] - \Zv[{}_2^3|{}_1^4] \\
\Zv_0 \det\L_{\;3,2}^{1,4}&=\Zv[{}_3^1|{}_2^4] - \Zv[{}_2^1|{}_3^4]
 = \Zv[{}_3^1|{}_2^4] - \Zv[{}_1^2|{}_3^4] \\
\Zv_0 \det\L_{\;2,1}^{3,4}&=\Zv[{}_2^3|{}_1^4] - \Zv[{}_1^3|{}_2^4]
 = \Zv[{}_2^3|{}_1^4] - z^2 \Zv[{}_3^1|{}_2^4]
\end{aligned}
\label{sys:3,1}
\end{equation}

As a consequence
\begin{equation}\label{A12form}
\frac{\Zv[{}_1^2|{}_3^4]}{\Zv_0}=\frac{\L_{\;1,3}^{2,4} + z^2 \L_{\;3,2}^{1,4}+\L_{\;2,1}^{3,4}}{1-z^2}.
\end{equation}
When $z\to 1$ both the numerator and denominator of \eqref{A12form}
converge to $0$, so we evaluate the limiting ratio using l'H\^opital's
rule.  We can expand $\L_{u,v}=L_{u,v}+(z-1) L'_{u,v}+O((z-1)^2)$,
where $L_{u,v}$ is symmetric and $L'_{u,v}$ is antisymmetric.  Then in
the limit $z\to1$ this gives
\begin{align}
\frac{Z[1,2|3,4]}{Z[1|2|3|4]} &= \lim_{z\to 1}
\frac{\Zv[{}_1^2|{}_3^4]}{\Zv[1|2|3|4]}
= \lim_{z\to 1} \frac{\L_{\;1,3}^{2,4} + z^2 \L_{\;3,2}^{1,4}+\L_{\;2,1}^{3,4}}{1-z^2} \notag\\
&=
- L'_{1,2} L_{3,4} -  L'_{2,3} L_{1,4} -  L'_{3,1} L_{2,4}
+ L_{1,2} L_{3,4} - L_{1,3} L_{2,4}. \label{Pu(12|34)}
\end{align}

It is useful to express this above formula in terms of the Green's function
where the $n$th node is the boundary.
We can express $L_{i,n}=-\sum_{j=1}^{n-1} L_{i,j}$ and
$\Zv[1,2,3,4]/\Zv[1|2|3|4] =
\det[\L_{i,j}]^{j=1,\dots,n-1}_{i=1,\dots,n-1}$ using \tref{Lprincminors}.  Let
$\Gv_{i,j} = G_{i,j} + (z-1)G'_{i,j} + O((z-1)^2)$.  Then using \eqref{LG} to express $\L$ in terms of $\Gv$ and taking the limit
$z\to1$, some algebraic manipulation yields
\begin{equation}\label{eee}
\frac{Z[1,2|3,4]}{Z[1,2,3,4]}
= - G'_{1,2} - G'_{2,3} - G'_{3,1} +G_{1,2} - G_{1,3} .
\end{equation}
(\cref{L2G} gives a much easier way to convert an $L$-formula into a $G$-formula.)

While the left-hand sides of these formulas \eqref{Pu(12|34)} and
\eqref{eee} are symmetric in nodes $1$ and~$2$, the right-hand
sides are not.  This asymmetry is due to the location of the zipper,
and moving the zipper would change the values of the $L'_{i,j}$ and the
$G'_{i,j}$ (albeit in a predictable way).  If we keep the zipper
between nodes $1$ and~$3$, then we should expect a different formula
for $Z[1,3|2,4]$ than what we would get by permuting the indices
$1,2,3$ in the formula for $Z[1,2|3,4]$.  Indeed, if we carry out the
computations as above, we obtain
\begin{subequations}\label{eq:3-1-pairing}
\begin{align}
\frac{Z[1,2|3,4]}{Z[1|2|3|4]} &= - L'_{1,2} L_{3,4} -  L'_{2,3} L_{1,4} -  L'_{3,1} L_{2,4} + L_{1,2} L_{3,4} - L_{1,3} L_{2,4} \label{Z(12|34)}\\
\frac{Z[1,3|2,4]}{Z[1|2|3|4]} &= - L'_{1,2} L_{3,4} -  L'_{2,3} L_{1,4} -  L'_{3,1} L_{2,4} \\
\frac{Z[2,3|1,4]}{Z[1|2|3|4]} &= - L'_{1,2} L_{3,4} -  L'_{2,3} L_{1,4} -  L'_{3,1} L_{2,4} + L_{1,4} L_{2,3} - L_{1,3} L_{2,4}
\end{align}
\end{subequations}
and
\begin{subequations}
\begin{align}\label{Pc(12|34)}
\frac{Z[1,2|3,4]}{Z[1,2,3,4]} &= - G'_{1,2} - G'_{2,3} - G'_{3,1} + G_{1,2} - G_{1,3}\\
\frac{Z[1,3|2,4]}{Z[1,2,3,4]} &= - G'_{1,2} - G'_{2,3} - G'_{3,1} \label{Pc(13|24)} \\
\frac{Z[2,3|1,4]}{Z[1,2,3,4]} &= - G'_{1,2} - G'_{2,3} - G'_{3,1} + G_{2,3} - G_{1,3} \label{Pc(23|14)}
\end{align}
\end{subequations}

\subsection{Pair of pants with \texorpdfstring{$(2,0,0)$}{(2,0,0)} boundary nodes}
\label{sec:pants-2-0-0}

The following two annulus cases are most easily viewed as special
cases of the case when the surface $\Sigma$ is a pair of pants with $2$ nodes on one boundary
and no other nodes, see \fref{pants:2-0-0}. Put an $\SL$ bundle with
monodromies $A$ and $B$ around the two central holes $C_A$ and $C_B$, and
supported on zippers from the holes to the boundary between nodes~$1$
and~$2$.

\begin{figure}[b]
$\ffrac{\includegraphics[scale=0.19]{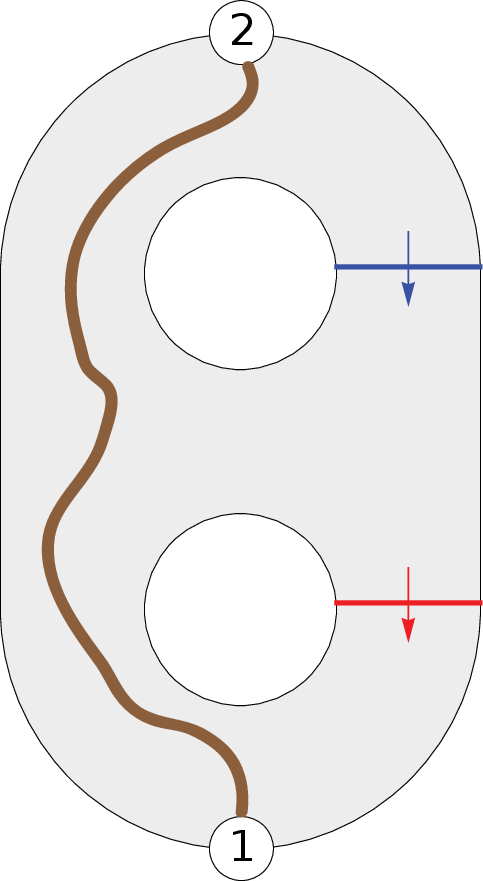}}{\mathclap{\ffrac{\displaystyle I\mathrlap{\phantom{A^{-1}}}}{\phantom{\scriptsize(\!AB\!)A(\!AB\!)^{-1}}}}}$ \hfill
$\ffrac{\includegraphics[scale=0.19]{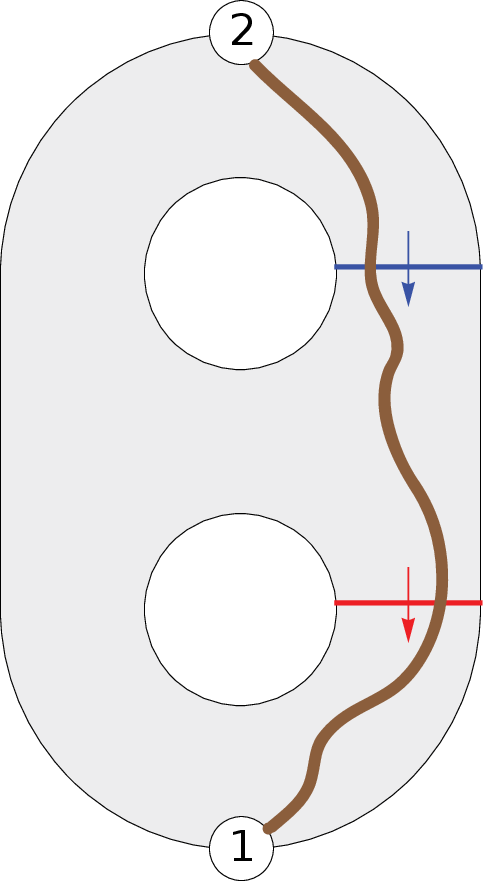}}{\mathclap{\ffrac{\displaystyle AB\mathrlap{\phantom{A^{-1}}}}{\phantom{\scriptsize(\!AB\!)A(\!AB\!)^{-1}}}}}$ \hfill
$\ffrac{\includegraphics[scale=0.19]{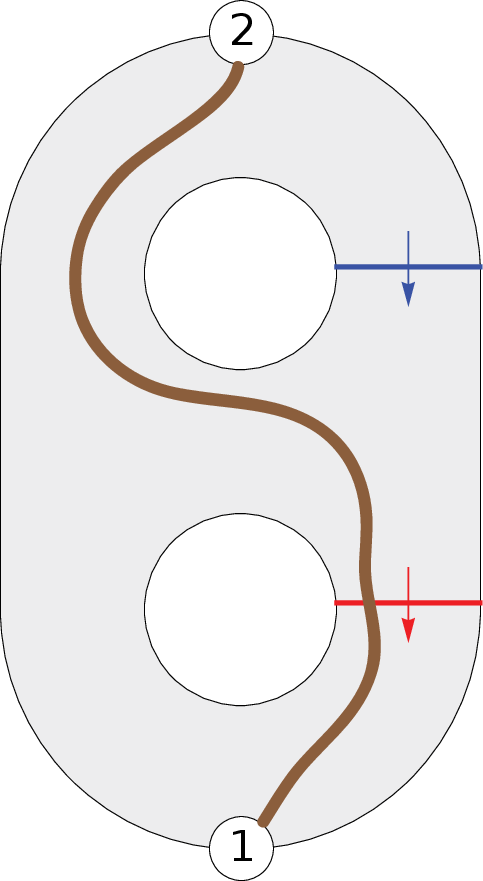}}{\mathclap{\ffrac{\displaystyle A\mathrlap{\phantom{A^{-1}}}}{\scriptsize(\!AB\!)^0\!A(\!AB\!)^{0}}}}$ \hfill
$\ffrac{\includegraphics[scale=0.19]{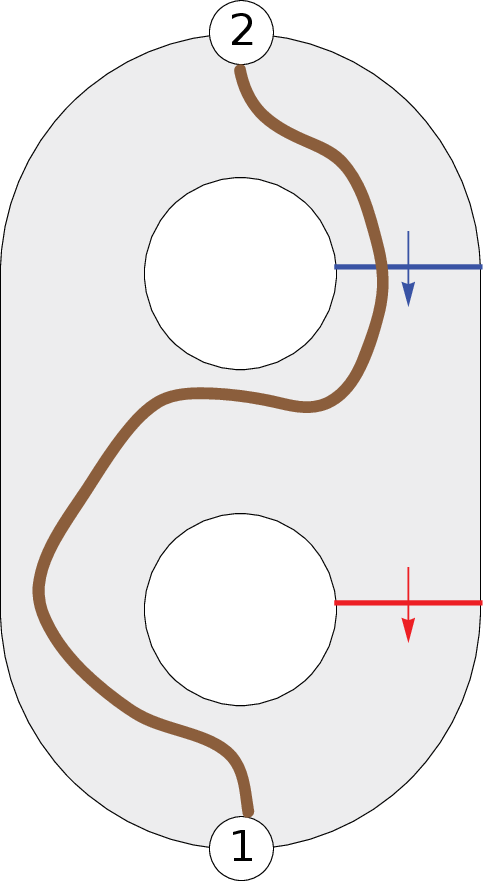}}{\mathclap{\ffrac{\displaystyle B\mathrlap{\phantom{A^{-1}}}}{\scriptsize(\!AB\!)^0B(\!AB\!)^{0}}}}$ \hfill
$\ffrac{\includegraphics[scale=0.19]{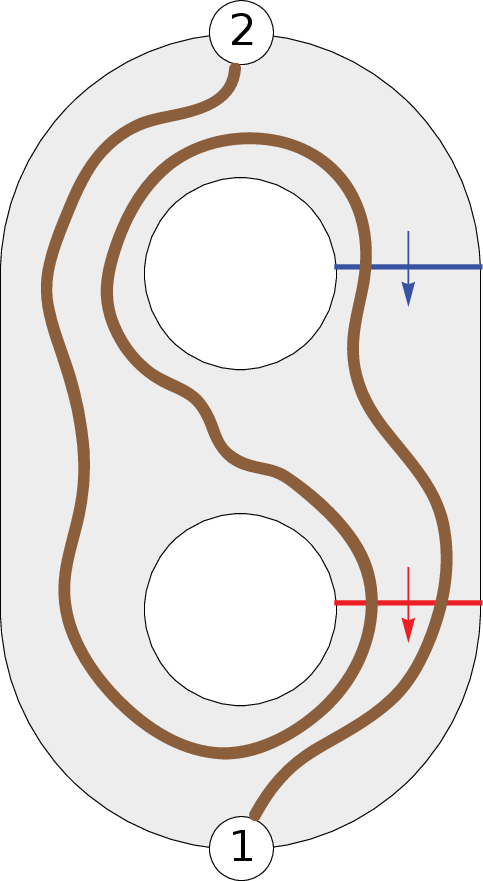}}{\mathclap{\ffrac{\displaystyle ABA^{-1}}{\scriptsize(\!AB\!)B(\!AB\!)^{-1}}}}$ \hfill
$\ffrac{\includegraphics[scale=0.19]{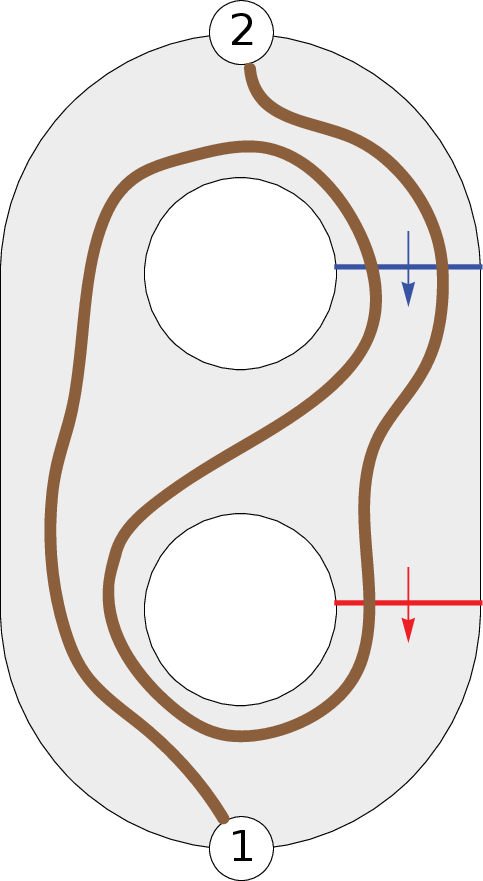}}{\mathclap{\ffrac{\displaystyle B^{-1}\!\!AB}{\scriptsize(\!AB\!)^{-1}\!A(\!AB\!)}}}$ \hfill
$\ffrac{\includegraphics[scale=0.19]{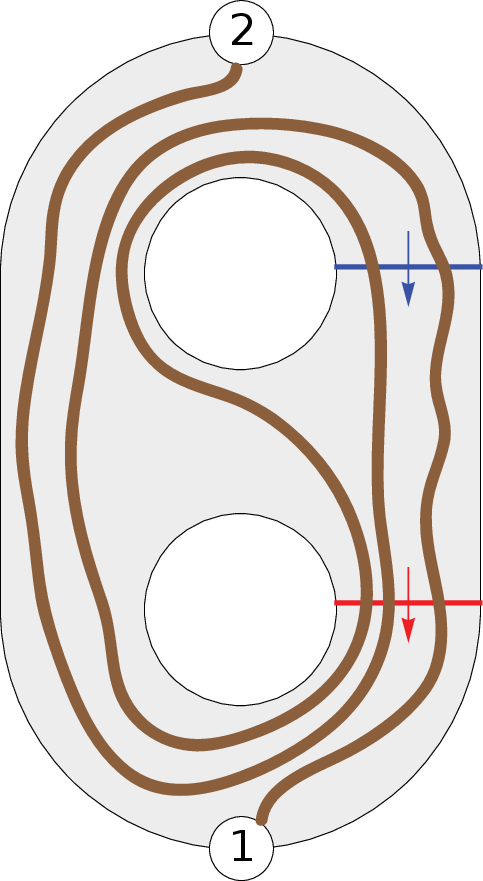}}{\mathclap{\ffrac{\displaystyle ABAB^{-1}\!\!A^{-1}}{\scriptsize(\!AB\!)A(\!AB\!)^{-1}}}}$
\caption{Some of the possible topological types for the path between nodes 1 and 2
  when the surface is a pair of pants with both nodes on one
  boundary.  The lower zipper (in red) has parallel transport $A$, and
  the upper zipper (in blue) has parallel transport~$B$.  For each diagram,
  the parallel transport of the path to~$1$ from~$2$ is shown.}
\label{pants:2-0-0}
\end{figure}

The parallel transport of a path to node $1$ from node $2$ is of the form
\begin{enumerate}
\item $I$, if the path has both holes on its right
\item $AB$, if the path has both holes on its left
\item $(AB)^{-k} A(AB)^{k}$ for some $k\in\Z$, if the path has the lower hole on its left and the upper hole on its right, and $k$ is the algebraic number of crossings that a dual path from the lower hole to the left boundary makes across the upper zipper
\item $(AB)^{-k} B(AB)^{k}$ for some $k\in\Z$, if the path has the lower hole on its right and the upper hole on its left, and $k$ is the algebraic number of crossings that a dual path from the upper hole to the left boundary makes across the lower zipper
\end{enumerate}
We let $c^{\text{(RR)}}$, $c^{\text{(LL)}}$, $c^{\text{(LR)}}_k$, and
$c^{\text{(RL)}}_k$ (for $k\in\Z$) denote the weighted sum of
cycle-rooted groves of the above types.  We further let
$c^{\text{(RR)}}_\ell$ and $c^{\text{(LL)}}_\ell$ denote the number of
cycle-rooted groves of type $c^{\text{(RR)}}$ and $c^{\text{(LL)}}$ in
which there are $\ell\in\N$ loops that surround both holes.

We need to choose matrices $A$ and $B$ for which $\det A=1$ and $\det
B=1$, and it is convenient to choose them so that $\Tr(A)=2$ and
$\Tr(B)=2$ (so that loops which surround one hole but not the other
have weight $0$), and so that $AB$ is diagonal.  We can take
\[
 A=\begin{bmatrix}\frac{2 x}{x+1} & y \\
 -\frac{(x-1)^2}{y (x+1)^2} & \frac{2}{x+1}\end{bmatrix}\qquad
 B=\begin{bmatrix} \frac{2 x}{x+1} & -\frac{y}{x} \\[3pt]
 \frac{(x-1)^2 x}{y (x+1)^2} & \frac{2}{x+1}\end{bmatrix}\qquad AB=\begin{bmatrix}x&0\\0&1/x\end{bmatrix}
\]
for variables $x$ and $y$.  Then
since $AB$ is diagonal, it is straightforward to evaluate
\begin{align*}
 (AB)^{-k} A(AB)^{k}&=\begin{bmatrix}\frac{2 x}{x+1} & y x^{-2 k} \\[3pt]
 -\frac{(x-1)^2 x^{2 k}}{y (x+1)^2} & \frac{2}{x+1}\end{bmatrix}\\
 (AB)^{-k} B(AB)^{k}&=\begin{bmatrix}\frac{2 x}{x+1} & -y x^{-2 k-1} \\[3pt]
 \frac{(x-1)^2 x^{1+2 k}}{y (x+1)^2} & \frac{2}{x+1}\end{bmatrix}.
\end{align*}

\begin{multline*}
\Zv[{}_1^2]=I \sum_{\ell\in\N} c^{\text{(RR)}}_\ell (2-x-1/x)^\ell + AB \sum_{\ell\in\N} c^{\text{(LL)}}_\ell (2-x-1/x)^\ell \\+
\sum_{k\in\Z} c^{\text{(LR)}}_k(AB)^{-k} A(AB)^{k}+\sum_{k\in\Z} c^{\text{(RL)}}_k(AB)^{-k} B(AB)^{k}.
\end{multline*}

Only the last two sums contribute to the $1,2$ entry of $\Zv[{}_1^2]$:
\[\Zv[{}_1^2]_{1,2}= y \sum_{k\in\Z}\big[c_k^{\text{(LR)}}x^{-2k}-c_k^{\text{(RL)}}x^{-2k-1}\big].\]
This is a Laurent series in $x$ from which one can extract the
coefficients $c_k^{\text{(LR)}}$ and $c_k^{\text{(RL)}}$.  Once these
are known, the coefficients $c_\ell^{\text{(RR)}}$ and
$c_\ell^{\text{(LL)}}$ can be extracted from $\Zv[{}_1^2]_{1,1}$ and $\Zv[{}_1^2]_{2,2}$.

\subsection{Annulus with \texorpdfstring{$(2,2)$}{(2,2)} boundary nodes}
\label{sec:annulus-2-2}

On the annulus with $4$ nodes, put nodes $1,2$ on the outer boundary and $3,4$ on the inner boundary.
Suppose we wish to compute the probability of the connections $13|24$
and $14|23$.
This computation can be used to compute the
probability that an edge $e$ is on the LERW from node~$1$ to node~$2$
(an equivalent calculation was done in \cite{Kenyon.54}).

Insert an extra edge $e_{34}$
from $3$ to $4$; this ``splits'' the inner boundary into two (see Figure \ref{annulus22}).
This is then a special case of the construction of
section \ref{sec:pants-2-0-0}.
In the notation of that section,
it suffices to use the limit $x\to-1$ and the value $x=1$ to distinguish crossings
$13|24$ and $14|23$: $[\Zv_{1,2}]_{1,2}$ in the limit $x\to-1$ gives the sum and in the case $x=1$ the difference
of the two desired quantities.
\begin{figure}[htbp]
\begin{center}
\includegraphics[scale=0.3]{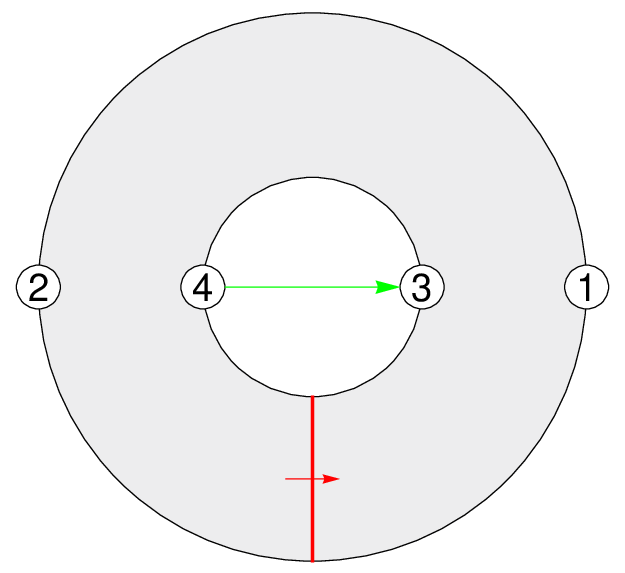}
\end{center}
\caption{Computing crossings $13|24$ and $14|23$.
\label{annulus22}}
\end{figure}

\subsection{Annulus with \texorpdfstring{$(4,0)$}{(4,0)} boundary nodes}
\label{sec:annulus-4-0}

When $4$ nodes are on the outer boundary and none on the inner (and
nodes $1,2,3,4$ are in counterclockwise order), the
case we have not yet discussed is the $14|23$ case: there are three
subcases depending on whether the paths from $1$ to $4$ and $2$ to $3$
go left or right of the inner boundary.

Again this is a special case of the $(2,0,0)$ case if we put an extra
edge between nodes $3$ and $4$.

In this case the only possible parallel transports to $1$ from $2$ are
(using the connection from that section)
\[AB,I,B,A,B^{-1}\!\!AB,\]
and only in the first two cases is there a possible extra loop surrounding
both $C_A$ and $C_B$.
Thus
\begin{multline*}
\Zv_{1,2}=I\big(c^{(\text{RR})}_0+c^{(\text{RR})}_1(2-x-1/x)\big)+AB\big(c_0^{(\text{LL})}+c_1^{(\text{LL})}(2-x-1/x)\big)\\
+B c_0^{(\text{RL})}
+A c_0^{(\text{LR})}+B^{-1}\!\!AB c_1^{(\text{LR})}.
\end{multline*}
The $1,2$ entry in $\Zv_{1,2}$ is
\[[\Zv_{1,2}]_{1,2}=-\frac{y}{x} c_0^{(\text{RL})}+y c_0^{(\text{LR})}+\frac{y}{x^2}c_1^{(\text{LR})}.\]
From this we can extract the three cases of interest.

\section{Annular-one surface graphs} \label{annular}
\label{annular1}

Suppose that the graph has $n$ nodes and is embedded on an annulus
such that nodes $1,\dots,n-1$
are on the inner
boundary of the annulus arranged in counterclockwise order, and node
$n$ is by itself on the outer boundary, and that the zipper is between
nodes $n-1$ and $1$ and directed in the counterclockwise direction
(from $n-1$ to $1$), as in section~\ref{sec:annulus-3-1}.  We call
these annular-one surface graphs; they are the next case after
circular planar graphs, and they play an important role in our
loop-erased random walk calculations in section~\ref{LERW}.
Annular-one surface graphs of course include the $(1,1)$ and $(3,1)$
cases that we did in the last section, but for expository purposes we
treated those special cases separately.  We are interested in
computing, for any partition $\sigma$ in which $n$ is not in a component
by itself, the weighted sum of groves of
type $\sigma$, which we denote $Z[\sigma]$.  We show how to compute
$Z[\sigma]/Z[1|2|\cdots|n]$ in terms of the response matrix $\L$, and
$Z[\sigma]/Z[1,2,\ldots,n]$ in terms of the Green's function $\Gv$.

\subsection{Reduction to partial pairings}
\label{annular1-partitions}

\begin{theorem}
  For a circular planar graph with $n$ nodes, for any partition
  $\sigma$ of the nodes, we can write $Z[\sigma]=\sum_m \alpha_m Z[\tau_m]$,
  where the $\tau_m$'s are partial pairings.
\end{theorem}
\begin{proof}
Let $i$ be the
smallest node label that is in a part of $\sigma$ of size more than~$2$,
and let $s$ be the size of this part.  We measure the
``complexity'' of partition $\sigma$ by $n (n-i) + s$.  Let $j$ be the
next-smallest item in $i$'s part of $\sigma$.  Let $\sigma^*$ denote
the partition obtained from $\sigma$ by ``de-listing''~$j$, i.e., by
regarding $j$ as an internal vertex which can occur in any of
$\sigma^*$'s parts.  If $\sigma$ has $k$ parts, then we can write
$Z[\sigma^*]=\sum_{\ell=1}^k Z[\sigma^*\text{\ with $j$ added to $\ell$th part}]$.
One of these terms is
 $Z[\sigma]$, so
\[ Z[\sigma] = Z[\sigma^*] - \sum_\ell Z[\sigma^*\text{\ with $j$ added to $\ell$th part}].\]
where the sum runs over all parts of $\sigma^*$ except the one
containing $i$.  Because the graph is circular planar, unless $j$ is
added to a part of $\sigma^*$ that is ``covered'' by the part
containing $i$, there will be no groves of that partition type.
Each nonzero term on the right has smaller complexity than
$\sigma$, so we can iterate this process to eventually express
$Z[\sigma]$ as a linear combination of $Z[\tau]$'s where $\tau$ is a partial pairing.
\end{proof}
For example,
\begin{align*}
  Z[1,5,8|2,3,4|6,7] &=   Z[1,8|2,3,4|6,7] - Z[1,8|2,3,4,5|6,7] - Z[1,8|2,3,4|5,6,7] \\
    &=   Z[1,8|2,4|6,7] - Z[1,8|2,5|6,7] - Z[1,8|2,4|5,7].
\end{align*}
There can be multiple such linear combinations for a given partition.

\begin{theorem} \label{combination-partial-pairing}
  For an annular-one surface graph with $n$ nodes, for any partition
  $\sigma$ of the nodes,
  we can write $Z[\sigma]=\sum_m \alpha_m Z[\tau_m]$,
  where the $\tau_m$'s are partial pairings.
\end{theorem}
\begin{proof}
If the part containing $n$ has size two, say that it is $\{h,n\}$,
then we list the remaining nodes in the order
$h+1,h+2,\dots,n-1,1,2,\dots,h-1$, and do the reductions described above
for circular planar graphs.  These will not increase the size of $n$'s part.
If the part containing $n$ has more than two nodes, then we first reduce its
size by internalizing nodes in the part other than $n$ until it has size two.
If $n$ started out in a singleton part, we
start out as in the circular planar case (with the order $1,\dots,n-1$)
until a node gets adjoined to $n$'s part.
\end{proof}

A similar reduction can be done
for the annulus with 2 nodes on each boundary,
but not for the annulus with 2 nodes on one boundary and 3 on the other.

\subsection{Partial pairings in terms of the response matrix}\label{annular1-complete-pairing}

Recall the computation $Z[1,2|3,4]/Z[1|2|3|4]$ in
\sref{sec:annulus-3-1}.  \tref{LABC} provided a family of equations
for subdeterminants of the response matrix in terms of grove partition
functions.  We solved these equations for the grove partition
functions in terms of the subdeterminants, and took the limit $z\to1$
to express $Z[1,2|3,4]/Z[1|2|3|4]$ in terms of $L_{u,v}$ and
$L'_{u,v}$.  We follow the same approach here.
For ease of exposition we focus on complete pairings.  Partial
pairings are handled in the same way, except that the subdeterminants
have extra rows and columns and minus signs corresponding to the internalized nodes
(recall \tref{LABC}.)
The singleton nodes have no effect on the determinant formulas, except insofar as
they affect the values of the $L_{u,v}$ and $L'_{u,v}$ variables.

For complete pairings,
there are $n-1$ ways to connect the two boundaries (ignoring windings).
When the annulus is cut along this connection, the domain becomes
planar, so there are $C_{n/2-1}$ ways to pair up the remaining nodes,
where $C_k$ is the $k$th Catalan number.
We have
\[ (n-1) C_{n/2-1} = (n-1) \frac{(n-2)!}{(n/2-1)!(n/2)!}
   = \frac12\frac{n!}{(n/2)!(n/2)!} = \frac12\binom{n}{n/2}.\]
So the number of annular pairings equals the number of equations
arising from the determinant formula.  In fact, there is a natural
bijection between the $\L$-determinants and the annular pairings which
is based on the cycle lemma of Dvoretzky and Motzkin \cite{DM}, which
we use in \aref{inverse} to show that these equations are linearly
independent for any even~$n$.

In any directed pairing, the
connection between node $n$ and the other boundary determines whether
or not and in what direction that any other directed pair crosses the
zipper.  Reversing the direction of any directed pair (other than the
pair containing $n$) that crosses the zipper introduces a factor of
$z^2$.  Since only even powers of $z$ appear, it is convenient to
change variables to $\zeta=z^2.$ For example,
$\Zv[{}_3^5|{}_1^2|{}_4^6] = \zeta \Zv[{}_5^3|{}_1^2|{}_4^6]$.

For more compact notation let $\ddddot\Zv_\sigma:=\Zv_\sigma/\Zv[1|2|\dots|n]$.
When expanding an $\L$-determinant into a signed sum of
$\ddddot\Zv_\sigma$'s, where $\sigma$ is a directed pairing, $\ddddot\Zv_\sigma$ can be
put into a canonical form $\zeta^{\text{power}} \ddddot\Zv_{\sigma'}$, where
$\sigma'$ is a directed pairing in which the pairs are directed counterclockwise
around the annulus.  Our goal is to solve for
$\ddddot\Zv_\tau=\ddddot\Zv_\tau(\zeta)$ in terms of the
$\L$-determinants and $\zeta$, and take the limit $\zeta\to 1$.

The system of linear equations can be represented by a matrix $\A_n$.
When recording the linear equation corresponding to $\det\L_R^S$, we
can re-order $R$ and $S$ in any manner, and this would just scale row
$\det\L_R^S$ of $\A_n$ by $\pm1$, which has no effect on our ability
to solve for the $\ddddot\Zv[\sigma]$s.  But the signs in $\A_n^{-1}$
are surprisingly nice when we order $R$ and $S$ in a manner that
corresponds to $\det\L_R^S$'s associated pairing in the aforementioned
bijection.  This canonical ordering is described in \aref{pairing-set}.

$\A_2$ is just the $1\times 1$ matrix whose entry is $1$
(since $\Zv[{}_1^2]/\Zv[1|2] = \L_{1,2} = \det \L_1^2$):

\[\begin{gathered}\phantom{\rf{[{}_1^2]}}\\
\A_2\quad=\quad
\bordermatrix[{[]}]{
& \srf{\ddddot{\Zv}[{}_1^2]}\\
\det\L_{\;1}^{2}& 1\\[2pt]
}
\end{gathered}\]
\enlargethispage{12pt}
The matrix $\A_4$ encodes the system of equations~\eqref{sys:3,1} we saw for the $(3,1)$ case:
\[\begin{gathered}\phantom{\rf{[{}_1^2|{}_3^4]}}\\
\A_4\quad=\quad
\bordermatrix[{[]}]{
& \srf{\ddddot{\Zv}[{}_1^2|{}_3^4]}& \srf{\ddddot{\Zv}[{}_3^1|{}_2^4]}& \srf{\ddddot{\Zv}[{}_2^3|{}_1^4]}\\
\det\L_{\;1, 3}^{2, 4}& 1& 0& -1\\[2pt]
\det\L_{\;3, 2}^{1, 4}& -1& 1& 0\\[2pt]
\det\L_{\;2, 1}^{3, 4}& 0& -\zeta & 1\\[2pt]
}
\end{gathered}\]
The first two rows of the next matrix $\A_6$ are
\[\begin{gathered}\phantom{\rf{[{}_1^2|{}_3^4|{}_5^6]}}\\
\bordermatrix[{[]}]{
& \srf{\ddddot{\Zv}[{}_1^2|{}_3^4|{}_5^6]}& \srf{\ddddot{\Zv}[{}_1^4|{}_2^3|{}_5^6]}& \srf{\ddddot{\Zv}[{}_5^1|{}_2^3|{}_4^6]}& \srf{\ddddot{\Zv}[{}_5^3|{}_1^2|{}_4^6]}& \srf{\ddddot{\Zv}[{}_4^5|{}_1^2|{}_3^6]}& \srf{\ddddot{\Zv}[{}_4^2|{}_5^1|{}_3^6]}& \srf{\ddddot{\Zv}[{}_3^4|{}_5^1|{}_2^6]}& \srf{\ddddot{\Zv}[{}_3^1|{}_4^5|{}_2^6]}& \srf{\ddddot{\Zv}[{}_2^3|{}_4^5|{}_1^6]}& \srf{\ddddot{\Zv}[{}_2^5|{}_3^4|{}_1^6]}\\
\det\L_{\;1, 3, 5}^{2, 4, 6}& 1& -1& 0& 0& -1& 0& 0& 0& 1& -1\\[2pt]
\det\L_{\;1, 2, 5}^{4, 3, 6}& 0& 1& 0& 0& 0& 0& 0& \zeta & -1& 0\\[2pt]
}
\end{gathered}\]
\old{
\[\begin{gathered}\phantom{\rf{[{}_1^2|{}_3^4|{}_5^6]}}\\
\bordermatrix[{[]}]{
& \srf{\ddddot{\Zv}[{}_1^2|{}_3^4|{}_5^6]}& \srf{\ddddot{\Zv}[{}_1^4|{}_2^3|{}_5^6]}& \srf{\ddddot{\Zv}[{}_5^1|{}_2^3|{}_4^6]}& \srf{\ddddot{\Zv}[{}_5^3|{}_1^2|{}_4^6]}& \srf{\ddddot{\Zv}[{}_4^5|{}_1^2|{}_3^6]}& \srf{\ddddot{\Zv}[{}_4^2|{}_5^1|{}_3^6]}& \srf{\ddddot{\Zv}[{}_3^4|{}_5^1|{}_2^6]}& \srf{\ddddot{\Zv}[{}_3^1|{}_4^5|{}_2^6]}& \srf{\ddddot{\Zv}[{}_2^3|{}_4^5|{}_1^6]}& \srf{\ddddot{\Zv}[{}_2^5|{}_3^4|{}_1^6]}\\
\det\L_{\;1, 3, 5}^{2, 4, 6}& 1& -1& 0& 0& -1& 0& 0& 0& 1& -1\\[2pt]
\det\L_{\;1, 2, 5}^{4, 3, 6}& 0& 1& 0& 0& 0& 0& 0& \zeta & -1& 0\\[2pt]
\det\L_{\;5, 2, 4}^{1, 3, 6}& 1& -1& 1& -1& 0& 0& -1& 0& 0& 0\\[2pt]
\det\L_{\;5, 1, 4}^{3, 2, 6}& -1& 0& 0& 1& 0& 0& 0& 0& 0& 1\\[2pt]
\det\L_{\;4, 1, 3}^{5, 2, 6}& 0& 0& \zeta & -\zeta & 1& -\zeta & 0& 0& -1& 0\\[2pt]
\det\L_{\;4, 5, 3}^{2, 1, 6}& 0& 1& -1& 0& 0& 1& 0& 0& 0& 0\\[2pt]
\det\L_{\;3, 5, 2}^{4, 1, 6}& -1& 0& 0& 0& 1& -\zeta & 1& -1& 0& 0\\[2pt]
\det\L_{\;3, 4, 2}^{1, 5, 6}& 0& 0& 0& \zeta & -1& 0& 0& 1& 0& 0\\[2pt]
\det\L_{\;2, 4, 1}^{3, 5, 6}& 0& 0& -\zeta & 0& 0& 0& \zeta & -\zeta & 1& -1\\[2pt]
\det\L_{\;2, 3, 1}^{5, 4, 6}& 0& 0& 0& 0& 0& \zeta ^2& -\zeta & 0& 0& 1\\[2pt]
}
\end{gathered}\]
}

Each $\det\L_R^S$ is a signed-sum of $(n/2)!$ of the $\ddddot{\Zv}[\sigma]$'s, but not all of these $\sigma$'s can be embedded in the annulus, so the rows generally have fewer than $(n/2)!$ nonzero entries.  For each pairing $\sigma$ that embeds in the annulus, the column $\ddddot{\Zv}[\sigma]$ contains $2^{n/2-1}$ nonzero entries: for each pair $\{i,j\}$ in $\sigma$, except the pair containing $n$, either $i\in R$ and $j\in S$ or else $j\in R$ and $i\in S$.

The inverses $\A_n^{-1}$ of these matrices are
\[\begin{gathered}\phantom{\rf{\det\L_{\;1}^{2}}}\\
\A_2^{-1}\quad=\quad
\bordermatrix[{[]}]{
& \srf{\det\L_{\;1}^{2}}\\
\ddddot{\Zv}[{}_1^2]& 1\\[2pt]
}
\end{gathered}\]
\[\begin{gathered}\phantom{\rf{\det\L_{\;1, 3}^{2, 4}}}\\
\A_4^{-1}\quad=\quad
\bordermatrix[{[]}]{
& \srf{\det\L_{\;1, 3}^{2, 4}}& \srf{\det\L_{\;3, 2}^{1, 4}}& \srf{\det\L_{\;2, 1}^{3, 4}}\\
\ddddot{\Zv}[{}_1^2|{}_3^4]& 1& \zeta & 1\\[2pt]
\ddddot{\Zv}[{}_3^1|{}_2^4]& 1& 1& 1\\[2pt]
\ddddot{\Zv}[{}_2^3|{}_1^4]& \zeta & \zeta & 1\\[2pt]
}
\times\frac{1}{(1-\zeta)^{1}}\end{gathered}\]
and the first two rows of $\A_6^{-1}$ are
\[\begin{gathered}\phantom{\rf{\det\L_{\;1, 3, 5}^{2, 4, 6}}}\\
\bordermatrix[{[]}]{
& \srf{\det\L_{\;1, 3, 5}^{2, 4, 6}}& \srf{\det\L_{\;1, 2, 5}^{4, 3, 6}}& \srf{\det\L_{\;5, 2, 4}^{1, 3, 6}}& \srf{\det\L_{\;5, 1, 4}^{3, 2, 6}}& \srf{\det\L_{\;4, 1, 3}^{5, 2, 6}}& \srf{\det\L_{\;4, 5, 3}^{2, 1, 6}}& \srf{\det\L_{\;3, 5, 2}^{4, 1, 6}}& \srf{\det\L_{\;3, 4, 2}^{1, 5, 6}}& \srf{\det\L_{\;2, 4, 1}^{3, 5, 6}}& \srf{\det\L_{\;2, 3, 1}^{5, 4, 6}}\\
\ddddot{\Zv}[{}_1^2|{}_3^4|{}_5^6]& \zeta \!+\!1& \zeta \!+\!1& \zeta ^2\!+\!\zeta & 2 \zeta & \zeta \!+\!1& \zeta ^2\!+\!\zeta & 2 \zeta & 2 \zeta & \zeta \!+\!1& 2\\[2pt]
\ddddot{\Zv}[{}_1^4|{}_2^3|{}_5^6]& \zeta & 1& \zeta & \zeta & \zeta & \zeta ^2& \zeta & \zeta & 1& 1\\[2pt]
}
\times\frac{1}{(1-\zeta)^{2}}\end{gathered}\]

Notice that in the above examples, each entry of the inverse matrix
$\A_n^{-1}$ is a polynomial in $\zeta$ with non-negative integer
coefficients and degree at most $n/2-1$, divided by
$(1-\zeta)^{n/2-1}$.  Notice also that these polynomials, when evaluated at $\zeta=1$,
depend only on the row, not upon the column.  In fact, these observations hold for general $n$,
and are a consequence of an explicit combinatorial expression for the inverse
annular matrix $\A_n^{-1}$ that we provide in Appendix~\ref{inverse}.

For example,
row $\ddddot{\Zv}[{}_1^4|{}_2^3|{}_5^6]$ of $\A_6^{-1}$ (the second row given above) tells us
\begin{equation}
\frac{\Zv[{}_1^4|{}_2^3|{}_5^6]}{\Zv[1|2|3|4|5|6]} =
\frac{1}{(1-\zeta)^2}\left\{\begin{gathered}
\zeta \det\L_{\;1, 3, 5}^{2, 4, 6}
+ \det\L_{\;1, 2, 5}^{4, 3, 6}
+\zeta \det\L_{\;5, 2, 4}^{1, 3, 6}
+ \zeta \det\L_{\;5, 1, 4}^{3, 2, 6}\\
+ \zeta\det\L_{\;4, 1, 3}^{5, 2, 6}
+ \zeta^2\det\L_{\;4, 5, 3}^{2, 1, 6}
+ \zeta\det\L_{\;3, 5, 2}^{4, 1, 6}\\
+ \zeta \det\L_{\;3, 4, 2}^{1, 5, 6}
+ \det\L_{\;2, 4, 1}^{3, 5, 6}
+ \det\L_{\;2, 3, 1}^{5, 4, 6}
\end{gathered}\right\}.
\end{equation}

We can expand out the $\L$-determinants into sums of products of the
$\L_{i,j}$, each of which depends on $\zeta$, where
$\L_{j,i}(\zeta)=\L_{i,j}(1/\zeta)$.  Since the denominator is
$(1-\zeta)^{n/2-1}$, to evaluate the limit $\zeta\to 1$, we can
differentiate the numerator and denominator $n/2-1$ times with respect
to $\zeta$ and then set $\zeta$ to $1$.
The denominator of course becomes $(-1)^{n/2-1}(n/2-1)!$.
The numerator will consist of monomials of degree $n/2$ in
the quantities
\[
 \left.\L_{i,j}\right|_{\zeta=1},\quad
 \left.\frac{\partial}{\partial\zeta}\L_{i,j}\right|_{\zeta=1},\quad
 \left.\frac{\partial^2}{\partial\zeta^2}\L_{i,j}\right|_{\zeta=1},\quad
 \dots,\quad
 \left.\frac{\partial^{n/2-1}}{\partial\zeta^{n/2-1}}\L_{i,j}\right|_{\zeta=1}.
\]
Surprisingly, in each case all the terms involving higher order derivatives
of $\L_{i,j}$ cancel upon setting $\zeta$ to~$1$.  The terms involving
the first derivative of $\L_{i,n}$ also cancel at $\zeta=1$.
We prove this cancellation in Appendix~\ref{no-higher-derivative}.
This cancellation is convenient,
since there are fewer quantities that we need to evaluate.
Recall that
\[L_{i,j}=\left.\L_{i,j}\right|_{z=1}=\left.\L_{i,j}\right|_{\zeta=1}.\]
Let us define
\[L'_{i,j}:=\left.\frac{\partial}{\partial z}\L_{i,j}\right|_{z=1}=
  2 \left.\frac{\partial}{\partial \zeta}\L_{i,j}\right|_{\zeta=1}.\]

\begin{theorem} \label{L''}
  For all positive even $n$ and pairings $\sigma$ of $\{1,\dots,n\}$,
  $Z_\sigma/Z_{1|2|\cdots|n}$ is a polynomial of
  degree $n/2$ in the quantities
  \[
  \{ L_{i,j}  : 1\leq i < j \leq n \}\quad\text{and}\quad
  \{ L'_{i,j} : 1\leq i < j \leq n-1 \}.
  \]
\end{theorem}
See \eqref{eq:3-1-pairing} for an example.
We prove this theorem in the appendix.
\begin{conjecture}
  The coefficients in the polynomials in \tref{L''} are all integers.
  (We have verified this for all $\sigma$ for all $n\leq 10$.)
\end{conjecture}

\subsection{Formulas using the Green's function}

\begin{theorem} \label{partial-pairing-greens}
  Let $\sigma$ be a partial pairing, in which the nodes $Q$ are in singleton parts, the nodes $T$ are internalized, and $n\notin Q\cup T$.
  In the above formulas expressing $\Zv[\sigma]/\Zv[1|2|\cdots|n]$ in terms of $\L$-determinants,
  we can replace each $(-1)^{|T|} \det\L_{R,T}^{S,T}$ with $\det\hat\Gv_{R,Q}^{S,Q}$, where $\hat\Gv_{i,j}=\Gv_{i,j}$ for $j<n$ and $\hat\Gv_{i,n}=1$,
  and the result will be a formula for $Z[\sigma]/Z[1,2,\dots,n]$.
\end{theorem}

\begin{proof}
  Since the
higher derivatives of $L_{i,j}$ and the first derivative of
$L_{i,n}$ always cancel out, we may compute $\ddddot Z[\sigma]$
using any convenient choice of $L'_{i,n}$, and for present purposes it
is convenient to make the choice for which $\sum_j \L_{i,j} = 0$ for
each~$i$.  Then writing the sequence $S$ as $S=S^*,n$, where $n\notin S^*$,
we may express the determinant $\det\L_{R,T}^{S,T}$ as
\[ \det\L_{R,T}^{S,T} = -\sum_{i\notin S\cup T} \det\L_{R,T}^{S^*,i,T}. \]
Suppose for now that both $R,Q,T$ and $S^*,Q,T$ are in sorted order.
We use Jacobi's formula on each summand
and the fact that $\L_{1,\dots,n-1}^{1,\dots,n-1}$ and $\Gv_{1,\dots,n-1}^{1,\dots,n-1}$ are negative inverses
to obtain
\enlargethispage{12pt}
\begin{align*}
 \frac{\det\L_{R,T}^{S,T}}{\det\L_{1,\ldots,n-1}^{1,\ldots,n-1}}
  &= - \sum_{i\notin S\cup T} (-1)^{\sum R + \sum S^* + i + |\{s\in S^*:s>i\}|} \det\big(-\Gv_{\{1,\dots,n-1\}\setminus\{S^*,i,T\}}^{\{1,\dots,n-1\}\setminus R,T}\big).
\intertext{Since $\{1,\dots,n-1\}\setminus (R,T) = S^*,Q$ and $\{1,\dots,n-1\}\setminus \{S^*,i,T\} = R,Q\setminus\{i\}$,
}
 \frac{\det\L_{R,T}^{S,T}}{\det\L_{1,\ldots,n-1}^{1,\ldots,n-1}}
  &=  (-1)^{\sum R + \sum S^* + |R|+|Q|} \sum_{i\in R\cup Q} (-1)^{i + |\{s\in S^*:s>i\}|} \det \Gv_{R,Q\setminus\{i\}}^{S^*,Q}.
\end{align*}
When we expand $\det\hat\Gv_{R,Q}^{S^*,n,Q}$ along column $n$, we obtain
\begin{align*}
  \det\hat\Gv_{R,Q}^{S^*,n,Q} &= \sum_{j=1}^{|R\cup Q|} (-1)^{j+|R|}\det\Gv_{R,Q\text{\ with $j$th item removed}}^{S^*,Q}
\end{align*}
If the $j$th item of $R,Q$ is $i$, then $i-j=|\{s\in S^*:s<i\}|$, so
\begin{align*}
  \det\hat\Gv_{R,Q}^{S^*,n,Q} &= (-1)^{|R|+|S^*|} \sum_{i\in R\cup Q} (-1)^{i+|\{s\in S^*:s>i\}|}\det\Gv_{R,Q\setminus\{i\}}^{S^*,Q}
\intertext{and so}
 \frac{\det\L_{R,T}^{S,T}}{\det\L_{1,\ldots,n-1}^{1,\ldots,n-1}}
&= (-1)^{\sum R + \sum S^* + |R|+|Q| + |R| + |S^*|} \det\hat\Gv_{R,Q}^{S^*,n,Q}.
\end{align*}
Since $R\cup S^*=\{1,2,\dots,2|R|-1\}$, which adds up to $|R|$ modulo 2, we have
\begin{equation}
 \frac{(-1)^{|T|} \det\L_{R,T}^{S,T}}{\det-\L_{1,\ldots,n-1}^{1,\ldots,n-1}}
= \det\hat\Gv_{R,Q}^{S,Q}.
\end{equation}
Observe that if we relax the assumption that $R,Q,T$, and $S^*,Q,T$ are in sorted order, the
left- and right-hand sides of the above equation change signs the same number of times.
Hence this equation holds regardless of the relative order of
the indices in $R$, $S$, $Q$, and $T$.

Finally, recall that $\frac{\Zv[1,2,\dots,n]}{\Zv[1|2|\dots|n]}=\det-\L_{1,\ldots,n-1}^{1,\ldots,n-1}.$
\end{proof}

\begin{corollary} \label{L2G}
  For a complete pairing $\sigma$, the Green's function formula for
  $Z[\sigma]/Z[1,2,\dots,n]$ can be obtained from the response-matrix formula for
  $Z[\sigma]/Z[1|2|\cdots|n]$ simply by replacing each $L_{i,j}$ with
  $G_{i,j}$ and each $L'_{i,j}$ with $G'_{i,j}$, and then setting $G_{i,n}=1$.
\end{corollary}

\begin{corollary} \label{global-const}
  For a partial pairing $\sigma$ in which node $n$ is paired, the
  Green's function formulas for $Z[\sigma]/Z[1,2,\dots,n]$ are invariant under the
  addition of global constant to the Green's function.
\end{corollary}
\begin{proof}
  The column indexed by $n$ in $\det\hat\Gv_{R,Q}^{S,Q}$ is all-ones.
\end{proof}

\subsection{Windings}

We can also extract information about the windings of the paths within
a grove pairing in a manner similar to that described in the $(1,1)$
case.  For a given directed pairing $\sigma$, we have
\[\Zv[\sigma] = \sum_k z^k Z[\sigma,(k)],\]
where $Z[\sigma,(k)]$ is the weighted sum of groves of type $\sigma$ in which
the algebraic number of zipper crossings (involving all pairs in $\sigma$) is $k$.  Then
  \[ \E\left[\parbox{\widthof{algebraic number of zipper crossings}}{algebraic number of zipper crossings\\ for groves of type $\sigma$}\right]
  = \lim_{z\to 1}\frac{\partial}{\partial z}\log\frac{\Zv[\sigma]}{\Zv[1|2|\cdots|n]}.\]

\section{The Green's function and its monodromy-derivative}
\label{Greens}

To carry out our loop-erased random walk computations for various
lattices, we will use our formulas for the connection probabilities in
annular-one graphs developed in section~\ref{annular}, and for this we
need the Green's function $G$ together with its derivative $G'$ with
respect to a zipper monodromy.  We will need $G$ and $G'$ for both
the full lattice, and the lattice after some of its edges have been cut.

\subsection{Green's function and potential kernel} \label{sec:Gbar}

The Green's function $G_{u,v}$ is infinite for recurrent lattices such
as $\Z^2$, but there is a quantity known as the potential kernel
$A_{u,v}$ which behaves like a Green's function, except that
$A_{u,u}=0$, and $G_{u,v}$ and $A_{u,v}$ have the opposite sign
convention (see \cite{Spitzer}).  Suppose that a graph~$\G$ is the
intersection of $\Z^2$ or another lattice $\LL$ with a region
surrounding the origin, with ``wired boundary conditions'', i.e., all
the lattice vertices in $\LL$ that are not in the region are merged
into a single vertex in $\G$ that plays the role of boundary.  If $R$
denotes the electrical resistance within $\G$ from the origin to the
boundary, then
\[ G^{\G}_{u,v} = R - A^{\LL}_{u,v} + o(1),\]
where the error term tends to $0$ for fixed $u$ and $v$ as
$R\to\infty$.  For translation invariant lattices, $A^{\LL}_{u,v}$
depends only on $u-v$, and is written as $a^{\LL}_{u-v}$.

Since all of our formulas for crossings of the annulus are invariant
when a global constant is added to the Green's function (involving
terms such as $G_{1,2}-G_{1,3}$), it is straightforward to take the
limit $\lim_{\G\to\LL}$ of these formulas by replacing each
$G^\G_{u,v}$ in the formula with $-A^{\LL}_{u,v}$, which we shall also denote
by $\bar G^\LL_{u,v}$.

For convenience let us work with a modified finite graph $\bar\G$ approximating the
lattice $\LL$, obtained
by adjoining an edge with conductance $-1/R$ to node $n$
of $\G$, defining node $n$ of $\bar\G$ to be the other endpoint of
this edge.  This has the effect of making the resistance in $\bar\G$ from the origin
to node $n$ exactly zero.

For any partition $\sigma$ for which $n$ is not in a singleton
part, we have $Z^{\bar\G}=-Z/R$ and $Z^{\bar\G}[\sigma]=-Z[\sigma]/R$,
so in particular we can compute $Z[\sigma]/Z = Z^{\bar\G}[\sigma]/Z^{\bar\G}$
by working with the Green's function $\Gv^{\bar\G}$ of this modified graph.
We define $\bar G_{u,v}$ and $\bar G'_{u,v}$ by the expansion
\[\Gv^{\bar\G}_{u,v} = \bar G_{u,v} + (z-1) \bar G'_{u,v} + O((z-1)^2).\]
Then in the limit $\G\to\LL$, we have $\bar G_{u,v} \to -A^{\LL}_{u,v}$.

It is well-known how to compute the potential kernel on periodic
lattices by taking the Fourier coefficients of the characteristic
polynomial of the lattice \cite{Spitzer}.  The potential kernel can
also be computed for any ``isoradial'' graph by doing local
computations \cite{Kenyon.isoradial}.  The square, triangular, and
honeycomb lattices are both periodic and isoradial, so for these
lattices either method can be employed.
For the squre lattice the potential kernel takes values in $\Q+\frac1{\pi}\Q$,
while for the triangular and honeycomb lattices it takes values
in $\Q+\frac{\sqrt{3}}{\pi}\Q$.

We shall make use of the following smoothness result:
\begin{lemma}[\cite{Stohr}]\label{Stohrthm}
For points $z=(z_1,z_2)$ far from $(0,0)$, the potential kernel on $\Z^2$
behaves like
\[ A^{\Z^2}_{0,z} = \frac{1}{2\pi} \log|z| + \frac{\frac32\log 2 + \gamma}{2\pi} + O(1/|z|^2),\]
where $\gamma$ is Euler's constant.
\end{lemma}
The asymptotics of the Green's function has also been studied on
other vertex-transitive 2-dimensional periodic lattices
\cite{FU96} \cite{KS04}, and also on other isoradial
graphs \cite[Thm.~7.3]{Kenyon.isoradial} \cite[Thm.~A.2]{Bucking}.
In particular, for the triangular
lattice the potential kernel is asymptotically
$(\log|z|+\log\sqrt{12}+\gamma)/(2\pi\sqrt{3}) + O(1/|z|^2)$, and for
the honeycomb lattice it is $(\log|z|+\log 2+\gamma) \sqrt{3}/(2\pi) +
O(1/|z|^2)$.

\subsection{Derivative of the Green's function}
\label{Greens'}

\subsubsection{Infinite sum formula}

Let $S$ be the adjacency matrix of the zipper, i.e.,
\[S_{k,\ell}=\begin{cases} 1 &
\text{there is a zipper edge directed from $k$ to $\ell$} \\
0 & \text{otherwise}.\end{cases}\]
Then $\Delta(z)=\Delta_0 + (1-z^{-1})S + (1-z) S^*$, so
\begin{align}
\Delta(z)^{-1}
&= \left(\Delta_0(1 + (1-z^{-1})\Delta_0^{-1} S+(1-z)\Delta_0^{-1}S^*)\right)^{-1}\nonumber\\
&= \Delta_0^{-1} -(1-z^{-1}) \Delta_0^{-1} S \Delta_0^{-1} -(1-z) \Delta_0^{-1} S^* \Delta_0^{-1} + O((z-1)^2)\nonumber \\
\Gv_{u,v} &= G_{u,v} - (z-1) \sum_{\text{zipper edges }(k,\ell)} c_{k,\ell} (G_{u,k} G_{\ell,v}-G_{u,\ell} G_{k,v})+ O((z-1)^2)\nonumber
\end{align}
The sum is over zipper edges $(k,\ell)$ in which the zipper direction is from
$k$ to $\ell$, and $c_{k,\ell}$ is the conductance of edge $(k,\ell)$.
The linear term in $z-1$ gives us the desired derivative:
\begin{align}
G'_{u,v} = \left.\partial_z \Gv_{u,v}\right|_{z=1}&=  - \sum_{\text{zipper edges }(k,\ell)} c_{k,\ell} (G_{u,k} G_{\ell,v}-G_{u,\ell} G_{k,v}).
\end{align}
For the modified graph $\bar\G$, we of course have
\begin{align}
\bar G'_{u,v} = \left.\partial_z \Gv^{\bar\G}_{u,v}\right|_{z=1}&=  - \sum_{\text{zipper edges }(k,\ell)} c_{k,\ell} (\bar G_{u,k} \bar G_{\ell,v} - \bar G_{u,\ell} \bar G_{k,v}).
\label{ebar}
\end{align}

For a vertical zipper in $\Z^2$ (or the triangular lattice or
honeycomb lattice) started in the face whose lower-left corner is the
origin, directed down towards infinity, we define
\begin{equation} \label{lattice-sum}
  \bar G'{}^{\LL}_{u,v} = - \sum_{\text{zipper edges }(k,\ell)} c_{k,\ell} (\bar G^{\LL}_{u,k} \bar G^{\LL}_{\ell,v} - \bar G^{\LL}_{u,\ell} \bar G^{\LL}_{k,v}).
\end{equation}
For fixed $u$ and $v$, for zipper edges $(k,\ell)$ at a distance $r$
from the origin, it is straightforward to use the smoothness result in
\lref{Stohrthm} to show that edge $(k,\ell)$ contributes $O(r^{-2}\log r)$
to the sum~\ref{lattice-sum}, so this sum is absolutely convergent.

We would like to know that $\bar G'{}^{\G}_{u,v}$ converges to $\bar
G'{}^{\LL}_{u,v}$ as defined in \eqref{lattice-sum} for a sequence of
$\G$'s converging to $\LL$.  For our purposes in section~\ref{LERW}
when we analyze loop-erased random on the lattice, we do not need this
convergence of $\bar G'$ for every sequence of $\G$'s converging to
$\LL$, it will suffice to have convergence for some sequence of $\G$'s
tending to $\LL$.  Perhaps the easiest way to show this is to exploit
the reflection symmetry that each of the square, triangular, and
honeycomb lattices possess.

\begin{lemma}\label{sum}
  If $\LL$ is the square, triangular, or honeycomb lattice, and
  $L\in\N$, let $\G_L=[-L^3,L^3]\times [-L,L^3] \cap\LL$ be the
  off-center box surrounding the origin and zipper (as in
  \fref{figure1} except with the lower boundary much closer to the
  origin than the other boundaries), where the lower boundary of the
  box is aligned with an axis of reflection symmetry of the lattice
  $\LL$.  Let $u,v$ be fixed points in $\LL$.
  Then \[\lim_{L\to\infty} \bar G'{}^{\G_L}_{u,v} = \bar G'{}^{\LL}_{u,v}.\]
\end{lemma}
\begin{proof}
  We can approximate $G^{\G_L}_{u,w}$ and $G^{\G_L}_{v,w}$ (for $w$
  within distance $L$ of the origin) using the Green's function of the
  lattice intersected with the upper-half plane.  More precisely,
  we approximate $G^{\G_L}_{p,q}$ by
  \[G^{\approx}_{p,q} := -A^\LL_{p,q} + A^\LL_{p^*,q},\]
  where $p^*=p-(0,2L)$ is the reflection of $p$ through the lower
  boundary of the box, and $A^\LL$ is the potential kernel of the
  lattice.  By construction $G^{\approx}_{p,q}$
  is zero for $q$ along the lower side of the box, and by the
  smoothness result from \lref{Stohrthm}, $G^\approx_{p,q}=O(1/L^2)$
  along the other three sides of the box.  Both $G^\approx_{p,q}$ and
  $G^{\G_L}_{p,q}$ are harmonic in both $p$ and $q$ within the box (except on
  the boundary), and $G^{\G_L}_{p,q}$ is zero along all four sides of the box.
  By the maximal principle for harmonic functions, for $p$ and $q$ within
  $\G_L$ we have
  \[|G^\approx_{p,q} - G^{\G_L}_{p,q}| = O(1/L^2),\]
  i.e., \[\bar G^{\G_L}_{p,q} = -A^\LL_{p,q} + A^\LL_{p^*,q} - A^\LL_{0^*,0} + O(1/L^2). \]

  Next we compare the contribution of a zipper edge $(k,\ell)$ to
  $\bar G'{}^{\G_L}_{u,v}$ and $\bar G'{}^{\LL}_{u,v}$:
\begin{align*}
  -(\bar G^{\G_L}_{u,k} \bar G^{\G_L}_{\ell,v} - \bar G^{\G_L}_{u,\ell} \bar G^{\G_L}_{k,v})
 =&
  -(-A^\LL_{u,k} + A^\LL_{u^*,k} - A^\LL_{0^*,0})(-A^\LL_{v,\ell} + A^\LL_{v^*,\ell} - A^\LL_{0^*,0})\\
& +(-A^\LL_{u,\ell} + A^\LL_{u^*,\ell} - A^\LL_{0^*,0})(-A^\LL_{v,k} + A^\LL_{v^*,k} - A^\LL_{0^*,0})\\& + O(L^{-2} \log L)\\
=&  -(\bar G^{\LL}_{u,k} \bar G^{\LL}_{\ell,v} - \bar G^{\LL}_{u,\ell} \bar G^{\LL}_{k,v}) \\
 &  -(\bar G^{\LL}_{u,k^*} \bar G^{\LL}_{\ell^*,v} - \bar G^{\LL}_{u,\ell^*} \bar G^{\LL}_{k^*,v}) \\
 &  + A^\LL_{u,k} A^\LL_{v^*,\ell} + A^\LL_{u^*,k} A^\LL_{v,\ell} - A^\LL_{u,\ell} A^\LL_{v^*,k} - A^\LL_{u^*,\ell} A^\LL_{v,k} \\
 &  + A^\LL_{0^*,0}\big[-A^\LL_{u,k} + A^\LL_{u^*,k} -A^\LL_{v,\ell} + A^\LL_{v^*,\ell} \\
 &   \phantom{+ A^\LL_{0^*,0}\big[} +A^\LL_{u,\ell} - A^\LL_{u^*,\ell} +A^\LL_{v,k} - A^\LL_{v^*,k}\big]\\
&+O(L^{-2}\log L)\\
=&  -(\bar G^{\LL}_{u,k} \bar G^{\LL}_{\ell,v} - \bar G^{\LL}_{u,\ell} \bar G^{\LL}_{k,v})
    -(\bar G^{\LL}_{u,k^*} \bar G^{\LL}_{\ell^*,v} - \bar G^{\LL}_{u,\ell^*} \bar G^{\LL}_{k^*,v}) \\
 &  + A^\LL_{u,k} (A^\LL_{v^*,\ell}-A^\LL_{0^*,0}) - (A^\LL_{u^*,\ell}-A^\LL_{0^*,0}) A^\LL_{v,k} \\
 &  - A^\LL_{u,\ell} (A^\LL_{v^*,k}-A^\LL_{0^*,0}) + (A^\LL_{u^*,k}-A^\LL_{0^*,0}) A^\LL_{v,\ell} \\
&+O(L^{-2}\log L)
\end{align*}
Recall that $u$ and $v$ are fixed, so they are within distance $O(1)$
of the origin.  The second term is $O(L^{-2}\log L)$.  If the zipper
edge $(k,\ell)$ is at distance $r$ from the origin, then the next
four terms largely cancel one another and add up to $O(1/(r L))$.
Upon summing over all zipper edges, we find $\bar G'{}^{\G_L}_{u,v} =
\bar G'{}^{\LL}_{u,v} + O(L^{-1} \log L)$.
\end{proof}

\subsubsection{Zipper deformations}

The next task we have is to evaluate in closed form the infinite sum
in \eqref{lattice-sum}.  This we can do for many lattices $\LL$, including
the square lattice, triangular lattice, and honeycomb lattice,
although it is not clear how to do this for arbitrary lattices.

We shall need to deform the path that the zipper takes.  In general
deforming the zipper while keeping its endpoints fixed has no effect
on $\bar G'_{u,v}$, unless the zipper is deformed across either $u$ or
$v$.  If the zipper is moved across $u$ in the direction of the arrow
on the zipper, then $\bar G'_{u,v}$ decreases by $\bar G_{u,v}$,
and similarly, moving the zipper across $v$ (in the direction of the
arrow) increases $\bar G'_{u,v}$ by $\bar G_{u,v}$.  We can also
move the endpoint of the zipper by adding a new zipper edge $(k,\ell)$
(or removing an old one) near the endpoint of the zipper, which of
course just adds (or removes) one term to the summations \eqref{ebar}
and \eqref{lattice-sum}.

\subsubsection{Closed-form evaluation of \texorpdfstring{$\bar G'$}{G'} on \texorpdfstring{$\Z^2$}{Z\texttwosuperior}}

Next we evaluate $\bar G'$ for $\Z^2$.  The first step is to rotate
the entire lattice $180^\circ$ about the terminal square of the
zipper.  The rotation of course preserves the lattice $\Z^2$, and maps
$u$ and $v$ to $(1,1)-u$ and $(1,1)-v$ respectively, but now the
zipper goes up to infinity rather than down to infinity.  Let $\bar
G'^\uparrow$ denote $\bar G'$ with the repositioned zipper.  We have
\[ \bar G'_{u,v} = \bar G'^\uparrow_{(1,1)-u,(1,1)-v} .\]

The next step is to deform the zipper so that it once again goes
downwards.  We can deform the initial segment of the zipper so that it
goes downwards, then circles around back up along a large-radius
circle, and then continues back up as before.  By \lref{sum},
the summands along the zipper starting with the large-radius
circle and the subsequent path to infinity are negligible.
So we have
\[ \bar G'_{u,v} = \bar G'^\uparrow_{(1,1)-u,(1,1)-v}
= \bar G'_{(1,1)-u,(1,1)-v} + \parbox{\widthof{another term if zipper was}}{\setlength{\baselineskip}{2pt}another term if zipper was deformed across $u$ or~$v$.}\]

Next we move the location of the start of the zipper, translating it
by $v+u-(1,1)$, by adding a finite number of new zipper edges.  Then
we deform the zipper again, making it go straight down; we have to add
another term if the zipper gets deformed across either $(1,1)-u$ or
$(1,1)-v$.  Because the lattice $\Z^2$ is invariant under such
translations, translating the starting face of the zipper is
equivalent to translating the vertices in the opposite direction.
Thus we have
\[
\bar G'_{u,v} = \bar G'_{v,u} + \text{finite number of terms}.
\]
Finally we use the antisymmetry of $\bar G'_{u,v}$:
\[
\bar G'_{u,v} = \frac{\text{finite number of terms}}{2}.
\]

This procedure is perhaps better explained by way of an example.
We can write
\begin{align*}
  \bar G'_{(0,0),(2,1)}
  &= \bar G'^{\uparrow}_{(1,1),(-1,0)} \\
  &= \bar G'_{(1,1),(-1,0)} + \bar G_{(1,1),(-1,0)} \\
  &= \bar G'_{(2,1),(0,0)} + (\bar G_{(1,1),(0,0)} \bar G_{(0,1),(-1,0)} - \bar G_{(1,1),(0,1)} \bar G_{(0,0),(-1,0)}) + \bar G_{(1,1),(-1,0)} \\
  &= \frac{(\bar G_{(1,1),(0,0)} \bar G_{(0,1),(-1,0)} - \bar G_{(1,1),(0,1)} \bar G_{(0,0),(-1,0)}) + \bar G_{(1,1),(-1,0)}}{2}\\
  &= \frac{1}{2\pi^2} +\frac{1}{\pi} - \frac{5}{32}
\end{align*}

In like manner we can compute $\bar G'_{u,v}$ for any pair of vertices
$u$ and $v$ in $\Z^2$.  The answer will always be in $\Q +
\frac{1}{\pi}\Q + \frac{1}{\pi^2} \Q$.

\subsubsection{Closed-form evaluation of \texorpdfstring{$\bar G'$}{G'} on the triangular lattice}

We can compute $\bar G'$ on the triangular lattice in essentially the
same manner as for $\Z^2$.  The key properties of the lattice that we
used is that it is invariant under $180^\circ$ rotations, and that for
any pair of vertices there is a lattice-invariant translation that
maps the first vertex to the second vertex.

\subsubsection{Closed-form evaluation of \texorpdfstring{$\bar G'$}{G'} on the honeycomb lattice}

The honeycomb lattice is invariant under $180^\circ$ rotations and is
vertex-transitive.  However, there are not lattice-invariant
translations between any pair of vertices: we can partition the
vertices into two color classes, black and white, such that any
lattice-preserving translation will map the black vertices to the
black vertices and the white vertices to the white vertices.

Suppose $u$ is a black vertex and $v$ is a white vertex.  After a
$180^\circ$ rotation about a hexagon, $u$ is mapped to a white vertex
$u'$ and $v$ is mapped to a black vertex $v'$.  We can then translate
$u'$ to $v$ and $v'$ to $u$ while preserving the lattice.  This allows
us to compute $\bar G'_{u,v}$ when $u$ is black and $v$ is white (or
vice versa).

\enlargethispage{15pt}
Since $\bar G'_{u,v}$ is harmonic in both $u$ and $v$ (except along
the zipper), when $u$ and $v$ have the same color, $\bar G'_{u,v}$ can
be expressed as
\[\bar G'_{u,v} = \frac13 (\bar G'_{u,w_1} + \bar G'_{u,w_2} + \bar G'_{u,w_3})\]
(plus another term if one of the edges $(v,w_i)$ crosses the zipper),
where the $w_i$'s are the neighbors of $v$, and the right-hand side we
can compute by the above method.

\subsection{Cutting edges}

Suppose that in the vector bundle setting,
we know the Green's function $\Gv=\Gv^\G$ for a graph $\G$,
and we wish to know the Green's function for the graph
$\G\setminus\{s,t\}$ obtained by deleting an edge $\{s,t\}$ of $\G$.
Recall that $c_{s,t}$ denotes the conductance of edge $(s,t)$, and let
us denote by $\tau$ the parallel transport to $s$ from $t$, so that
$\Delta^\G_{s,t}=-c_{s,t}\tau$ and $\Delta^\G_{t,s}=-c_{s,t}\tau^*$.
Then it is readily checked that
\begin{align}
  \Gv^{\G\setminus\{s,t\}}_{u,v}
& = \Gv_{u,v} - \frac{(\Gv_{u,s}-\Gv_{u,t} \tau^*) (\Gv_{s,v} - \tau \Gv_{t,v})}{\alpha_{s,t}} \label{Ghat}
\intertext{where}
\alpha_{s,t} &= \Gv_{s,s} +\Gv_{t,t} -\Gv_{s,t}\tau^* - \tau\Gv_{t,s} - 1/c_{s,t}
\end{align}
(which is a scalar).
Indeed, if we let $f(u,v)$ denote the purported Green's
function on the right-hand side of \eqref{Ghat},
then $f(u,v)=0$ when either $u$ or $v$ is the boundary, and we have
\begin{align*}
\sum_v  f(u,v) \Delta^\G_{v,w}
& = \delta_{u,w} - \frac{(\Gv_{u,s}-\Gv_{u,t} \tau^*) (\delta_{s,w} - \tau \delta_{t,w})}{\alpha_{s,t}}.
\intertext{If $w\neq s$ and $w\neq t$ then $\Delta^{\G\setminus\{s,t\}}_{v,w}=\Delta^{\G}_{v,w}$, so}
\sum_v  f(u,v) \Delta^{\G\setminus\{s,t\}}_{v,w} &= \delta_{u,w}\quad\quad\text{(if $w\neq s$ and $w\neq t$).}
\end{align*}
Suppose now $w=s$.  Then
\begin{align*}
\sum_v  f(u,v) \Delta&^{\G\setminus\{s,t\}}_{v,s} =
\sum_v  f(u,v) \Delta^{\G}_{v,s} + f(u,t) c_{t,s} \tau^* - f(u,s) c_{t,s}\\
  &= \delta_{u,s} - \frac{\Gv_{u,s}-\Gv_{u,t}\tau^*}{\alpha_{s,t}}
 + \left[\Gv_{u,t} - \frac{(\Gv_{u,s}-\Gv_{u,t}\tau^*) (\Gv_{s,t} - \tau\Gv_{t,t})}{\alpha_{s,t}}\right]
   c_{t,s} \tau^* \\
& \phantom{= \delta_{u,s} - \frac{\Gv_{u,s}-\Gv_{u,t}\tau^*}{\alpha_{s,t}}}
 - \left[\Gv_{u,s} - \frac{(\Gv_{u,s}-\Gv_{u,t}\tau^*) (\Gv_{s,s} - \tau\Gv_{t,s})}{\alpha_{s,t}}\right] c_{t,s}\\
  &= \delta_{u,s} - \frac{\Gv_{u,s}-\Gv_{u,t}\tau^*}{\alpha_{s,t}} c_{t,s}
 \left[ \frac{1}{c_{t,s}} + \alpha_{s,t} + (\Gv_{s,t}\tau^* - \Gv_{t,t} - \Gv_{s,s} + \tau\Gv_{t,s}) \right]\\
&=\delta_{u,s}
\end{align*}
by the choice of $\alpha_{s,t}$.  The case $w=t$ is similar.

Let us return to the line bundle setting, with a zipper monodromy of $z$,
that we are interested in near $z=1$.  If $(s,t)$ is a zipper edge then
$\tau=z$ or $\tau=1/z$ (depending on the zipper direction), and otherwise
$\tau=1$.  Let $\tau'=\partial_z \tau|_{z=1}$.  Recall that $G=\Gv|_{z=1}$ and
$G'=\partial_z\Gv|_{z=1}$.  From \eqref{Ghat} it is evident that
\begin{align}
  G^{\G\setminus\{s,t\}}_{u,v}
& = G_{u,v} - \frac{(G_{u,s}-G_{u,t}) (G_{s,v} - G_{t,v})}{a_{s,t}} \label{Gcut}
\intertext{where}
a_{s,t} &= G_{s,s} +G_{t,t} - 2 G_{s,t} - 1/c_{s,t} \label{acut}.
\end{align}

We have
\begin{align*}
  \partial_z \alpha_{s,t} &= \partial_z \Gv_{s,s} + \partial_z\Gv_{t,t} -(\partial_z\Gv_{s,t})\tau^* -\Gv_{s,t}\partial_z\tau^* -(\partial_z\tau)\Gv_{t,s} -\tau\partial_z\Gv_{t,s}\\
  \left.\partial_z \alpha_{s,t}\right|_{z=1} &= 0.
\end{align*}
Using this, we can differentiate \eqref{Ghat}
with respect to the zipper monodromy $z$ and set $z=1$ to obtain
\begin{align}
\big(G^{\G\setminus\{s,t\}}_{u,v}\big)'
 = G'_{u,v}
&- (G'_{u,s}-G'_{u,t} + G_{u,t}\tau') (G_{s,v} - G_{t,v}) / a_{s,t} \label{Gcut'bar}\\
&+ (G_{u,s}-G_{u,t}) (G'_{s,v} -G'_{t,v} - \tau' G_{t,v}) / a_{s,t}. \notag
\end{align}

One final edge-cutting formula is
\begin{equation}  \frac{Z^{\G\setminus\{s,t\}}}{Z^{\G}} = 1-c_{s,t}(G_{s,s}+G_{t,t}-2 G_{s,t}).
 \label{Zcut}
\end{equation}
This holds because directed edge $(s,t)$ occurs in a uniform
spanning tree of $\G$ with probability $c_{s,t}(G_{s,s}-G_{s,t})$,
and likewise directed edge $(t,s)$ may occur in the spanning tree.

These formulas \eqref{Gcut}, \eqref{acut}, \eqref{Gcut'bar} and
\eqref{Zcut} of course apply to the modified graph~$\bar\G$ simply by
replacing $G$ and $G'$ with $\bar G$ and $\bar G'$.

\section{Loop-erased random walk}
\label{LERW}

In this section we show how to compute the probability that the LERW from $(0,0)$ to~$\infty$ in $\Z^2$
(or the triangular or honeycomb lattices) passes through any given vertex or edge.

\subsection{Preliminary remarks} \label{sec:LERW-prelim}

We let $P_{v,w}$ denote the probability that the LERW started from
$(0,0)$ to~$\infty$ passes through edge $(v,w)$, in the direction from
$v$ to $w$.  Likewise we let $P_w$ denote the probability that the LERW
passes through vertex~$w$.  It is straightforward that
 \[ P_w = \sum_{v:v\sim w} P_{w,v} = \sum_{v:v\sim w} P_{v,w} + \delta_{w,0}.\]
Our strategy is to compute these edge probabilities.

We find it conceptually convenient to work with finite graphs $\G$,
set up our equations for spanning tree and grove event probabilities
in terms of the finite-graph Green's function $G^\G$ and its
derivative $(G^\G)'$ using the formulas from section~\ref{annular},
and then afterwards take the limit $\G\to\LL$ using the formulas from
section~\ref{Greens}.  We are at liberty to use any convenient
sequence $\G_k$ of finite graphs that converge to the lattice $\LL$,
since the limiting measure on spanning trees of $\LL$ is independent
of the choice of $\G_k$, and the event that the LERW from $(0,0)$ to~$\infty$
uses a given edge is a measurable event in the limiting
measure.  So we choose a sequence $\G_k$ for which it is convenient to
compute $(G^{\G_k})'$, as described in section~\sref{Greens'}.  These
graphs $\G_k$ have a wired boundary vertex (see \fref{figure1}) that
we will label $\infty$, even though the graphs are finite.  The other
vertices of $\G_k$ we will label by their coordinates in $\LL$.  We
define $P^\G_{v,w}$ and $P^\G_{w}$ in the same way as we defined
$P_{v,w}$ and $P_{w}$; for fixed $v$ and~$w$, $\lim_{\G\to\LL}
P^\G_{v,w} = P_{v,w}$.

Once an edge traversal probability $P^\LL_{v,w}$ has
been computed, finding the corresponding probability $P^\LL_{w,v}$ for
the reversed edge is straightforward:

\begin{lemma}[\cite{Kenyon.longrange}]\label{dimerdifference}
  $ P^\LL_{v,w} - P^\LL_{w,v} = c_{v,w} (\bar G_{0,v} - \bar G_{0,w}) $
\end{lemma}

\begin{proof}
  Let the origin $0$ be at the center of a large box with wired
  boundary whose size we will send to infinity.  For simple random
  walk started at $0$, the expected number of traversals of $(v,w)$
  minus the expected number of traversals of $(w,v)$ is the edge
  conductance $c_{v,w}$ times the difference in Green's functions.
  The same also holds for loop-erased random walk started at $0$,
  since cycles are reversible.
\end{proof}

The intensity of the undirected edge $\{v,w\}$ is
$P^\LL_{\{v,w\}}=P^\LL_{v,w}+P^\LL_{w,v}$, and the undirected edge
intensities turn out to be nicer numbers than the directed edge
intensities.  The vertex intensities are easily calculated from the
undirected edge intensities, and given the potential kernel of $\LL$,
the directed edge intensities are easily recovered from the undirected
intensities.

\subsection{Computation of directed edge intensities}

Suppose $\G$ is a finite connected graph, $u$, $v$, $w$, and $r$ are
vertices of $\G$, and $(v,w)$ is an edge of $\G$.  If the path
connecting $u$ to~$r$ within a spanning tree of $\G$ passes through an
edge $(v,w)$ in the direction from $v$ to $w$, then deleting this edge
results in a grove of type $u,v\mid w,r$, and conversely, adjoining edge
$(v,w)$ to a grove of type $u,v\mid w,r$ yields a spanning tree of $\G$ in
which the path from $u$ to $r$ passes through the edge $(v,w)$ in the
direction from $v$ to $w$.  If $u$ is the vertex considered to be
the origin in $\G$ and $r$ is the wired boundary, then
\[ P^\G_{v,w} = \frac{Z[u,v\mid w,r]}{Z[u,v,w,r]} .\]

\begin{window}[3,r,\includegraphics[scale=0.33]{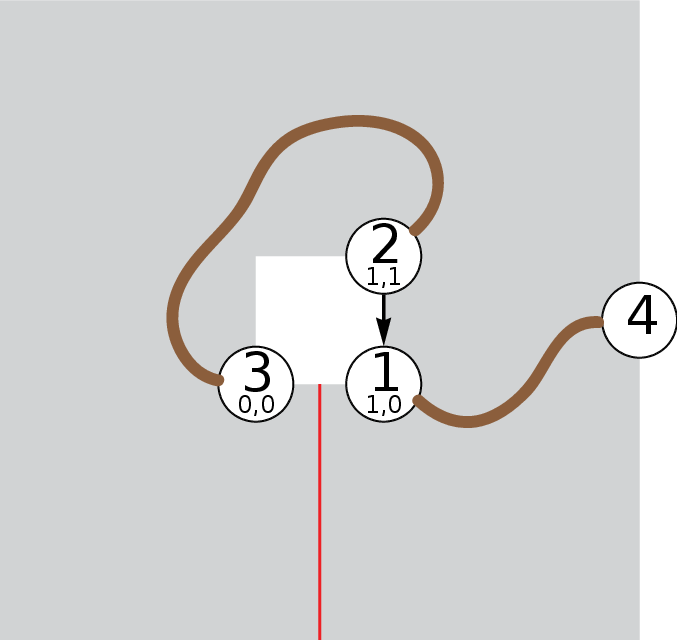},{}]
For example, consider the probability that the LERW from $(0,0)$ to
$\infty$ in $\Z^2$ uses the directed edge from $(1,1)$ to $(1,0)$.  As
described in section~\ref{sec:LERW-prelim}, we approximate $\Z^2$ by a
large finite grid $\G$ with wired boundary.  There are four vertices
of interest, so we declare them to be nodes.  Their coordinates are
$(1,0)$, $(1,1)$, $(0,0)$, and $\infty$, and we also refer to them as
nodes $1$, $2$, $3$, and~$4$.  Since $\G$ is planar and nodes $1$,
$2$, and $3$ bound the same face, we can view $\G$ as a surface graph
embedded on the annulus, and place a zipper from the central face to
the outer boundary.  This annular surface graph with four nodes is the
example we showed earlier in \fref{figure1}, and which we also
show schematically on the right.
Recall equation \eqref{Pc(23|14)} for annular-one graphs:
\end{window}
\[
\frac{Z[3,2|1,4]}{Z[1,2,3,4]}
                 = - G'_{1,2} - G'_{2,3} - G'_{3,1} +G^{\bar\G}_{2,3} - G^{\bar\G}_{1,3}.
\]
As discussed in section~\ref{sec:Gbar}, we can use either the Green's
function $G$ and its monodromy derivative $G'$ for the original graph
$\G$, or for the graph $\bar\G$ which has an auxillary vertex
connected to the boundary.
Using the method described in section~\ref{Greens}, we compute
\[
\begin{gathered}\bordermatrix[{[]}]{
\mathllap{\bar G^{\Z^2}}& (1,0)& (1,1)& (0,0)\\
\mathllap{(1,0)} & 0& -\frac{1}{4}& -\frac{1}{4}\\[2pt]
\mathllap{(1,1)} & -\frac{1}{4}& 0& -\frac{1}{\pi }\\[2pt]
\mathllap{(0,0)} & -\frac{1}{4}& -\frac{1}{\pi }& 0\\[2pt]
}\end{gathered}
\hspace{1.0in}
\begin{gathered}\bordermatrix[{[]}]{
\mathllap{\bar G'{}^{\Z^2}}& (1,0)& (1,1)& (0,0)\\
\mathllap{(1,0)} & 0& -\frac{3}{32}& -\frac{5}{32}\\[2pt]
\mathllap{(1,1)} & \frac{3}{32}& 0& -\frac{1}{2 \pi }\\[2pt]
\mathllap{(0,0)} & \frac{5}{32}& \frac{1}{2 \pi }& 0\\[2pt]
}\end{gathered}
\]
We evaluate $\lim_{\G\to\Z^2} P^\G_{(1,1),(1,0)}$ by substituting $\bar G^{\Z^2}$ for $\bar G^\G = G^{\bar\G}$ and
$\bar G'{}^{\Z^2}$ for $\bar G'{}^\G = (G^{\bar\G})'$:
\begin{equation} \label{Z2-1110}
\begin{aligned}
  P^{\Z^2}_{(1,1),(1,0)} &= \frac{3}{32} + \frac{1}{2\pi} - \frac{5}{32} -\frac{1}{\pi} + \frac{1}{4}
= +\frac{3}{16} -\frac{1}{2\pi}.
\end{aligned}
\end{equation}

The edge $(1,1)(1,0)$ was close enough to the origin for all three
node $(0,0)$, $(1,1)$, and $(1,0)$ to be incident to the same face,
which made it straightforward to apply our formulas for annular networks
to compute the edge intensity.  For edges further away, more work is required.

For an arbitrary directed edge $(v,w)$, we identify a set of edges to
cut so as to place the starting point $u$ of the LERW and the
endpoints of the edge on the same face, and call the cut graph $\tilde
G$.  The vertices $u$, $v$, and $w$, together with the endpoints of
all of the cut edges, comprise the nodes on the inner boundary of the
annulus.  We number these nodes in counterclockwise order so that the
zipper starts between nodes $1$ and $n-1$.  The vertex labeled
$\infty$ becomes node~$n$, which is on the other boundary of the
annulus.  The cut graph is an annular-one surface graph.

A grove of type $u,v\mid w,r$ in $\G$ may contain some of the cut edges,
and upon removing these edges, it becomes a grove of some other type
in $\tilde\G$.  We can enumerate
 all possible subsets of the cut edges
and all possible grove types $\sigma$ in $\tilde\G$
which
combine to form a grove of type $u,v\mid w,r$ in $\G$, and thereby express $Z^\G[u,v\mid w,r]$ as a
\begin{wrapfigure}[9]{r}{0pt}\smash{\raisebox{-85pt}{\includegraphics[scale=0.3]{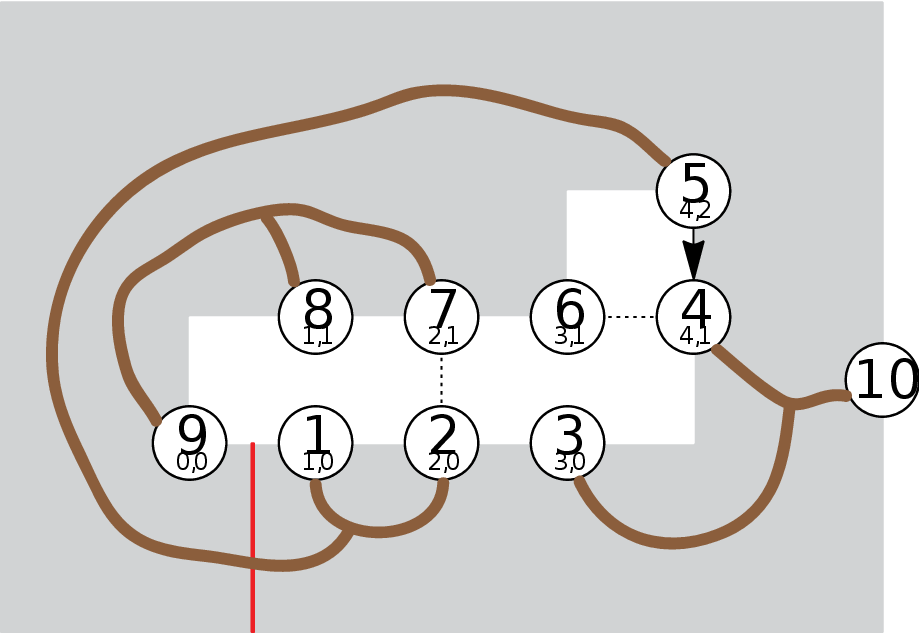}}}\end{wrapfigure}
linear combination of $Z^{\tilde\G}[\sigma]$'s.
For example, to compute the intensity of the directed edge $(4,2)(4,1)$ in $\Z^2$,
we can cut four edges, as shown at right.  In this case there are $10$ nodes,
and the grove $Z[1,2,5|3,4,10|6|7,8,9]$ together with cut-edges $(2,7)$ and $(4,6)$
are among those which combine to form a grove of type $Z[9,5|4,10]$ in the original graph.
Using
\tref{combination-partial-pairing}, for each such grove type $\sigma$
we can express $Z^{\tilde\G}[\sigma]$ as a linear combination of
$Z^{\tilde\G}[\tau]$'s, where the $\tau$'s are partial pairings of the
nodes of $\tilde\G$ in which node $n$ is paired.  For each such
partial pairing $\tau$, we can use \tref{partial-pairing-greens}
to express $Z^{\tilde\G}[\tau]/Z^{\tilde\G}$ as the $\zeta\to 1$ limit of a power of $1/(1-\zeta)$ times
a linear combination (with polynomial in $\zeta$ coefficients) of determinants
of matrices whose entries are of the form
\[G^{\tilde\G}_{i,j}+\frac{\zeta-1}{2} G'{}^{\tilde\G}_{i,j}. \]

We can then replace each Green's function entry $G^{\tilde\G}_{i,j}$
with the potential kernel $\bar{G}^{\tilde\G}_{i,j}$ and each
$G'{}^{\tilde\G}_{i,j}$ with $\bar{G}'{}^{\tilde\G}_{i,j}$.  We then
take the limit where $\G$ tends to the infinite lattice, which
replaces each $\bar{G}^{\tilde\G}_{i,j}$ with
$\bar{G}^{\tilde\LL}_{i,j}$ and each $\bar{G}'{}^{\tilde\G}_{i,j}$
with $\bar{G}'{}^{\tilde\LL}_{i,j}$, where $\tilde\LL$ is the infinite
lattice with some edges cut.  Each of these can then be evaluated in
closed form using the formulas in section~\ref{Greens}.  We then
multiply by $Z^{\tilde\LL}/Z^\LL$ using \eqref{Zcut} to obtain
$Z^{\LL}[u,v\mid w,\infty]/Z^{\LL}$, which is the directed edge intensity.

There are many steps in these computations, but the whole process can
be handled by computer, and it was not previously known that these
LERW intensities were computable or had a closed form expression.  We record the
results of our LERW intensity computations for the square lattice in
\fref{square-intensity}, for the honeycomb lattice in
\fref{hex-intensity}, and for the triangular lattice in
\fref{tri-intensity}.

\newcommand{\longrat}[2]{\genfrac{}{}{0pt}{}{\textstyle #1}{\textstyle #2}}
\newcommand{\rtthree}{{}\scalebox{.625}[1]{$\scriptstyle\sqrt{\scalebox{1.6}[1]{$\scriptstyle 3$}}$}}

\begin{figure}[b]
\begin{center}
\begin{tikzpicture}[scale=3.95392]
\begin{scope}[line width=2pt]
\draw [color=black!12.5](0., 1.)--(-0.1, 1.);
\draw [color=black!12.5](0., 1.)--(0., 2.);
\draw [color=black!12.5](0., 1.)--(1., 1.);
\draw [color=black!12.5](1., 0.)--(1., -0.1);
\draw [color=black!12.5](1., 0.)--(1., 1.);
\draw [color=black!12.5](1., 0.)--(2., 0.);
\draw [color=black!25.](0., 0.)--(0., -0.1);
\draw [color=black!25.](0., 0.)--(0., 1.);
\draw [color=black!25.](0., 0.)--(-0.1, 0.);
\draw [color=black!25.](0., 0.)--(1., 0.);
\draw [color=black!4.81025](3., 3.)--(3.1, 3.);
\draw [color=black!4.81025](3., 3.)--(3., 3.1);
\draw [color=black!5.3819](2., 3.)--(2., 3.1);
\draw [color=black!5.3819](2., 3.)--(3., 3.);
\draw [color=black!5.3819](3., 2.)--(3.1, 2.);
\draw [color=black!5.3819](3., 2.)--(3., 3.);
\draw [color=black!5.92688](1., 3.)--(1., 3.1);
\draw [color=black!5.92688](3., 1.)--(3.1, 1.);
\draw [color=black!5.94329](1., 3.)--(2., 3.);
\draw [color=black!5.94329](3., 1.)--(3., 2.);
\draw [color=black!6.21611](0., 3.)--(0., 3.1);
\draw [color=black!6.21611](3., 0.)--(3.1, 0.);
\draw [color=black!6.30464](2., 2.)--(2., 3.);
\draw [color=black!6.30464](2., 2.)--(3., 2.);
\draw [color=black!6.33144](0., 3.)--(-0.1, 3.);
\draw [color=black!6.33144](0., 3.)--(1., 3.);
\draw [color=black!6.33144](3., 0.)--(3., -0.1);
\draw [color=black!6.33144](3., 0.)--(3., 1.);
\draw [color=black!7.38557](1., 2.)--(1., 3.);
\draw [color=black!7.38557](1., 2.)--(2., 2.);
\draw [color=black!7.38557](2., 1.)--(2., 2.);
\draw [color=black!7.38557](2., 1.)--(3., 1.);
\draw [color=black!8.19988](0., 2.)--(0., 3.);
\draw [color=black!8.19988](2., 0.)--(3., 0.);
\draw [color=black!8.31641](0., 2.)--(-0.1, 2.);
\draw [color=black!8.31641](0., 2.)--(1., 2.);
\draw [color=black!8.31641](2., 0.)--(2., -0.1);
\draw [color=black!8.31641](2., 0.)--(2., 1.);
\draw [color=black!9.60831](1., 1.)--(1., 2.);
\draw [color=black!9.60831](1., 1.)--(2., 1.);
\path (0., 0.5) node[rotate=90]{\maxsizebox{3.95392cm}{!}{\resizebox{!}{18pt}{$\hspace{9pt}\frac{1}{4}\hspace{9pt}$}}};
\path (0., 1.5) node[rotate=90]{\maxsizebox{3.95392cm}{!}{\resizebox{!}{18pt}{$\hspace{9pt}\frac{1}{8}\hspace{9pt}$}}};
\path (0., 2.5) node[rotate=90]{\maxsizebox{3.95392cm}{!}{\resizebox{!}{18pt}{$\hspace{9pt}\frac{1}{8}{-}\frac{1}{2 \pi }{+}\frac{2}{\pi^2}{-}\frac{3}{\pi^3}{+}\frac{1}{\pi^4}\hspace{9pt}$}}};
\path (0.5, 0.) node[rotate=0]{\maxsizebox{3.95392cm}{!}{\resizebox{!}{18pt}{$\hspace{9pt}\frac{1}{4}\hspace{9pt}$}}};
\path (0.5, 1.) node[rotate=0]{\maxsizebox{3.95392cm}{!}{\resizebox{!}{18pt}{$\hspace{9pt}\frac{1}{8}\hspace{9pt}$}}};
\path (0.5, 2.) node[rotate=0]{\maxsizebox{3.95392cm}{!}{\resizebox{!}{18pt}{$\hspace{9pt}\frac{1}{2 \pi }{-}\frac{3}{4 \pi^2}\hspace{9pt}$}}};
\path (0.5, 3.) node[rotate=0]{\maxsizebox{3.95392cm}{!}{\resizebox{!}{18pt}{$\hspace{9pt}{-}\frac{1}{4}{+}\frac{2}{\pi }{-}\frac{15}{4 \pi^2}{+}\frac{4}{\pi^3}{-}\frac{10}{\pi^4}{+}\frac{8}{\pi^5}{+}\frac{4}{\pi^6}\hspace{9pt}$}}};
\path (1., 0.5) node[rotate=90]{\maxsizebox{3.95392cm}{!}{\resizebox{!}{18pt}{$\hspace{9pt}\frac{1}{8}\hspace{9pt}$}}};
\path (1., 1.5) node[rotate=90]{\maxsizebox{3.95392cm}{!}{\resizebox{!}{18pt}{$\hspace{9pt}\frac{1}{8}{-}\frac{1}{4 \pi }{+}\frac{1}{2 \pi^2}\hspace{9pt}$}}};
\path (1., 2.5) node[rotate=90]{\maxsizebox{3.95392cm}{!}{\resizebox{!}{18pt}{$\hspace{9pt}\frac{1}{8}{-}\frac{1}{2 \pi }{+}\frac{3}{2 \pi^2}{-}\frac{2}{\pi^3}{+}\frac{2}{\pi^4}\hspace{9pt}$}}};
\path (1.5, 0.) node[rotate=0]{\maxsizebox{3.95392cm}{!}{\resizebox{!}{18pt}{$\hspace{9pt}\frac{1}{8}\hspace{9pt}$}}};
\path (1.5, 1.) node[rotate=0]{\maxsizebox{3.95392cm}{!}{\resizebox{!}{18pt}{$\hspace{9pt}\frac{1}{8}{-}\frac{1}{4 \pi }{+}\frac{1}{2 \pi^2}\hspace{9pt}$}}};
\path (1.5, 2.) node[rotate=0]{\maxsizebox{3.95392cm}{!}{\resizebox{!}{18pt}{$\hspace{9pt}\frac{1}{8}{-}\frac{1}{2 \pi }{+}\frac{3}{2 \pi^2}{-}\frac{2}{\pi^3}{+}\frac{2}{\pi^4}\hspace{9pt}$}}};
\path (1.5, 3.) node[rotate=0]{\maxsizebox{3.95392cm}{!}{\resizebox{!}{18pt}{$\hspace{9pt}\frac{3}{16}{-}\frac{3}{4 \pi }{+}\frac{7}{8 \pi^2}{+}\frac{4}{3 \pi^4}{+}\frac{8}{\pi^6}\hspace{9pt}$}}};
\path (2., 0.5) node[rotate=90]{\maxsizebox{3.95392cm}{!}{\resizebox{!}{18pt}{$\hspace{9pt}\frac{1}{2 \pi }{-}\frac{3}{4 \pi^2}\hspace{9pt}$}}};
\path (2., 1.5) node[rotate=90]{\maxsizebox{3.95392cm}{!}{\resizebox{!}{18pt}{$\hspace{9pt}\frac{1}{8}{-}\frac{1}{2 \pi }{+}\frac{3}{2 \pi^2}{-}\frac{2}{\pi^3}{+}\frac{2}{\pi^4}\hspace{9pt}$}}};
\path (2., 2.5) node[rotate=90]{\maxsizebox{3.95392cm}{!}{\resizebox{!}{18pt}{$\hspace{9pt}\frac{1}{16}{-}\frac{5}{24 \pi }{+}\frac{13}{12 \pi^2}{-}\frac{7}{3 \pi^3}{+}\frac{4}{\pi^4}{-}\frac{8}{3 \pi^5}\hspace{9pt}$}}};
\path (2.5, 0.) node[rotate=0]{\maxsizebox{3.95392cm}{!}{\resizebox{!}{18pt}{$\hspace{9pt}\frac{1}{8}{-}\frac{1}{2 \pi }{+}\frac{2}{\pi^2}{-}\frac{3}{\pi^3}{+}\frac{1}{\pi^4}\hspace{9pt}$}}};
\path (2.5, 1.) node[rotate=0]{\maxsizebox{3.95392cm}{!}{\resizebox{!}{18pt}{$\hspace{9pt}\frac{1}{8}{-}\frac{1}{2 \pi }{+}\frac{3}{2 \pi^2}{-}\frac{2}{\pi^3}{+}\frac{2}{\pi^4}\hspace{9pt}$}}};
\path (2.5, 2.) node[rotate=0]{\maxsizebox{3.95392cm}{!}{\resizebox{!}{18pt}{$\hspace{9pt}\frac{1}{16}{-}\frac{5}{24 \pi }{+}\frac{13}{12 \pi^2}{-}\frac{7}{3 \pi^3}{+}\frac{4}{\pi^4}{-}\frac{8}{3 \pi^5}\hspace{9pt}$}}};
\path (2.5, 3.) node[rotate=0]{\maxsizebox{3.95392cm}{!}{\resizebox{!}{18pt}{$\hspace{9pt}\frac{1}{32}{-}\frac{1}{12 \pi }{+}\frac{13}{18 \pi^2}{-}\frac{14}{9 \pi^3}{+}\frac{40}{9 \pi^4}{-}\frac{64}{9 \pi^5}{+}\frac{32}{9 \pi^6}\hspace{9pt}$}}};
\path (3., 0.5) node[rotate=90]{\maxsizebox{3.95392cm}{!}{\resizebox{!}{18pt}{$\hspace{9pt}{-}\frac{1}{4}{+}\frac{2}{\pi }{-}\frac{15}{4 \pi^2}{+}\frac{4}{\pi^3}{-}\frac{10}{\pi^4}{+}\frac{8}{\pi^5}{+}\frac{4}{\pi^6}\hspace{9pt}$}}};
\path (3., 1.5) node[rotate=90]{\maxsizebox{3.95392cm}{!}{\resizebox{!}{18pt}{$\hspace{9pt}\frac{3}{16}{-}\frac{3}{4 \pi }{+}\frac{7}{8 \pi^2}{+}\frac{4}{3 \pi^4}{+}\frac{8}{\pi^6}\hspace{9pt}$}}};
\path (3., 2.5) node[rotate=90]{\maxsizebox{3.95392cm}{!}{\resizebox{!}{18pt}{$\hspace{9pt}\frac{1}{32}{-}\frac{1}{12 \pi }{+}\frac{13}{18 \pi^2}{-}\frac{14}{9 \pi^3}{+}\frac{40}{9 \pi^4}{-}\frac{64}{9 \pi^5}{+}\frac{32}{9 \pi^6}\hspace{9pt}$}}};
\node (0,0) {\contour{white}{origin}};
\end{scope}
\end{tikzpicture}
\end{center}
\vspace*{-12pt}
\caption{Undirected edge intensities of loop-erased random walk on
  $\Z^2$.  For $x\geq 1$, the edges $(x,x)(x,x-1)$ and
  $(x,x-1)(x+1,x-1)$ appear to have identical intensities,
  despite there being no lattice symmetry that would imply
  this.  The cases $x=1$ and $x=2$ are shown here, we have
  also checked the cases $x=3,4,5$.
}
\label{square-intensity}
\end{figure}

\begin{figure}[htpb]
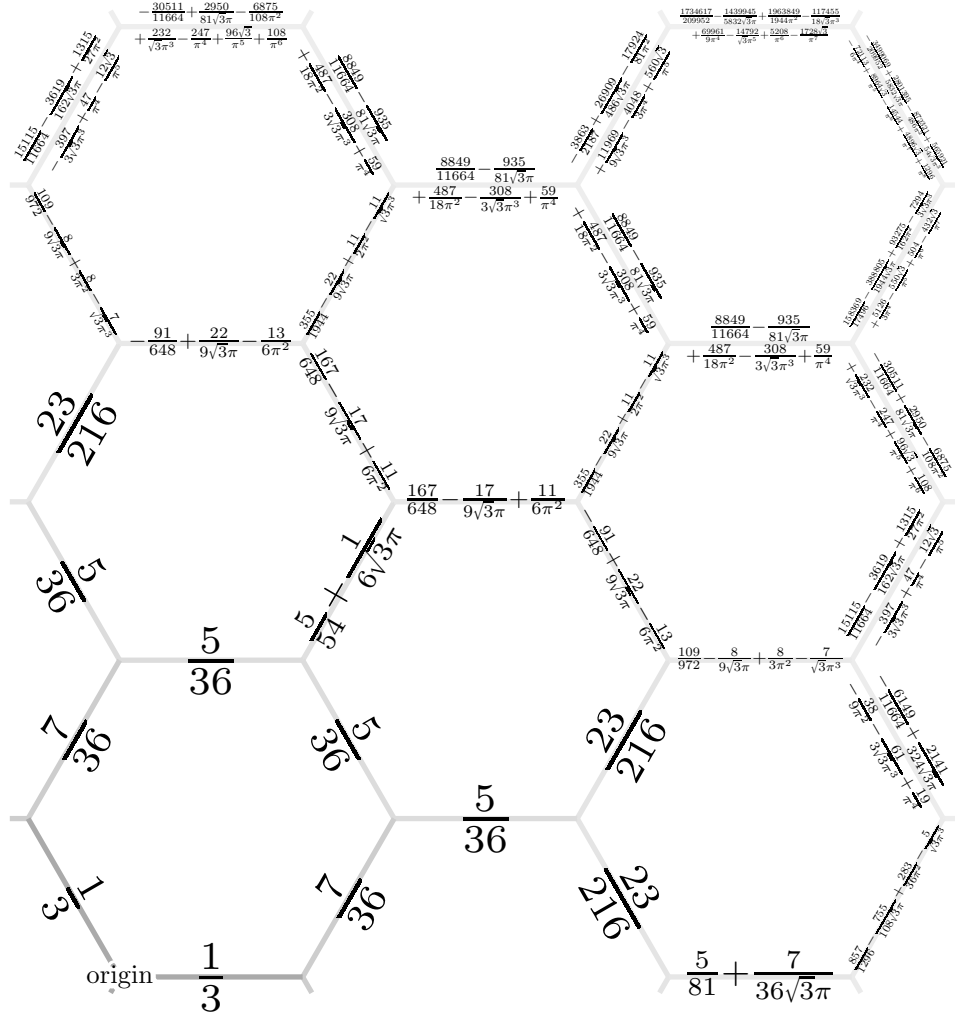

\begin{center}
\include{lerw-intensity-hexagonal}
\end{center}
\vspace*{-12pt}
\caption{Undirected edge intensity of loop-erased random walk on the
  honeycomb lattice.  Using coordinates for which vertex $(x,y)$ is
  the one located at position $(x-y/2,y\sqrt{3}/2)$, the edge
  intensities $(3x-1,3x-1)(3x-1,3x-2)$ and $(3x-1,3x-2)(3x,3x-2)$ are
  identical (for $x=1,2,3$), as are the intensities for edges
  $(3x,1)(3x+1,2)$ and $(3x,1)(3x,0)$ (for $x=1,2,3$ and perhaps all~$x$),
  despite there being no lattice symmetry that would imply this.}
\label{hex-intensity}
\end{figure}

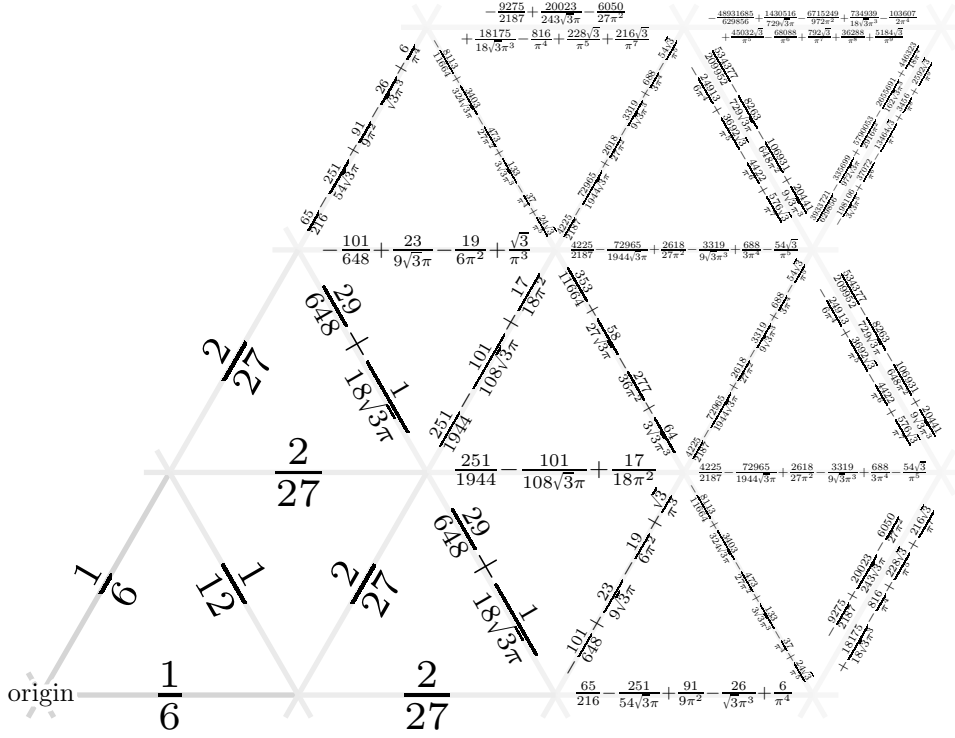
\begin{figure}[htpb]
\begin{center}
\begin{tikzpicture}[scale=3.41961]
\begin{scope}[line width=2pt]
\draw [color=black!16.6667](0., 0.)--(-0.05, -0.0866025);
\draw [color=black!16.6667](0., 0.)--(-0.05, 0.0866025);
\draw [color=black!16.6667](0., 0.)--(0.05, -0.0866025);
\draw [color=black!16.6667](0., 0.)--(-0.1, 0.);
\draw [color=black!16.6667](0., 0.)--(0.5, 0.866025);
\draw [color=black!16.6667](0., 0.)--(1., 0.);
\draw [color=black!2.88004](3.5, 2.59808)--(3.55, 2.68468);
\draw [color=black!2.88004](3.5, 2.59808)--(3.6, 2.59808);
\draw [color=black!2.9347](3.5, 2.59808)--(3.45, 2.68468);
\draw [color=black!2.98167](3.5, 2.59808)--(3.55, 2.51147);
\draw [color=black!3.28305](2.5, 2.59808)--(2.55, 2.68468);
\draw [color=black!3.28305](3.5, 0.866025)--(3.6, 0.866025);
\draw [color=black!3.28453](2.5, 2.59808)--(3.5, 2.59808);
\draw [color=black!3.28453](3.5, 0.866025)--(3.55, 0.952628);
\draw [color=black!3.32168](2.5, 2.59808)--(2.45, 2.68468);
\draw [color=black!3.32168](3.5, 0.866025)--(3.55, 0.779423);
\draw [color=black!3.37162](3., 1.73205)--(3.1, 1.73205);
\draw [color=black!3.37162](3., 1.73205)--(3.5, 2.59808);
\draw [color=black!3.47352](3., 1.73205)--(2.5, 2.59808);
\draw [color=black!3.47352](3.5, 0.866025)--(3., 1.73205);
\draw [color=black!3.7215](1.5, 2.59808)--(1.55, 2.68468);
\draw [color=black!3.7215](3., 0.)--(3.1, 0.);
\draw [color=black!3.73529](1.5, 2.59808)--(1.45, 2.68468);
\draw [color=black!3.73529](1.5, 2.59808)--(2.5, 2.59808);
\draw [color=black!3.73529](3., 0.)--(3.05, -0.0866025);
\draw [color=black!3.73529](3., 0.)--(3.5, 0.866025);
\draw [color=black!4.04103](2., 1.73205)--(2.5, 2.59808);
\draw [color=black!4.04103](2., 1.73205)--(3., 1.73205);
\draw [color=black!4.04103](2.5, 0.866025)--(3., 1.73205);
\draw [color=black!4.04103](2.5, 0.866025)--(3.5, 0.866025);
\draw [color=black!4.11652](1.5, 2.59808)--(1.4, 2.59808);
\draw [color=black!4.11652](2., 1.73205)--(1.5, 2.59808);
\draw [color=black!4.11652](3., 0.)--(2.5, 0.866025);
\draw [color=black!4.11652](3., 0.)--(2.95, -0.0866025);
\draw [color=black!4.26684](2.5, 0.866025)--(2., 1.73205);
\draw [color=black!4.86409](1., 1.73205)--(1.5, 2.59808);
\draw [color=black!4.86409](2., 0.)--(3., 0.);
\draw [color=black!4.87972](1., 1.73205)--(0.95, 1.81865);
\draw [color=black!4.87972](1., 1.73205)--(2., 1.73205);
\draw [color=black!4.87972](2., 0.)--(2.05, -0.0866025);
\draw [color=black!4.87972](2., 0.)--(2.5, 0.866025);
\draw [color=black!5.29426](1.5, 0.866025)--(2., 1.73205);
\draw [color=black!5.29426](1.5, 0.866025)--(2.5, 0.866025);
\draw [color=black!5.49629](1., 1.73205)--(0.9, 1.73205);
\draw [color=black!5.49629](1.5, 0.866025)--(1., 1.73205);
\draw [color=black!5.49629](2., 0.)--(1.5, 0.866025);
\draw [color=black!5.49629](2., 0.)--(1.95, -0.0866025);
\draw [color=black!7.40741](0.5, 0.866025)--(0.45, 0.952628);
\draw [color=black!7.40741](0.5, 0.866025)--(1., 1.73205);
\draw [color=black!7.40741](0.5, 0.866025)--(1.5, 0.866025);
\draw [color=black!7.40741](1., 0.)--(1.05, -0.0866025);
\draw [color=black!7.40741](1., 0.)--(1.5, 0.866025);
\draw [color=black!7.40741](1., 0.)--(2., 0.);
\draw [color=black!8.33333](0.5, 0.866025)--(0.4, 0.866025);
\draw [color=black!8.33333](1., 0.)--(0.5, 0.866025);
\draw [color=black!8.33333](1., 0.)--(0.95, -0.0866025);
\path (0.25, 0.433013) node[rotate=60]{\maxsizebox{3.41961cm}{!}{\resizebox{!}{18pt}{$\hspace{10pt}\frac{1}{6}\hspace{10pt}$}}};
\path (0.5, 0.) node[rotate=0]{\maxsizebox{3.41961cm}{!}{\resizebox{!}{18pt}{$\hspace{10pt}\frac{1}{6}\hspace{10pt}$}}};
\path (0.75, 0.433013) node[rotate=-60]{\maxsizebox{3.41961cm}{!}{\resizebox{!}{18pt}{$\hspace{10pt}\frac{1}{12}\hspace{10pt}$}}};
\path (0.75, 1.29904) node[rotate=60]{\maxsizebox{3.41961cm}{!}{\resizebox{!}{18pt}{$\hspace{10pt}\frac{2}{27}\hspace{10pt}$}}};
\path (1., 0.866025) node[rotate=0]{\maxsizebox{3.41961cm}{!}{\resizebox{!}{18pt}{$\hspace{10pt}\frac{2}{27}\hspace{10pt}$}}};
\path (1.25, 0.433013) node[rotate=60]{\maxsizebox{3.41961cm}{!}{\resizebox{!}{18pt}{$\hspace{10pt}\frac{2}{27}\hspace{10pt}$}}};
\path (1.25, 1.29904) node[rotate=-60]{\maxsizebox{3.41961cm}{!}{\resizebox{!}{18pt}{$\hspace{10pt}\frac{29}{648}{+}\frac{1}{18 \rtthree \pi }\hspace{10pt}$}}};
\path (1.25, 2.16506) node[rotate=60]{\maxsizebox{3.41961cm}{!}{\resizebox{!}{18pt}{$\hspace{10pt}\frac{65}{216}{-}\frac{251}{54 \rtthree \pi }{+}\frac{91}{9 \pi^2}{-}\frac{26}{\rtthree \pi^3}{+}\frac{6}{\pi^4}\hspace{10pt}$}}};
\path (1.5, 0.) node[rotate=0]{\maxsizebox{3.41961cm}{!}{\resizebox{!}{18pt}{$\hspace{10pt}\frac{2}{27}\hspace{10pt}$}}};
\path (1.5, 1.73205) node[rotate=0]{\maxsizebox{3.41961cm}{!}{\resizebox{!}{18pt}{$\hspace{10pt}{-}\frac{101}{648}{+}\frac{23}{9 \rtthree \pi }{-}\frac{19}{6 \pi^2}{+}\frac{\rtthree}{\pi^3}\hspace{10pt}$}}};
\path (1.75, 0.433013) node[rotate=-60]{\maxsizebox{3.41961cm}{!}{\resizebox{!}{18pt}{$\hspace{10pt}\frac{29}{648}{+}\frac{1}{18 \rtthree \pi }\hspace{10pt}$}}};
\path (1.75, 1.29904) node[rotate=60]{\maxsizebox{3.41961cm}{!}{\resizebox{!}{18pt}{$\hspace{10pt}\frac{251}{1944}{-}\frac{101}{108 \rtthree \pi }{+}\frac{17}{18 \pi^2}\hspace{10pt}$}}};
\path (1.75, 2.16506) node[rotate=-60]{\maxsizebox{3.41961cm}{!}{\resizebox{!}{18pt}{$\hspace{10pt}{-}\frac{8113}{11664}{+}\frac{3403}{324 \rtthree \pi }{-}\frac{473}{27 \pi^2}{+}\frac{133}{3 \rtthree \pi^3}{-}\frac{37}{\pi^4}{+}\frac{24 \rtthree}{\pi^5}\hspace{10pt}$}}};
\path (2., 0.866025) node[rotate=0]{\maxsizebox{3.41961cm}{!}{\resizebox{!}{18pt}{$\hspace{10pt}\frac{251}{1944}{-}\frac{101}{108 \rtthree \pi }{+}\frac{17}{18 \pi^2}\hspace{10pt}$}}};
\path (2.25, 0.433013) node[rotate=60]{\maxsizebox{3.41961cm}{!}{\resizebox{!}{18pt}{$\hspace{10pt}{-}\frac{101}{648}{+}\frac{23}{9 \rtthree \pi }{-}\frac{19}{6 \pi^2}{+}\frac{\rtthree}{\pi^3}\hspace{10pt}$}}};
\path (2.25, 1.29904) node[rotate=-60]{\maxsizebox{3.41961cm}{!}{\resizebox{!}{18pt}{$\hspace{10pt}\frac{353}{11664}{+}\frac{58}{27 \rtthree \pi }{-}\frac{277}{36 \pi^2}{+}\frac{64}{3 \rtthree \pi^3}\hspace{10pt}$}}};
\path (2.25, 2.16506) node[rotate=60]{\maxsizebox{3.41961cm}{!}{\resizebox{!}{18pt}{$\hspace{10pt}\frac{4225}{2187}{-}\frac{72965}{1944 \rtthree \pi }{+}\frac{2618}{27 \pi^2}{-}\frac{3319}{9 \rtthree \pi^3}{+}\frac{688}{3 \pi^4}{-}\frac{54 \rtthree}{\pi^5}\hspace{10pt}$}}};
\path (2., 2.59808) node[rotate=0]{\maxsizebox{3.41961cm}{!}{\resizebox{!}{18pt}{$\hspace{20.pt}\longrat{{-}\frac{9275}{2187}{+}\frac{20023}{243 \rtthree \pi }{-}\frac{6050}{27 \pi^2}}{{+}\frac{18175}{18 \rtthree \pi^3}{-}\frac{816}{\pi^4}{+}\frac{228 \rtthree}{\pi^5}{+}\frac{216 \rtthree}{\pi^7}}\hspace{20.pt}$}}};
\path (2.5, 0.) node[rotate=0]{\maxsizebox{3.41961cm}{!}{\resizebox{!}{18pt}{$\hspace{10pt}\frac{65}{216}{-}\frac{251}{54 \rtthree \pi }{+}\frac{91}{9 \pi^2}{-}\frac{26}{\rtthree \pi^3}{+}\frac{6}{\pi^4}\hspace{10pt}$}}};
\path (2.5, 1.73205) node[rotate=0]{\maxsizebox{3.41961cm}{!}{\resizebox{!}{18pt}{$\hspace{10pt}\frac{4225}{2187}{-}\frac{72965}{1944 \rtthree \pi }{+}\frac{2618}{27 \pi^2}{-}\frac{3319}{9 \rtthree \pi^3}{+}\frac{688}{3 \pi^4}{-}\frac{54 \rtthree}{\pi^5}\hspace{10pt}$}}};
\path (2.75, 0.433013) node[rotate=-60]{\maxsizebox{3.41961cm}{!}{\resizebox{!}{18pt}{$\hspace{10pt}{-}\frac{8113}{11664}{+}\frac{3403}{324 \rtthree \pi }{-}\frac{473}{27 \pi^2}{+}\frac{133}{3 \rtthree \pi^3}{-}\frac{37}{\pi^4}{+}\frac{24 \rtthree}{\pi^5}\hspace{10pt}$}}};
\path (2.75, 1.29904) node[rotate=60]{\maxsizebox{3.41961cm}{!}{\resizebox{!}{18pt}{$\hspace{10pt}\frac{4225}{2187}{-}\frac{72965}{1944 \rtthree \pi }{+}\frac{2618}{27 \pi^2}{-}\frac{3319}{9 \rtthree \pi^3}{+}\frac{688}{3 \pi^4}{-}\frac{54 \rtthree}{\pi^5}\hspace{10pt}$}}};
\path (2.75, 2.16506) node[rotate=-60]{\maxsizebox{3.41961cm}{!}{\resizebox{!}{18pt}{$\hspace{20.pt}\longrat{\frac{534377}{209952}{-}\frac{8263}{729 \rtthree \pi }{-}\frac{106931}{648 \pi^2}{+}\frac{20441}{9 \rtthree \pi^3}}{{-}\frac{24913}{6 \pi^4}{+}\frac{3692 \rtthree}{\pi^5}{-}\frac{4422}{\pi^6}{+}\frac{576 \rtthree}{\pi^7}}\hspace{20.pt}$}}};
\path (3., 0.866025) node[rotate=0]{\maxsizebox{3.41961cm}{!}{\resizebox{!}{18pt}{$\hspace{10pt}\frac{4225}{2187}{-}\frac{72965}{1944 \rtthree \pi }{+}\frac{2618}{27 \pi^2}{-}\frac{3319}{9 \rtthree \pi^3}{+}\frac{688}{3 \pi^4}{-}\frac{54 \rtthree}{\pi^5}\hspace{10pt}$}}};
\path (3.25, 0.433013) node[rotate=60]{\maxsizebox{3.41961cm}{!}{\resizebox{!}{18pt}{$\hspace{20.pt}\longrat{{-}\frac{9275}{2187}{+}\frac{20023}{243 \rtthree \pi }{-}\frac{6050}{27 \pi^2}}{{+}\frac{18175}{18 \rtthree \pi^3}{-}\frac{816}{\pi^4}{+}\frac{228 \rtthree}{\pi^5}{+}\frac{216 \rtthree}{\pi^7}}\hspace{20.pt}$}}};
\path (3.25, 1.29904) node[rotate=-60]{\maxsizebox{3.41961cm}{!}{\resizebox{!}{18pt}{$\hspace{20.pt}\longrat{\frac{534377}{209952}{-}\frac{8263}{729 \rtthree \pi }{-}\frac{106931}{648 \pi^2}{+}\frac{20441}{9 \rtthree \pi^3}}{{-}\frac{24913}{6 \pi^4}{+}\frac{3692 \rtthree}{\pi^5}{-}\frac{4422}{\pi^6}{+}\frac{576 \rtthree}{\pi^7}}\hspace{20.pt}$}}};
\path (3.25, 2.16506) node[rotate=60]{\maxsizebox{3.41961cm}{!}{\resizebox{!}{18pt}{$\hspace{20.pt}\longrat{\frac{3933721}{629856}{-}\frac{335699}{972 \rtthree \pi }{+}\frac{5790053}{2916 \pi^2}{-}\frac{2655691}{162 \rtthree \pi^3}{+}\frac{446323}{18 \pi^4}}{{-}\frac{198106}{3 \rtthree \pi^5}{+}\frac{37072}{\pi^6}{-}\frac{13464 \rtthree}{\pi^7}{+}\frac{3456}{\pi^8}{+}\frac{2592 \rtthree}{\pi^9}}\hspace{20.pt}$}}};
\path (3., 2.59808) node[rotate=0]{\maxsizebox{3.41961cm}{!}{\resizebox{!}{18pt}{$\hspace{20.pt}\longrat{{-}\frac{48931685}{629856}{+}\frac{1430516}{729 \rtthree \pi }{-}\frac{6715249}{972 \pi^2}{+}\frac{734939}{18 \rtthree \pi^3}{-}\frac{103607}{2 \pi^4}}{{+}\frac{45032 \rtthree}{\pi^5}{-}\frac{68088}{\pi^6}{+}\frac{792 \rtthree}{\pi^7}{+}\frac{36288}{\pi^8}{+}\frac{5184 \rtthree}{\pi^9}}\hspace{20.pt}$}}};
\node (0,0) {\contour{white}{origin}};
\end{scope}
\end{tikzpicture}
\end{center}
\vspace*{-12pt}
\caption{Undirected edge intensity of loop-erased random walk on the
  triangular lattice.  The edge-intensities of $(x,x)(x,x-1)$ and
  $(x,x-1)(x+1,x-1)$ are identical for $x=1,2,3$, and perhaps all $x$,
  despite there being no lattice symmetry that would imply
  this.}
\label{tri-intensity}
\end{figure}

\subsection{Loop-erased random walk on \texorpdfstring{$\Z\times\R$}{Z x R}}
\enlargethispage{12pt}

\old{
\begin{figure}[htpb]
\begin{center}
\includegraphics[scale=0.5]{ZR-tree}
\end{center}
\caption{Portion of uniform spanning tree and LERW from $(0,0)$ to $\infty$
 on $\Z\times(\Z/20)$, which approximates $\Z\times\R$.  The probability that
the LERW from $(0,0)$ passes through $(1,0)$ is $1/4-1/\pi^2$.}
\label{ZR-tree}
\end{figure}
}

We consider next a weighted version of $\Z^2$, where each horizontal
edge has weight $c$, and each vertical edge has weight $1/c$.  This
graph is isoradial, so we may compute the Green's function using
\cite{Kenyon.isoradial}.  Because the lattice is symmetric under a
$180^\circ$ rotation and invariant under translations, we can also
compute $G'$.  This gives us all the necessary information we need to
compute the probability that LERW from $(0,0)$ passes through $(1,0)$.
It is convenient to let $c=\tan\theta$.
After a computation similar to the ones above, we find that
the LERW passes through vertex $(1,0)$ with probability
\[
\frac{1}{4} + \frac{\theta}{2\pi} - \frac{\theta^2}{\pi^2 \sin^2\theta}\left(1-\frac{2\theta}{\pi}\right).
\]
When $\theta=\pi/4$, we have $c=1$, and this above probability reduces
to $5/16$, in agreement with our earlier calculation for $\Z^2$.

\enlargethispage{24pt}
In the isoradial embedding of the lattice into the plane, if the
horizontal edges have length~$1$, then the vertical edges have length
$c$.  An interesting special case is the limit $c\to 0$.  Then random
walk on this weighted graph converges to a standard Brownian motion in
the vertical direction, except at a Poisson set of times with
intensity $1$, where the walk jumps left or right with equal
probability.  The random walk on this graph is then a continuous-time
random walk on $\Z$ in the horizontal direction and a Brownian motion
on $\R$ in the vertical direction.  From the above formula, we see that in
this limit, the probability that the LERW passes through $(1,0)$
converges to $1/4-1/\pi^2$.

\appendix

\section{The annular matrix}

Recall the matrix that we introduced in
section~\sref{annular1-complete-pairing} for computing grove partition
functions for pairings in which $n-1$ nodes are on one boundary of an
annulus and the last node is on the other boundary.  The rows are
indexed by subsets of $\{1,\dots,n\}$ of size $n/2$ and the columns
are indexed by annular pairings.  Since $n$ is even and positive,
we let $k=n/2-1$.
In this appendix we derive the key
properties of these matrices that we use.  We review the
``cycle lemma'' in \aref{pairing-set}, which gives a canonical
association between the rows and columns of the matrix and simplifies
the subsequent analysis.  In \aref{inverse} we show that the matrix is
nonsingular and give a combinatorial description of the inverse.  In
\aref{no-higher-derivative} we derive a formula about the inverse
annular matrix which shows that the higher order derivatives of
$\L_{i,j}$ do not appear in the formulas for the grove partition
functions.  The determinant of the annular matrix is surprisingly
simple, it is a power of $1-\zeta$, so that while this fact is not
specifically used in the computation of grove partition functions, we
give a derivation in \aref{determinant}.

\pagebreak
\subsection{Annular pairings, subsets, and the cycle lemma}
\label{pairing-set}

Recall that a standard Dyck path of order~$k$ has $2k+1$ points,
numbered $0,1,\ldots,2k$, and $2k$ steps, numbered $1,\ldots,2k$,
where each step is either $+1$ or $-1$, and the partial sums are
non-negative.  Dyck paths are among the structures enumerated by the
$k$th Catalan number $\frac{(2k)!}{k!(k+1)!}$, and are in bijective
correspondence with non-crossing pairings of $\{1,\dots,2k\}$ (see
e.g., \cite[exercise~6.19(r,n)]{stanley}).

We define a ``cyclic Dyck path'' of order $k$ to have $2k+1$ points and $2k+1$ steps, with the points $0$ and $2k+1$ identified, and the steps numbered $1,\ldots,2k+1$, where one of the $2k+1$ steps is $0$ (the ``flat step'') and the other steps are $\pm1$ and define a standard Dyck path of order $k$ when read in cyclic order starting after the flat step.  Cyclic Dyck paths are in bijective correspondence with annular perfect matchings that have $2k+1$ nodes on one boundary (corresponding to the $2k+1$ steps) and one node ($2k+2$) on the other boundary.  The annular perfect matching pairs node $2k+2$ with the flat step of the cyclic Dyck path, and every $+1$ step is paired with its associated $-1$ step in the usual way for standard Dyck paths.

Given a cyclic Dyck path of order $k$, the set of down steps is a subset of
$\{1,\dots,2k+1\}$ of size $k$.
The ``cycle lemma'' bijection of Dvoretzky and Motzkin \cite{DM} (see also
\cite{MR1034142})
states that for each subset of $\{1,\dots,2k+1\}$ of size $k$, there is a
unique cyclic Dyck path giving rise to it in this way, i.e., the set of down
steps uniquely determines which of the remaining steps is the flat step.
In our setting, $k=n/2-1$,
in $\det\L_R^S$, the indices in $R$ give the locations of the up steps and
the flat step, and
the indices in $S$ give the locations of the down steps together with $2k+2$.
The flat step (in $R$) gets paired up with $2k+2$ (in $S$),
and the chords under
the Dyck path determine the rest of the pairing.  The following
example illustrates the bijection adapted to $(n-1,1)$-annular
pairings:
\[
 {}_{\{3,4,6,7,8,11\}}^{\{1,2,5,9,10,12\}}
 \!\Rightarrow\! \raisebox{-12pt}{\includegraphics[scale=0.43]{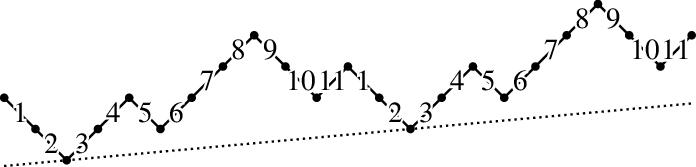}}
 \!\Rightarrow\! \raisebox{-12pt}{\includegraphics[scale=0.43]{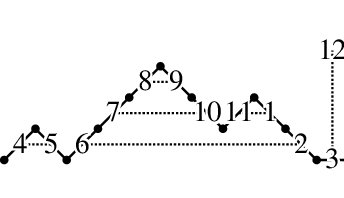}}
 \!\Rightarrow\!
 {}_4^5|{}_6^2|{}_{\;7}^{10}|{}_8^9|{}_{11}^{\,1}|{}_{\;3}^{12}
\]
In the reverse direction, we have
\[
 1,11|2,6|3,12|4,5|7,10|8,9
 \Rightarrow \raisebox{-2pt}{\includegraphics[scale=0.55]{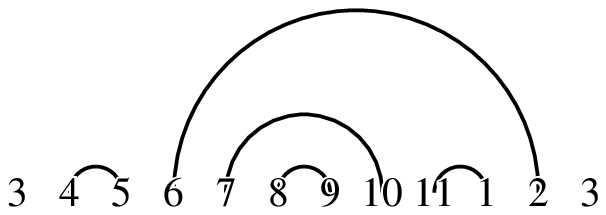}}
 \Rightarrow
 {}_4^5|{}_6^2|{}_{\;7}^{10}|{}_8^9|{}_{11}^{\,1}|{}_{\;3}^{12}
 \Rightarrow
 {}_{4,6,\;7,\;8,11,\;3}^{5,2,10,9,\,1,\;12}.
\]
Notice that we obtain the same $2\times (k+1)$ array of numbers, where the columns represent the annular pairing, and the rows represent the sets $R$ and $S$.

Suppose $S\subset\{1,\dots,2k+2\}$ is a set for which $2k+2\in S$ and $|S|=k+1$.
Suppose $\tau$ is a cyclic Dyck path of order $k$ (which we may interpret as an annular pairing).
We let $S\cdot\tau$ denote the number of up-steps of $\tau$ in $S$.
We say that an up-step of $\tau$ is wrapped if its corresponding down step has a smaller index.  We let $S:\tau$ denote the number of wrapped up-steps of $\tau$ in~$S$.
Define
\begin{equation} \label{eq:A}
  A_{S,\tau} =
 \begin{cases} (-1)^{S\cdot\tau} \zeta^{S:\tau} &
               \parbox{2.1in}{if $S\setminus\{2k+2\}$ is obtained by taking one endpoint from each chord in $\tau$,}\\[9pt]
               0 & \text{otherwise}.
\end{cases}
\end{equation}
Then $A_{S,\tau}$ is the entry of the annular matrix $\A_{2k+2}$ corresponding to
row $\det\L_R^S$ and column $\ddddot{\Zv}[\tau]$ (where $R=\{1,\dots,2k+2\}\setminus S$ and the indices of $R$ and $S$ ordered as described above).
When $A_{S,\tau}$ is nonzero, it can be rewritten as
\begin{equation} \label{eq:A-alternate}
A_{S,\tau} = (-1)^{S\cdot\tau}
  \zeta^{\text{\# indices in $S\setminus\{2k+2\}$ after $\tau$'s flat step}
   \,-\, \text{\# down steps of $\tau$ after $\tau$'s flat step}}.
\end{equation}

\subsection{Inverse annular matrix}\label{inverse}

Next we show that the inverse of the annular matrix can be expressed
in terms of objects known as ``cover-inclusive Dyck tilings''.
\begin{window}[11,r,\raisebox{-5pt}{\includegraphics{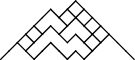}}\raisebox{-1in}{\rule{0pt}{1.4in}},{}]
\noindent
Dyck tilings were independently introduced by Kenyon and Wilson \cite{KW3} and
Shigechi and Zinn-Justin \cite{MR2927185}, and were studied
further in \cite{kim,KMPW,fisher-nadeau}.  For any pair of Dyck paths
$\lambda$ and $\mu$ of order $k$, if the path $\mu$ dominates
$\lambda$ in the sense that at each position $\mu$ is higher than
$\lambda$, then the region between $\lambda$ and $\mu$, denoted by
$\lambda/\mu$, is a (rotated) skew Young diagram, which can be
tiled by $\sqrt{2}\times\sqrt{2}$ squares rotated by $45^\circ$.  A
Dyck tile is obtained from a Dyck path by replacing each vertex of the
Dyck path with such a rotated $\sqrt{2}\times\sqrt{2}$ square, and
then gluing the squares together.
A cover-inclusive (c.i.) Dyck tiling of
$\lambda/\mu$ is a tiling of $\lambda/\mu$ by Dyck tiles such that the
Dyck paths associated to any two Dyck tiles either cover disjoint
portions of the horizontal axis, or the region covered by one tile is
a subset of the region covered by the other tile, with the larger tile
underneath the smaller tile.  The diagram at right shows an example
c.i.\ Dyck tiling of the region between two Dyck paths.
This definition extends naturally to cyclic Dyck paths $\lambda$ and $\mu$
with flat steps at the same location.
\end{window}

\begin{theorem}\label{thm:A-inverse}
Let $\lambda$ be a cyclic Dyck path of order~$k$, and let $S\subset\{1,\dots,2k+2\}$ have size $k+1$ and contain $2k+2$.  Define
\begin{multline} \label{eq:A-inverse}
  B_{\lambda,S}
  = \sum_\mu \,\,[\text{\rm\# of c.i.\ Dyck tilings of $\lambda/\mu$}] \,\times \\[-5pt]
   \zeta^{\text{\rm\# indices in $S$ at which $\mu$ has an up-step}
} \,\times \\
  \zeta^{-\text{\rm\# indices in $S\setminus\{2k+2\}$ after $\lambda$'s flat step}
   \,+\, \text{\rm\# down steps of $\lambda$ after $\lambda$'s flat step}}.
\end{multline}
Then the matrix $(B_{\lambda,S})/(1-\zeta)^k$ is the inverse of the annular matrix $\A_{2k+2}$.
Furthermore, $B_{\lambda,S}$ is a polynomial in $\zeta$ (i.e., no negative powers) of degree at most~$k$.
\end{theorem}

To prove this theorem we start with a lemma:
\begin{lemma}  \label{lem1}
Let $\tau$ and $\mu$ be cyclic Dyck paths of order $k$ on a $(2k+1)$-cycle,
with the flat step of $\mu$ at position $2k+1$.
\[
 \sum_{\parbox{1.03in}{\rm\scriptsize subsets $S$ obtained by taking one endpoint from each chord of $\tau$}}
  (-1)^{S\cdot\tau} \zeta^{S:\tau} \zeta^{S\cdot\mu}
 = \begin{cases} (-1)^{|\mu/\tau|} (1-\zeta)^k & \parbox{\widthof{\rm down some chords of $\tau$}}{\rm\setlength{\baselineskip}{2pt} if it is possible to push down some chords of $\tau$ to obtain $\mu$,} \\[14pt] 0 & \text{\rm otherwise}.\end{cases}\]
(If the flat steps of $\tau$ and $\mu$ are different, then it is not
possible to push down chords of $\tau$ to obtain $\mu$, so the second case applies.)
\end{lemma}
\begin{proof}
  Suppose that $\tau$'s flat step is also located at $2k+1$.  Then $S:\tau=0$ for any~$S$.
  There are two subcases:
  \begin{enumerate}
  \item Suppose each chord $\{i,j\}$ of $\tau$ connects an up-step of
    $\mu$ to a down-step of~$\mu$.  Then it is possible to ``push
    down'' some of the chords of $\tau$ to obtain $\mu$.  Let $S_0$ be
    the set of the down steps of $\mu$, so $S_0\cdot\mu=0$.  For any
    other set $S$ obtained from taking one endpoint from each chord
    of~$\tau$, $S\cdot\mu$ is precisely the number of chords of $\tau$
    on which $S$ and $S_0$ disagree.  Thus the sum is
    $(-1)^{S_0\cdot\tau} (1-\zeta)^k$.
    Now $S_0\cdot\tau$ is the number chords of $\tau$ that we push
    down to obtain~$\mu$.  Each time a chord of $\tau$ is pushed down,
    the area between the modified Dyck path and $\mu$ changes by an
    odd amount, so the parity of the area of $\mu/\tau$ is the parity of $S_0\cdot\tau$.
    \item
  Otherwise, there is some chord $\{i,j\}$ of $\tau$ for which both
  $i$ and $j$ are up steps in $\mu$.  For each set $S$ in the sum, let
  $S'$ be the symmetric difference of $S$ with $\{i,j\}$.  Then
  $\zeta^{S\cdot\mu}=\zeta^{S'\cdot\mu}$ and $\zeta^{S:\tau}=\zeta^{S':\tau}=1$, but
  $(-1)^{S\cdot\tau}=-(-1)^{S'\cdot\tau}$, so the sum is $0$ in this case.
  \end{enumerate}

  Next suppose that the flat step of $\tau$ differs from the flat step of $\mu$.  There are several subcases:
  \begin{enumerate}
  \item Suppose there is an unwrapped chord of $\tau$ at two up-steps or two down-steps of $\mu$.  Then by the argument of the previous paragraph the sum is~$0$.
  \item Suppose $\tau$ has a wrapped chord $(2k+1,j)$, and $\mu$ has an up-step at~$j$.
  Let $S'$ be the symmetric difference of $S$ with $\{2k+1,j\}$.
  Then the terms corresponding to $S$ and $S'$ add up to $0$.
  \item Suppose $\tau$ has an unwrapped chord $(i,2k+1)$, and $\mu$ has a down-step at $i$.
  Let $S'$ be the symmetric difference of $S$ with $\{i,2k+1\}$.
  Then the terms corresponding to $S$ and $S'$ add up to $0$.
  \item Suppose $\tau$ has a wrapped chord $(2k+1,j)$, and $\mu$ has a down-step at~$j$.
  Then the subpath of $\tau$ consisting of steps $1,\dots,j-1$ is a Dyck path, whereas the subpath of $\mu$ on the same interval has more up-steps than down-steps.  This implies subcase 1 occurs within this interval.
  \item Suppose $\tau$ has an unwrapped chord $(i,2k+1)$, and $\mu$ has an up-step at~$i$.
  Then the subpath of $\tau$ consisting of steps $i+1,\dots,2k$ is a Dyck path, whereas the subpath of $\mu$ on the same interval has more down-steps than up-steps, which again implies subcase 1 occurs.  \qedhere
  \end{enumerate}
\end{proof}

\begin{proof}[Proof of Theorem~\ref{thm:A-inverse}]
In an earlier article we proved \cite[Theorem~1.5]{KW3} that if $\lambda$ and $\tau$ are standard Dyck paths of order~$k$, then
\begin{multline} \label{eq:dyck-inverse}
   \sum_{\mu \text{ above $\lambda$}} (-1)^{|\lambda/\mu|} [\text{\# of c.i.\ Dyck tilings of $\lambda/\mu$}] \,\times \\ \begin{cases}1 & \parbox{\widthof{if it is possible to push down}}{if it is possible to push down some chords of $\tau$ to obtain $\mu$,} \\[5pt] 0 & \text{otherwise} \end{cases}
 \ = \ \begin{cases} 1 & \text{if $\lambda=\tau$,}\\[3pt]0&\text{otherwise.}\end{cases}
\end{multline}
When $\lambda$ and $\tau$ are cyclic Dyck paths of order $k$, this formula is still true: If $\lambda$ and $\tau$ have their flat steps at the same place, then it is a straightforward consequence of the formula for standard Dyck paths, and if $\lambda$ and $\tau$ have their flat steps in different places, the formula trivially holds since there are no $\mu$'s between $\lambda$ and~$\tau$.

In each nonzero summand of \eqref{eq:dyck-inverse}, $\lambda$ is
dominated by $\mu$ which is dominated by~$\tau$.  Since $|\lambda/\mu|
+ |\mu/\tau| = |\lambda/\tau|$, we can multiply both sides of
\eqref{eq:dyck-inverse} by $(-1)^{|\lambda/\tau|}$ to effectively replace
$(-1)^{|\lambda/\mu|}$ with $(-1)^{|\mu/\tau|}$.

If $\lambda$ has its flat step at position $2k+1$, then so does $\mu$,
so we can multiply both sides of \eqref{eq:dyck-inverse} by
$(1-\zeta)^k$ and then use Lemma~\ref{lem1} to replace
$(-1)^{|\mu/\tau|}\times(1-\zeta)^k\times$ the conditional expression
on the left-hand side with a summation over $S$:

\begin{multline*}
   \sum_{\mu \text{ above $\lambda$}} [\text{\# of c.i.\ Dyck tilings of $\lambda/\mu$}] \times
   \\
    \sum_{\parbox{\widthof{\scriptsize $S$: each chord of $\tau$}}{\scriptsize $S$: each chord of $\tau$ intersects $S$ once}}
     (-1)^{S\cdot\tau} \zeta^{S:\tau} \zeta^{S\cdot\mu}
   = \begin{cases} (1-\zeta)^k & \text{if $\lambda=\tau$,} \\ 0 & \text{otherwise}.
\end{cases}
\end{multline*}

Changing the order of summation, we obtain
\begin{multline} \label{eq:almost-proved}
  \sum_{\parbox{\widthof{\scriptsize $S$: each chord of $\tau$}}{\scriptsize $S$: each chord of $\tau$ intersects $S$ once}} (-1)^{S\cdot\tau} \zeta^{S:\tau}
  \sum_{\mu \text{ above $\lambda$}} [\text{\# of c.i.\ Dyck tilings of $\lambda/\mu$}]  \zeta^{S\cdot\mu}
  \\ = \begin{cases} (1-\zeta)^k & \text{if $\lambda=\tau$,} \\ 0 & \text{otherwise}.
\end{cases}
\end{multline}
Next we use formula~\eqref{eq:A} for $A_{S,\tau}$ and
definition \eqref{eq:A-inverse} for $B_{\lambda,S}$ (using that $\lambda$'s
flat step is at position $2k+1$)
to rewrite the summand of \eqref{eq:almost-proved}
as $A_{S,\tau} B_{\lambda,S}$.
Since $A_{S,\tau}$ is zero unless each chord of $\tau$ intersects $S$ once,
we can extend the summation to include all $S$, and obtain
\begin{equation}\label{eq:BA=1}
  \sum_S  B_{\lambda,S} A_{S,\tau} = \begin{cases} (1-\zeta)^k &   \text{if $\lambda = \tau$,} \\[9pt]
0 & \text{otherwise}.
\end{cases}
\end{equation}
for cyclic Dyck paths $\lambda$ and $\tau$ of order $k$ when $\lambda$ has its flat step at location $2k+1$.

Next we argue that \eqref{eq:BA=1} also holds when $\lambda$'s flat step is in other locations.
For a given $\lambda$, $S$, and $\tau$, suppose that we cycically decrease all the indices by the same amount modulo $2k+1$ (except $2k+2$, which indexes the node on the other boundary), to obtain $\lambda'$, $S'$, and $\tau'$.
When $B_{\lambda,S} A_{S,\tau}$ is nonzero, we see from \eqref{eq:A-inverse} and \eqref{eq:A-alternate} that
$B_{\lambda',S'} A_{S',\tau'}$ differs from it by a power of $\zeta$.
The indices in $S'$ after $\tau'$'s flat step correspond to the indices in $S$ after $\tau$'s flat step
together with the indices in $S'$ after $\lambda'$'s flat step, unless $\tau'$'s flat step gets wrapped around,
at which point the number of indices of $S'$ after $\tau'$ flat step drops by $k$.
The downsteps of $\tau'$ after $\tau'$'s flat step similarly correspond to the downsteps of $\tau$ after $\tau$'s flat
step together with the downsteps of $\tau'$ after $\lambda'$'s flat step, until $\tau'$'s flat step gets wrapped around,
at which point there is a similar jump by $k$.
Thus we have
\[B_{\lambda',S'} A_{S',\tau'} = B_{\lambda,S} A_{S,\tau} \times
\frac{\zeta^{\text{\# down steps of $\lambda'$ after $\lambda'$'s flat step}}}
{\zeta^{\text{\# down steps of $\tau'$ after $\lambda'$'s flat step}}}.
\]
Upon summing over $S'$, we see that we obtain \eqref{eq:BA=1} scaled by a power of $\zeta$,
which is still zero when $\lambda\neq\tau$, and the power of $\zeta$ is $1$ when $\lambda=\tau$.
Thus the identity \eqref{eq:BA=1} holds
for general cyclic Dyck paths $\lambda$ and $\tau$.

Next we check that $B_{\lambda,S}$ is a polynomial in $\zeta$.  Referring to \eqref{eq:A-inverse}, $\mu$ and $\lambda$ have their flat step in the same place.
Consider the indices in $S$ after $\lambda$'s flat step.  Each such index contributes a factor $\zeta^{-1}$ in \eqref{eq:A-inverse}, but also a factor of $\zeta$ if $\mu$ has an up step at that index.  But because $\mu$ dominates $\lambda$, for each such down step of $\mu$ after the flat step, there is also a down step of $\lambda$ after the flat step, which also contributes a factor of $\zeta$.  Thus $B_{\lambda,S}$ has no negative powers of $\zeta$.

Next we bound the degree of $B_{\lambda,S}$.
For each down step of $\lambda$ after $\lambda$'s flat step, there must also be an up step of $\lambda$, and hence also of $\mu$.  Each such up step of $\mu$ makes no net contribution to the power of $\zeta$, so the degree of $B_{\lambda,S}$ is at most $k$.
\end{proof}

\subsection{Cancellation of higher order derivatives}
\label{no-higher-derivative}
We start with an identity:

\begin{theorem}
Suppose $\sigma$ and $\tau$ are cyclic Dyck paths of order $k$, and $U\subset\{1,\dots,2k+2\}$, and $|U|\leq\ell<k$.
Then
  \begin{equation} \label{dalpha}
  \sum_S
   \left[\left(\frac{\zeta\,d}{d\zeta}\right)^{\ell-|U|}
   B_{\sigma,S}\right]_{\zeta=1}
   \times A_{S,\tau}\Big|_{\zeta=1}
   \times (-1)^{|S\cap U|} =0.
  \end{equation}
\end{theorem}

\begin{proof}
  For a general finite graph with $n=2k+2$ nodes and
  with general parallel transports, define $\ddddot\Zv[\sigma]$
  as in the case of a planar graph or annular-one graph.
  Let $R=\{1,\dots,n\}\setminus S$, and define
  \[D_R^S = \sum_\sigma A_{S,\sigma} \ddddot\Zv[\sigma].\]
  For annular graphs
  $D_R^S = \det\L_R^S$, but for general graphs these
  two quantities will be different.  We
  can ``recover'' the $\ddddot\Zv[\sigma]$'s from these
  $D_R^S$'s by multiplying by $\A_n^{-1}$:
  \begin{equation} \label{lHopital}
   \frac{\Zv[\sigma]}{\Zv[1|2|\cdots|n]} = \frac{1}{(1-\zeta)^k} \sum_{S} B_{\sigma,S}(\zeta) D_R^S ,
  \end{equation}
  where $k=n/2-1$ and $B_{\sigma,S}$ was defined in \eqref{eq:A-inverse}.

  Let us consider now the complete graph on $n$ nodes, so that
  $\L=-\Delta$ and $\L_{i,j}$ is the edge weight between nodes $i$ and
  $j$, times the parallel transport to $i$ from $j$.  Note that in
  this case $\L$ is (except for the diagonal entries) a general
  Hermitian matrix.  Each $D_R^S$ is a polynomial in
  the entries of $\L$, with coefficients that involve powers of
  $\zeta$.  (For the complete graph, any matrix times the vector of
  $\ddddot\Zv[\sigma]$'s will yield polynomials in the
  $\L_{i,j}$'s for $i\neq j$.)

  We change variables by setting $\zeta=e^t$.
  Then $\frac{d}{dt}
  \L_{j,i} = - \frac{d}{dt} \L_{i,j}$.  The $\zeta\to1$ limit is of
  course equivalent to $t\to 0$, and $(1-e^t)^k$ has a zero of
  order $k$ at $t=0$, with $\frac{d^k}{dt^k} (1-e^t)^k = (-1)^k k!$.

  For general nonzero edge weights and smooth (in $t$) parallel
  transports on the complete graph, $\Zv[\sigma]$ is finite and
  $\Zv[1|2|\cdots|n]=1$, so for any $\ell<k$ it must be that we get zero
  when we differentiate the numerator from \eqref{lHopital}, i.e.,
  $\sum_S B_{\sigma,S}(e^t) D_R^S$, $\ell$ times
  with respect to $t$ and then set $t$ to $0$:
  \[
  0 = \sum_{m=0}^\ell \binom{\ell}{m} \sum_{S} \frac{d^{\ell-m}}{dt^{\ell-m}}B_{\sigma,S}(e^t)\Big|_{t=0} \times \frac{d^{m}}{dt^{m}} D_R^S \Big|_{t=0}.
  \]
  Since $\L$ is generic, we can rescale the $m$th derivative of each $\L_{i,j}$ by a factor of $\beta^m$, and deduce that for each $m,\ell$ with $m\le \ell<k$
  \begin{equation} \label{dalphaddet}
  0 = \sum_{S} \frac{d^{\ell-m}}{dt^{\ell-m}}B_{\sigma,S}(e^t)\Big|_{t=0} \times \frac{d^{m}}{dt^{m}} D_R^S\Big|_{t=0}.
  \end{equation}

  Next we write $D_R^S()$ for the polynomial function
  of a Hermitian matrix, which, when evaluated on the response matrix
  $\L$ of the complete graph, gives $D_R^S=D_R^S(\L)$ (see \eqref{DRSLT}).  Let
  $d_i$ denote the differential operator for which $d_i
  \L_{i,j}=\frac{1}{2}\frac{d}{dt}\L_{i,j}$ and
  $d_i\L_{j,i}=\frac{1}{2}\frac{d}{dt}\L_{j,i}$ but $d_i \L_{h,j}=0$
  for $h,j\neq i$.  For the complete graph, each monomial of the
  polynomial $D_R^S$ includes each index $i$ exactly
  once.  (This also holds for annular graphs, though of course the
  polynomials are different.)  Using this property of these polynomials,
  we can write the $m$th derivative of $\L=\L(t)$ as follows:
  \begin{equation} \label{dkdetL}
   \frac{d^m}{dt^m} D_R^S(\L) = \sum_{i_1,i_2,\dots\in\{1,\dots,n\}} d_{i_1} \cdots d_{i_m} \; D_R^S(\L) .
  \end{equation}

  Given a set $U$ of $m$ nodes, for each $i\in U$ and each $j$ we
  rescale $\frac{d}{dt} \L_{i,j}\Big|_{t=0}$ by a factor of $\beta$,
  without changing any of the other derivatives at $t=0$.  If $i,j\in
  U$, then we rescale $\frac{d^2}{dt^2} \L_{i,j}\Big|_{t=0}$ by a
  factor of $\beta^2$.  (Recall that $\L$ is generic Hermitian, so we
  can do this.)  The coefficient of $\beta^m$ within the $m$th
  derivative is obtained from \eqref{dkdetL} by including only those terms for which
  $i_1,\dots,i_m$ is a permutation of~$U$.
  Then substituting \eqref{dkdetL} into \eqref{dalphaddet} and taking the coefficient of $\beta^m$,
  we find
  \begin{equation} \label{dalphadet}
  0 = \sum_{S} \frac{d^{\ell-m}}{dt^{\ell-m}}B_{\sigma,S}(e^t)\Big|_{t=0} \times D_R^S(\L^{(U)})\Big|_{t=0},
  \end{equation}
  where $\L^{(U)}$ where is the Hermitian matrix obtained from $\L$ by
  replacing $\L_{i,j}$ with $\frac{d}{dt} \L_{i,j}$ for each $i\in U$
  and $j\notin U$ or $i\notin U$ and $j\in U$, and replacing
  $\L_{i,j}$ with $\frac{d^2}{dt^2}\L_{i,j}$ for each $i,j\in U$.

  From our definition of $D_R^S$, for the complete graph we have
  \begin{equation}\label{DRSLT}
   D_R^S(\L^{(U)}) = \sum_\rho A_{S,\rho}
   \prod_{\substack{\{i,j\}\in\rho\\i\notin S,j\in S}} \L^{(U)}_{i,j},
  \end{equation}
  where the sum is over annular directed pairings $\rho$.
  Next we take the pairing $\tau$, and for each pair
  $\{i,j\}$ of $\tau$, we rescale $\L^{(U)}_{i,j}$ by a factor
  $\gamma$.  Then the coefficient of $\gamma^{n/2}$ in \eqref{dalphadet} only arises when
  $\rho=\tau$ in the above sum, so
  \begin{equation} \label{dalphadet2}
  0 = \sum_{S} \frac{d^{\ell-m}}{dt^{\ell-m}}B_{\sigma,S}(e^t)\Big|_{t=0} \times A_{S,\tau}\Big|_{t=0} \times \prod_{\substack{\{i,j\}\in\tau\\i\notin S,j\in S}} \L^{(U)}_{i,j}\Big|_{t=0},
  \end{equation}
  For each $S$, the product term in the formula takes the same (generically nonzero) value, except for a sign,
  which is given by the parity of $S\cap U$.
  So we cancel this factor (keeping the sign),
  and obtain \eqref{dalpha}.
\end{proof}

We now restate and prove \tref{L''}:
\begin{theorem}
  Suppose that an annular-one graph has $n$ nodes, and that $\sigma$
  is a partial pairing of the $\{1,\dots,n\}$ which has $k+1$ pairs, one of which
  contains $n$.  Then
  $Z_\sigma/Z_{1|2|\cdots|n}$ is a polynomial of
  degree $k+1$ in the quantities
  \[
  \{ L_{i,j}  : 1\leq i < j \leq n \}\quad\text{and}\quad
  \{ L'_{i,j} : 1\leq i < j \leq n-1 \}.
  \]
\end{theorem}

\begin{proof}
  Let $Q$ denote the singleton nodes of $\sigma$, and $T$ the unlisted
  / internal nodes.  Let $h(i)$ denote the $i$th element of
  $\{1,\dots,n\}\setminus(Q\cup T)$.  For $S\subset\{1,\dots,2k+2\}$ with
  $2k+2\in S$ and $|S|=k+1$, and $R=\{1,\dots,2k+2\}\setminus S$, the
  relevant determinants are of the form
  $D_R^S=\det\L_{h(R),T}^{h(S),T}$, and \eqref{lHopital} holds for these $D_R^S$'s.
  Thus
  \[
  \frac{Z[\sigma]}{Z[1|\cdots|n]} = \frac{(-1)^k}{k!}
  \sum_{m=0}^k \binom{k}{m} \sum_{S} \frac{d^{k-m}}{dt^{k-m}}B_{\sigma,S}(e^t)\Big|_{t=0} \times \frac{d^{m}}{dt^{m}} \det\L_{h(R),T}^{h(S),T}\Big|_{t=0}.
  \]
  We rewrite the $m$th derivative of $\det\L_{h(R),T}^{h(S),T}$ using the
  differential operators~$d_i$, as in \eqref{dkdetL}.  Let $U$ be the
  set of nodes for which we applied $d_i$ an odd number of times.  If
  for some $i$ we applied $d_i$ more than once, then the size of the
  set $U$ will be less than $m$, and then by \eqref{dalpha} (with $\tau=\sigma$),
  the coefficient of such terms is $0$.

  Next, suppose the variable $\L_{i,n}$ is differentiated.  Since $n$
  is in each set $h(S)$, we may as well replace $U$ with
  $U\setminus\{n\}$ and introduce a global sign, but then since $|U|$
  is smaller, we see from \eqref{dalpha} that the coefficient of such
  terms is $0$.
\end{proof}

\subsection{Determinant}\label{determinant}

\begin{theorem}
  The determinant of the annular matrix is
  \[\det \A_n = (1-\zeta)^{2^{n-2} - \frac12 \binom{n}{n/2}}.\]
\end{theorem}
\begin{proof}
Since $\det\A_n$ is a polynomial in $\zeta$, and the formula for
$\A_n^{-1}$ is well defined whenever $\zeta\neq 1$, it follows that
$\det\A_n$ can only have a root at $\zeta=1$.

\enlargethispage{12pt}
We split the original
zipper into $n-1$ zippers each with parallel transport $z^{1/(n-1)}$,
and then deform these zippers so that their endpoints lie in each
of the $n-1$ intervals between the nodes.  When we deform a zipper
across node $i(\neq n)$ in the counterclockwise direction, the parallel
transport from~$i$ to any other node~$j$ is multiplied by
$z^{1/(n-1)}$.  For each column of the annular matrix, say indexed by
directed pairing~$\sigma$, the column is scaled by $z^{\mp 1/(n-1)}$
according to whether node~$i$ is a source or destination in $\sigma$.
Likewise, each row of the annular matrix, say indexed by $\det\L_R^S$,
is scaled by $z^{\pm 1/(n-1)}$ according to whether $i\in R$ or $i\in S$.
The effect of deforming these zippers is to conjugate the annular matrix $\A_n$
by a diagonal matrix, yielding a new more symmetric matrix $\A^*_n$ for which
$\det\A^*_n = \det\A_n$.

We change variables to
\[ w = z^{2/(n-1)} = \zeta^{1/(n-1)}\]
so that the nonzero entries of $\A^*_n$ are integral powers of $w$.
For example, the first two rows of $\A^*_6$ are
\[\begin{gathered}\phantom{\rf{[{}_1^2|{}_3^4|{}_5^6]}}\\
\bordermatrix[{[]}]{
& \srf{\ddddot{\Zv}[{}_1^2|{}_3^4|{}_5^6]}& \srf{\ddddot{\Zv}[{}_2^3|{}_1^4|{}_5^6]}& \srf{\ddddot{\Zv}[{}_5^1|{}_2^3|{}_4^6]}& \srf{\ddddot{\Zv}[{}_1^2|{}_5^3|{}_4^6]}& \srf{\ddddot{\Zv}[{}_4^5|{}_1^2|{}_3^6]}& \srf{\ddddot{\Zv}[{}_5^1|{}_4^2|{}_3^6]}& \srf{\ddddot{\Zv}[{}_3^4|{}_5^1|{}_2^6]}& \srf{\ddddot{\Zv}[{}_4^5|{}_3^1|{}_2^6]}& \srf{\ddddot{\Zv}[{}_2^3|{}_4^5|{}_1^6]}& \srf{\ddddot{\Zv}[{}_3^4|{}_2^5|{}_1^6]}\\
\det\L_{\;1, 3, 5}^{2, 4, 6}& 1 & -w & 0 & 0 & -w & 0 & 0 & 0 & w^2 & -w^3 \\[2pt]
\det\L_{\;2, 1, 5}^{3, 4, 6}& 0 & 1 & 0 & 0 & 0 & 0 & 0 & w^4 & -w & 0 \\[2pt]
}
\end{gathered}\]
and the other rows are determined by cyclic rotations.

From the determinant formula for $\det\L_R^S$, we see that the diagonal entries of
$\A^*_n$ are all $1$.  Consider the column indexed by directed pairing
$\sigma$.  Since $\A^*_n$ is symmetric under cyclic rotations of the
indices $1,\ldots,n-1$, let us assume for convenience that $\sigma$
pairs $n-1$ to $n$, so that we can
write \[\sigma={}_{a_{1,0}}^{a_{1,1}}|\cdots|{}_{a_{n/2-1,0}}^{a_{n/2-1,1}}|{}_{n-1}^{\;n}.\]
Referring to the above bijection, since $n-1$ pairs to $n$, for each
$j$ we have $a_{j,0}<a_{j,1}$.  Column $\sigma$ contains $2^{n/2-1}$
nonzero entries, one for each sequence $f_1,\ldots,f_{n/2-1}$ of
$n/2-1$ $0$'s and $1$'s, where the row is indexed
by \[\det\L_{\displaystyle a_{1,f_1},\ldots,a_{n/2-1,f_{n/2-1}},n-1}^{\displaystyle a_{1,1-f_1},\ldots,a_{n/2-1,1-f_{n/2-1}},n}.\]
Since the pair $(a_{j,0},a_{j,1})$ crosses $a_{j,1}-a_{j,0}$ zippers,
this pair contributes $w^{f_j}$ to the matrix entry.  In particular,
all nonzero nondiagonal entries of $\A^*_n$ have positive powers of $w$.
This implies \[ \left.\det\A_n\right|_{\zeta=0} = \left.\det\A^*_n\right|_{w=0} = 1. \]

For a given column, the row that maximizes the power of $w$ is the one for
which $f_0,\ldots,f_{n/2-1}=1,\ldots,1$, and the power is the area
under the Dyck path.  The mapping from a column $\sigma$ to the row
$\det\L_R^S$ which has the highest power of $w$ is also a bijection,
in fact it is a simple variant of the cycle lemma bijection.
Hence the leading coefficient of the polynomial $\det\A^*_n$
is $\pm 1$, and the degree is
\[ \deg\det\A^*_n = (n-1)\times \sum_{\text{Dyck paths of length $n-2$}} \text{area under Dyck path}.\]
For Dyck paths of length $2 k = n-2$, the above sum is
(see Sloane's \href{http://oeis.org/A008549}{A008549})
\[4^k-\binom{2k+1}{k} = 2^{n-2} - \binom{n-1}{n/2-1} = 2^{n-2} - \frac12 \binom{n}{n/2}.\]

Because $\det\A_n$ has a root only at $\zeta=1$, the constant term is $1$,
and the degree is $2^{n-2} - \frac12 \binom{n}{n/2}$, the determinant formula
follows.
\end{proof}

\bibliographystyle{amsalpha}
\bibliography{54}
\end{document}